\documentclass[11pt, a4paper]{article} 

\usepackage[final]{graphicx}
\usepackage{xcolor}
\usepackage[english]{babel} 
\usepackage[utf8]{inputenc} 
\usepackage{vmargin} 
\setmargins{2.5cm}       
{1.5cm}                        
{16.5cm}                      
{23.42cm}                    
{10pt}                           
{1cm}                           
{0pt}                             
{2cm}                           

\usepackage{amsmath}
\usepackage{xcolor}

\usepackage{amssymb}
\usepackage{vmargin}
\usepackage{braket}
\usepackage{enumerate}
\usepackage{empheq} 
\usepackage{bigstrut}
\usepackage{float} 
\usepackage{mathrsfs}
\usepackage{mdframed} 
\usepackage{xcolor} 
\usepackage{wrapfig}

\usepackage{version}

\newcommand{\eps}{\varepsilon} 
\newcommand{\myeq}[1]{\ensuremath{\stackrel{\text{#1}}{=}}}
\newcommand{\wto}{\rightharpoonup}
\newcommand{\wsto}{\overset{\ast}{\rightharpoonup}}
\newcommand{\myleq}[1]{\ensuremath{\stackrel{\text{#1}}{\leq}}}

\newcommand{\defeq}{\vcentcolon=}

\usepackage[nottoc]{tocbibind}
\usepackage{amsthm}
\usepackage{wrapfig}
\usepackage{caption}
\usepackage{esint}
\usepackage{hyperref}
\usepackage{titlesec}
\usepackage{commath}
\usepackage{accents}
\usepackage{braket}

\newcommand\te{\tau_{\varepsilon}}
\newcommand\restr[2]{\ensuremath{\left.#1\right|_{#2}}}

\numberwithin{equation}{section}
\newtheorem{theorem}{Theorem}[section]
\newtheorem{prop}[theorem]{Proposition}
\newtheorem{lemma}[theorem]{Lemma}
\newtheorem{corollary}[theorem]{Corollary}
\newtheorem{remark}[theorem]{Remark}
\newtheorem{definition}[theorem]{Definition}

\usepackage{enumitem} 
\setlist[enumerate,1]  {label={\rm (\roman*)}, leftmargin=1.5em} 

\usepackage{bm}

\renewcommand{\abs}[1]{\left| #1 \right|}

\newcommand*\Lap{\mathop{}\!\mathbin\bigtriangleup}

\titleformat{\paragraph}
{\normalfont\normalsize\bfseries}{\theparagraph}{1em}{}
\titlespacing*{\paragraph}
{0pt}{3.25ex plus 1ex minus .2ex}{1.5ex plus .2ex}

\newcommand*\adh[1]{\overline{#1}} 

\newcommand{\RN}{{\mathbb{R}^N}}

\newcommand{\R}{\mathbb{R}}

\newcommand{\rnorm}[1]{{\left\vert\kern-0.25ex\left\vert\kern-0.25ex\left\vert #1 
		\right\vert\kern-0.25ex\right\vert\kern-0.25ex\right\vert}}
	
\def\Xint#1{\mathchoice
	{\XXint\displaystyle\textstyle{#1}}%
	{\XXint\textstyle\scriptstyle{#1}}%
	{\XXint\scriptstyle\scriptscriptstyle{#1}}%
	{\XXint\scriptscriptstyle\scriptscriptstyle{#1}}%
	\!\int}
\def\XXint#1#2#3{{\setbox0=\hbox{$#1{#2#3}{\int}$ }
		\vcenter{\hbox{$#2#3$ }}\kern-.580\wd0}}

\def\dashint{\Xint-}
\renewcommand{\fint}{\dashint}   

\newcommand{\bnorm}[1]{\mathrel{\left\|#1\right\|}}
\newcommand{\hil}{H}
\renewcommand{\norm}[1]{\mathrel{\|#1\|}}
\DeclarePairedDelimiter{\normdos}{\lVert}{\rVert}

\mathtoolsset{showonlyrefs} 

\begin{document}
	\title{A brush problem. Homogenization involving thin domains and PDEs in graphs.}

	\author{ José M. Arrieta${}^{(1)}$ \\
		Joaquín Domínguez-de-Tena${}^{(2)}$}

	\date{\today}
	\maketitle
	
	\setcounter{footnote}{2}
	\begin{center}
		Departamento de Análisis Matemático y  Matem\'atica Aplicada\\ Universidad
		Complutense de Madrid\\ 28040 Madrid, Spain \\ and \\
		Instituto de Ciencias Matem\'aticas \\
		CSIC-UAM-UC3M-UCM, Spain 
	\end{center}
	
	\makeatletter
	\begin{center}
		${}^{(1)}$ {E-mail: jarrieta@ucm.es
			}
		\\
		${}^{(2)}$ {E-mail:
			joadomin@ucm.es}
	\end{center}
	\makeatother

	\noindent {$\phantom{ai}$ {\bf Key words and phrases:}  Homogenization, thin domains, PDEs in graphs, brush-type domain.} 
	\newline{$\phantom{ai}$ {\bf Mathematical Subject Classification
			2020:} \
		35B27, 35J25, 35R02}
	

\begin{abstract}
	This work analyses the homogenization of a linear elliptic equation with Neumann boundary conditions in a comb/brush domain, composed of a fixed base and a family of thin teeth. The teeth are defined as rescalings of order less than or equal to $\varepsilon$ of a model tooth of arbitrary shape. Periodicity in their distribution is not assumed; instead, the existence of an asymptotic limit density $\theta$, which may vanish in certain regions, is assumed. The convergence analysis is performed using an adaptation of the unfolding operator method to a non-periodic framework. Finally, it is shown that, under certain conditions on the geometry of the teeth, the resulting limit problem can be interpreted as a differential equation on a graph.
\end{abstract}

\section{Introduction}

For several decades now, the field of homogenization in partial differential equations has gained significant interest and relevance. Its origins lie in the modelling of real world processes where the studied system exhibits repetitive structures at a comparatively smaller scale than the rest of the system. These microstructures may be represented by rapidly varying coefficients appearing in the equation, or by very rapidly variations of the domain, as in the case of perforated domain, or by highly oscillating boundaries,  among others.  Starting with the pioneering works \cite{SanP, BLP, ZKO}  where the basic methods to analyze these problems where set, the literature in the field is now very rich with new methods and applications to many different situations. Without being exhaustive, we refer to \cite{All, Ng, CioDon, librocioranescu,MKE,OSY,CasDiaz,CioMur} and references therein for a good view of different analysis and applications of homogenization theory.

The idea behind homogenization theory is to simplify the equations under study, obtaining a new ``averaged'' or ``homogenized'' equation that incorporates the effects of the microstructures in the equation. By doing so, the resultant equation does not have the rapid variations of the original one (in terms of variations of coefficients or domain or boundary oscillations) although these variations have been encoded into the equation through the new homogenized coefficients.  The homogenized equation may even contain new terms due to the microstructures, see for instance \cite{RauchTaylor} or the very well known article \cite{CioMur}.

On the other hand, the analysis of the behavior of solutions of PDEs in thin domains is nowadays a very active area of study with a lot of results from many different perspectives.  For instance, in the pioneer work \cite{HR} the authors analysed the behavior of the asymptotic states (attractor of the system) of a nonlinear reaction diffusion equation posed in a thin domain. They showed the convergence to an appropriate nonlinear 1D reaction diffusion equation, where the coefficients of the limiting problem incorporate the geometry of the thin domain. See also \cite{prizzi2001effect} for thin domains with ``holes'' whose asymptotic limit problem can be reinterpreted as an equation in a graph,  \cite{R} for a survey of PDEs on thin domains, \cite{AS} for a refinement of some results of \cite{HR}.  Another relevant type of domains studied in the literature are the so-called dumbbell domain, which roughly speaking consists of two fix domains joined by a thin channel whose width is given by a small parameter. This domain became very relevant when analyzing the phenomena of pattern formation in reaction-diffusion equations. Starting from the pioneer works \cite{Mat79, CasHol78},  it has been also a very active area also from different perspectives since the past century. See for instance \cite{J, Arr2, JK, dumbell, ACLc2, ACLc3}

Also, thin domains with oscillations at the boundary have been considered, obtaining the homogenized limit equation for which a combination of techniques from thin domains and homogenization is usually needed. See for instance \cite{brushthin3, VP1} for an analysis of the linear elliptic problem with locally periodic oscillatory boundaries.

\par\medskip

In this work, we deal with a brush (or a comb) domain which consists, roughly speaking, on a fix base domain at which a family of bristles (or teeth) are attached. The number of teeth goes to infinity and the width of each one goes to zero as the parameter $\varepsilon\to 0$. Let us remark that we usually refer to comb if the structure is 2D and brush if it is 3D or higher dimension, but we will use the term ``brush'' to refer to both situations.  Each bristle may be regarded as a thin domain. And attaching thin domains to a fix base domain may be considered as a dumbbell-type domain perturbation.  That is why, the understanding and analysis of thin and dumbbell domains will be helpful and inspiring when dealing with these brush domains.    

\begin{figure}[htbp]
	\centering
	\includegraphics[trim=0 70 0 70, clip, width=\textwidth]{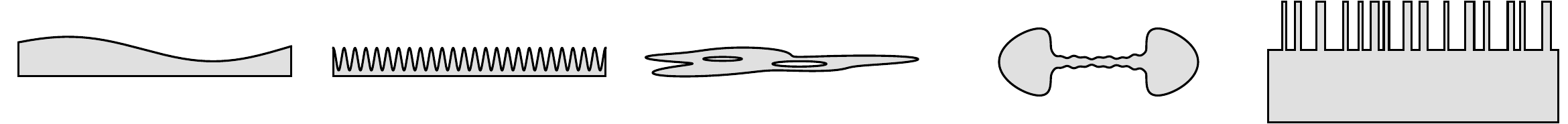}
	\caption{\small Examples of domains with thin structures. From left to right: three examples of thin domains, including one of the form $0 < y < \varepsilon g(x)$ \cite{HR}, an oscillating thin domain $0 < y < \varepsilon g(x/\varepsilon^\alpha)$ \cite{VP1}, and a thin domain of arbitrary shape with holes $\Omega_\varepsilon = \{(x,y) : (x, y/\varepsilon) \in \Omega\}$ \cite{prizzi2001effect}. The last two are domains with thin structures: a dumbbell domain, where two bodies are joined by a thin channel \cite{dumbell}, and a comb domain, where a body has several teeth attached to its top \cite{gaudiello}.}
	\label{fig:thin_domains}
\end{figure}
The starting point of this work is the outstanding and inspiring article \cite{gaudiello}, where the authors study a Neumann elliptic problem in brush-type domain $\Omega_\varepsilon\subset \R^N$.  
 Each tooth is a pure cylinder of fixed height and of section $\omega_\eps\subset \R^{N-1}$ with diameter less than $\eps$. Remarkably, instead of assuming a periodicity condition on the distribution of the teeth, as is customary in this kind of problems, the authors only assume a convergence of the density of teeth in certain sense to a limit density $\theta$ which is assumed to be strictly positive in the region where the teeth are located. In addition, they treat the case of source terms in $L^1(\Omega_\varepsilon)$,
 for which a more involved analysis and the consideration of renormalized solutions is required.

In the work that we start with this paper and which is continued in \cite{jtxtheta0}, a geometric generalization of the above problem is treated. We make two main generalizations. The first one consists in considering more general  shapes for the teeth instead of just purely cylindrical ones. We will consider teeth with shapes similar to the thin domains treated in the works \cite{HR, AS} or even \cite{prizzi2001effect}, in which thin domains with non connected cross sections are considered,   see Figure \ref{fig:supercombinado}.  Having teeth with these more general geometries  makes the limiting homogenized equation different, where the geometry of the thin domain is encoded into the equation. Moreover, in some situations where the thin domains do not have connected cross sections, the limiting homogenized problem needs to be reinterpreted as an equation in certain graph which is related to the geometry of the thin domain,  see Section \ref{sec:graph}.

The second generalization is the consideration of the case in which the asymptotic limit density of teeth, given by the function $\theta$,  may degenerate and become null in part of the region where the teeth concentrate. Notice that in  \cite{gaudiello} it is assumed that $\theta$ is bounded away from zero in the whole region where the teeth are located. 
The function $\theta$ appears in the limit equation as a ``parameter'' in the transmission conditions imposed at the boundary between the base of the brush and the asymptotic domain where the teeth are located. It actually measures the strength of the transmission conditions. Having this function $\theta$ degenerating at some region represents a degeneracy in the equation which needs to be understood properly and need some appropriate technical treatment.  

The weak formulation of the limit problem is expressed in \eqref{eqn:thmmaineq3} and the convergence of the solutions in the brush domain to the solutions of the limit problem is stated in our main result Theorem \ref{thm:main}. 

Moreover, in the second part of this work, see \cite{jtxtheta0},  we treat the very degenerate case where $\theta\equiv 0$. This includes the case of having just one single tooth (or a fix number of teeth $n$), as it is the case of a standard dumbbell domain \cite{dumbell}. But we may have an increasing number of teeth $n_\eps\to +\infty$ as $\eps\to0$, but in such a way that if the cross section of each teeth has $(N-1)$ dimensional measure $\sim m_\eps$ and  $n_\eps m_\eps\to 0$.  Our main result, Theorem \ref{thm:main},  also applies to the case where $\theta\equiv 0$ although this result does not give any information about the behavior of the solutions in the asymptotic region where the teeth are located.  As a matter of fact, to say something meaningful about the behavior of the solutions in the asymptotic teeth, we need to deal with other techniques drawn from dumbbell domain studies, see \cite{dumbell, ACLc2, ACLc3}. This case is treated in \cite{jtxtheta0}.

\par\medskip

As we mentioned above, our intention is to proceed with some geometrical generalizations. In this respect and in order to be clearer we have avoided dealing with more general elliptic operator and with source terms in $L^1$ (as it is treated in \cite{gaudiello}). We rather stick to the case of 
the canonical elliptic equation $-\Lap u + u = f$ with the source term $f$ in $L^2$.

The techniques used to address this new geometric setting are different from those in \cite{gaudiello}, although they are strongly inspired by the original paper. In the present paper, we make use of the so-called unfolding operator method, initially developed for the treatment of periodic problems \cite{CDG,CDG2,librocioranescu} that later has shown great versatility, being applied in other contexts, such as non-periodic homogenization \cite{tesismanuel,VP1}. Following these works, we adapt the method to our particular situation.
Our work is also inspired by the study of problems in thin domains \cite{R, HR, AS, J, Arr2, JK}, and in particular by  \cite{prizzi2001effect}.

As we have mentioned, the limit problem we obtain has an interpretation in terms of graphs. This is given in Section \ref{sec:graph} for a wide class of shapes of the teeth that, following \cite{prizzi2001effect},  we call ``nicely-decomposed domains'', for which a classical formulation of the limit problem formulated in Theorem \ref{thm:main} makes sense.

The problem presented here arises in various fields of biology, physics, mechanics, and engineering. We refer to \cite{gaudiello} and reference therein for studies related to the present one. In particular, we would refer to \cite{app1, app2, app3, app4, app+}. Other studies related to this type of domain include the foundational work of Brizzi and Chalot \cite{brizzichalot}, as well as subsequent developments such as \cite{brush1,brush2,brush3,brush4,brush5,brush6, brush+}. Additional works involving similar structures and treating control problems are \cite{brushcontrol1,brushcontrol2,brushcontrol3,brushcontrol4, brushcontrol+}.  In the works \cite{app3, brushotherbc1, brush2, brushotherbc+} different boundary conditions are treated and in  \cite{brushthin1,brushthin2,brushthin3,VP1,tesismanuel,app2,brushthin5,brushthin6,brushthin7} the authors study the case in which the height of the teeth degenerates, \cite{brushmorecomplex} for more complex geometries and other references in the paper \cite{gaudiello}.

The organization of the paper is the following. In Section \ref{sec:statmentjtx}, we will describe the problem we will treat and state the main result of the paper, which will be proved later. Section \ref{sec:jtxtools} will be dedicated to developing the tools and functional framework we will need to treat the problem. In Section \ref{sec:jtxresults} we will make the proofs in order to obtain the main result stated in Section \ref{sec:statmentjtx}. Finally, section \ref{sec:graph} will be dedicated to give a better interpretation of the limit problem in terms of a PDE problem in graphs when regularity assumptions are made on the geometry of the teeth.  We include at the end an Appendix with several auxiliary results that will help to understand the 
analysis in this paper. 
\par\medskip

\section{Statement of the problem}
\label{sec:statmentjtx}
In this section, we describe in a rigorous way the problem we are considering.  We include a thorough description of the brush domains and the statement of the main result we will prove. We also include an overview of the method we are using.
\subsection{Description of the problem}
\label{sec:desc}
The main goal of this paper is the study of the convergence of weak solutions of $-\Lap u + u = f$ subject to homogeneous Neumann boundary condition in a family of \textit{brush} domains $\Omega_\varepsilon$ which depend on a parameter $\varepsilon>0$ which we let tend to zero $0$.

Our family of domains will belong to the Euclidean space $\R^{N+1}$, with $N\geq 1$. We adopt this notation because the last dimension, which we will refer to as the \textit{vertical dimension}, will play a distinguished role. The remaining dimensions will be referred to as \textit{horizontal directions}. We will use the notation $x$ or $\xi$ to denote points in the horizontal space $\RN$, and $y$ to denote points in the vertical space $\R$. Thus, a point in $\R^{N+1}$ will be described by $(x, y)$ or $(\xi,y)$, where $x,  \xi, \in \mathbb{R}^N$ and $y \in \mathbb{R}$. To a lesser extent, we will also use $z \in \mathbb{R}^N$ or $s \in \mathbb{R}$.

Moreover, when considering a set $U \subset \mathbb{R}^{N+1}$, we will use the notation $U^a$ to refer to the \textit{upper part}, 
$U^a = \{ (x, y) \in U : y > 0 \}$
and  $U^b$ to refer to the lower part, $U^b = \{ (x, y) \in U : y < 0 \}$. Note that $a$ stands for \textit{above} and $b$ for \textit{below}. This notation is inherited from \cite{gaudiello}.

Let us start describing the domains $\Omega_\varepsilon$. Our domain $\Omega_\varepsilon$ will be divided into two parts: the upper part $\Omega_\varepsilon^a$ and the lower part $\Omega^b$. The lower part will be independent of $\varepsilon$ and will serve as the body of our brush. We will denote it by $\Omega^b$, it will be a bounded Lipschitz domain and it will satisfy
$$
\Omega^b \subset \{ (x, y) \in \mathbb{R}^{N+1} : x \in \mathbb{R}^N,\, y < 0 \}.
$$
We will denote by $R_0$ a positive number such that $\Omega^b\subset B_{\R^{N+1}}(0,R_0)$.

This $\Omega^b$ will contain an open cylinder $\Omega' \times (-1, 0)$, where $\Omega' \subset \mathbb{R}^N$ will be the domain where the teeth of the brush are attached to. 
The upper part $\Omega_\varepsilon^a$ will be given by the collection of  $N_\varepsilon$ disjoint teeth denoted by $Y_\varepsilon^n\subset \{(x,y)\in\R^{N+1}:y>0\}$, where $n$ is running from $1$ to $N_\varepsilon$,
\begin{equation}
	\Omega_\varepsilon^a = \bigsqcup_{n=1}^{N_\varepsilon} Y^n_\varepsilon
\end{equation}
Each tooth $Y^n_\varepsilon$ is obtained by shrinking a fixed base domain $Y\subset \{(\xi,y)\in\R^{N+1}:y>0\}$ in the horizontal directions and placing it at the base point $({\bar x^n_\varepsilon}, 0)$, with $x^n_\eps\in \Omega'$, that is
\begin{equation}
	\label{eqn:defYneps}
	Y^n_\varepsilon=({\bar x^n_\varepsilon},0)+\{ (l^n_\varepsilon \xi,y):  (\xi,y)\in Y\}.
\end{equation}
We assume that there exists $C>0$, independent of $\varepsilon$ and $n$ such that
\begin{equation}
	\label{eqn:ass1}
	0<l_\varepsilon^n \leq C\varepsilon \qquad \forall \varepsilon>0, \ \forall n=1,\ldots N_\eps
\end{equation}

Notice that all of these teeth will have the same shape, given by $Y$, but with different thickness, given by $l_\varepsilon^n>0$. We do not require any periodicity in the distribution of the base points $\bar x^n_\eps$ through $\Omega'$. 

Now, let us describe the basic assumptions on the model tooth $Y$.

\begin{enumerate}
	\item $Y\subset \{(\xi,y): |\xi|<R_1, 0<y<L\}=B(0,R_1)\times (0,L)$ for some $L,R_1>0$.
	\item $\partial Y\cap (\RN\times \{0\})=\bar \omega\times \{0\}$, where $\omega\subset \RN$ is a Lipschitz domain satisfying $0\in \omega$ and
	$
	\abs{\omega}=1
	$,
	which represents the base of the tooth $Y$.
	\item There exists $\delta_0>0$ such that 
	\begin{equation}
		\label{eqn:delta0}
		\omega\times(0,\delta_0)\subset Y
	\end{equation}
\end{enumerate}
For every $n=1,\ldots, N_\varepsilon$, we denote by $\omega_\varepsilon^n$ the base of teeth $Y_\varepsilon^n$, 
$$\omega_\varepsilon^n={\bar x^n_\varepsilon}+l_\varepsilon^n \omega=\left\{x\in \RN: \frac{x-\bar x^n_\varepsilon}{l_\varepsilon^n}\in \omega\right\}$$
and we assume that $\omega_\varepsilon^n\subset \Omega'$. We also define   $$\omega_\varepsilon=\bigcup_{n=1}^{N_\varepsilon} \omega_\varepsilon^n\subset \Omega'$$ and denote its characteristic function in $\RN$ by $\chi_{\omega_\varepsilon}$.

In addition, we also assume that, for different $n,m=1,\ldots,N_\varepsilon$, the sets $B({\bar x^n_\varepsilon},R_1l_\eps^n)$ and $B({\bar x^m_\varepsilon},R_1l_\eps^m)$ are disjoint, so that the teeth do not intersect each other.

With all these assumptions, our brush domain $\Omega_\varepsilon$ is defined as
\begin{equation}
	\label{eqn:omegaeps}
	\Omega_\varepsilon = \Omega^b \cup \Omega_\varepsilon^a \cup (\omega_\varepsilon\times\{0\}).
\end{equation}
The properties required above guarantee that $\Omega_\varepsilon$ is a bounded, connected, open set. 

For a better understanding, we attach an example of a 2-dimensional domain in Figure \ref{fig:supercombinado} with all the notations described above.

\begin{figure}[H]
	\centering 
	\includegraphics[width=1\textwidth]{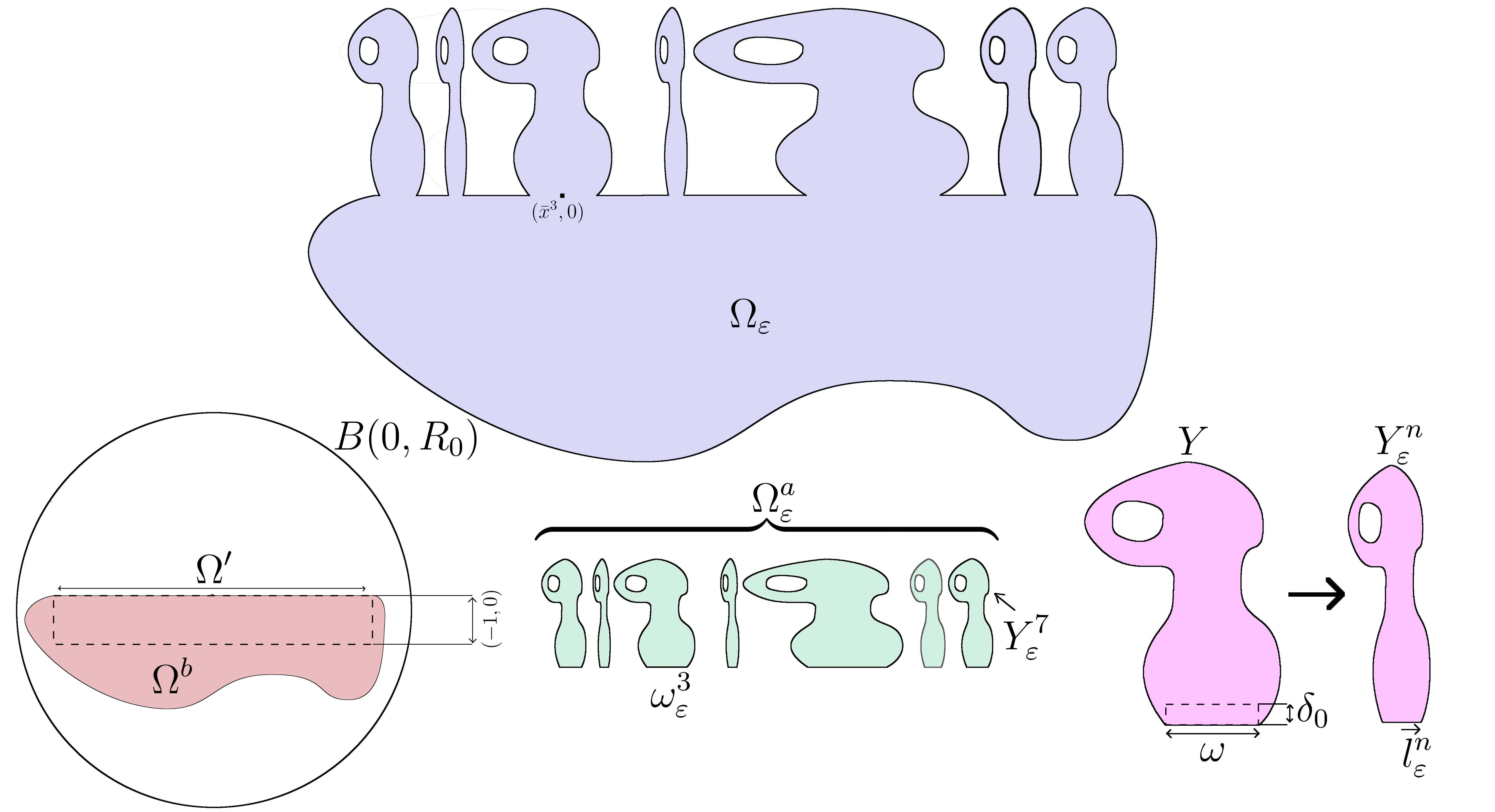}	
	\caption{\small \small An example of a domain $\Omega_\varepsilon$ in 2 dimensions is shown in the figure. In blue, at the top, the domain in question is displayed. Below, their components $\Omega^b$ and $\Omega_\varepsilon^a$ are depicted in red at the bottom left and green at the bottom center, respectively. In pink, at the bottom right, the unit cell $Y$ and an example of a rescaled cell $Y_\varepsilon^n$ are shown. Observe that the set $\Omega^b$ is contained in the ball $B(0, R_0)$ and contains the set $\Omega' \times (-1,0)$. The set $\Omega_\varepsilon^a$ is composed of the cells $Y_\varepsilon^n$ for $n = 1, \ldots, N_\varepsilon$ which have different horizontal size. Note that $Y$ contains the set $\omega \times \delta_0$ and observe that the unit cell $Y$ has an arbitrary shape which may not be simply connected.
}
	\label{fig:supercombinado}
\end{figure}
Now that we have described the domains, we define the problem to be considered. Given $f \in L^2(\R^{N+1})$, we consider the linear elliptic problem:
\begin{equation}
	\label{eqn:mainnonvariational}
	\left\{
	\begin{aligned}
		-\Delta u_\varepsilon+u_\varepsilon&=f,\qquad &&\text{in} \ \Omega_\varepsilon \\
		\frac{\partial u_\varepsilon}{\partial n}&=0, \qquad\qquad &&\text{on} \ \partial\Omega_\varepsilon.
	\end{aligned}
	\right.
\end{equation}
The weak formulation of the problem is given by: find $u_\varepsilon\in H^1(\Omega_\varepsilon)$ such that 
\begin{equation}
	\label{eqn:main}
	\int_{\Omega_\varepsilon}\nabla u_\varepsilon \nabla \varphi + \int_{\Omega_\varepsilon}u_\varepsilon\varphi = \int_{\Omega_\varepsilon} f\varphi,  \qquad \forall \varphi\in H^1(\Omega_\varepsilon).
\end{equation}

The existence and uniqueness of solutions of \eqref{eqn:main} is obtained via Lax-Milgram theorem. The main goal of this paper is to analyse the behaviour of the family of solutions $u_\eps$ as $\eps\to 0$. To do so, we will assume that
$\chi_{\omega_\varepsilon}$ converges $w^*-L^\infty(\Omega')$ to a function $\theta\in L^\infty(\Omega')$ with $0\leq \theta \leq 1$. That is, 
\begin{equation}
	\label{eqn:ass2}
	\chi_{\omega_\varepsilon} \wsto \theta \qquad w^*-L^\infty(\Omega'), \hbox{ when }\eps\to 0.
\end{equation}

\textbf{Remark: }note that, by Banach-Alaoglu's theorem, this assumption will be always satisfied at least for a subsequence of $\varepsilon_n\to 0$.\\

For some reasons we will see later, the set in which $\theta$ vanishes will be important, so we will denote it as \begin{equation}\label{def-theta0}
	\Theta_0=\{x\in \Omega': \theta(x)=0\}.
\end{equation}

\begin{figure}[H]
	\centering
	\includegraphics[width=1\textwidth]{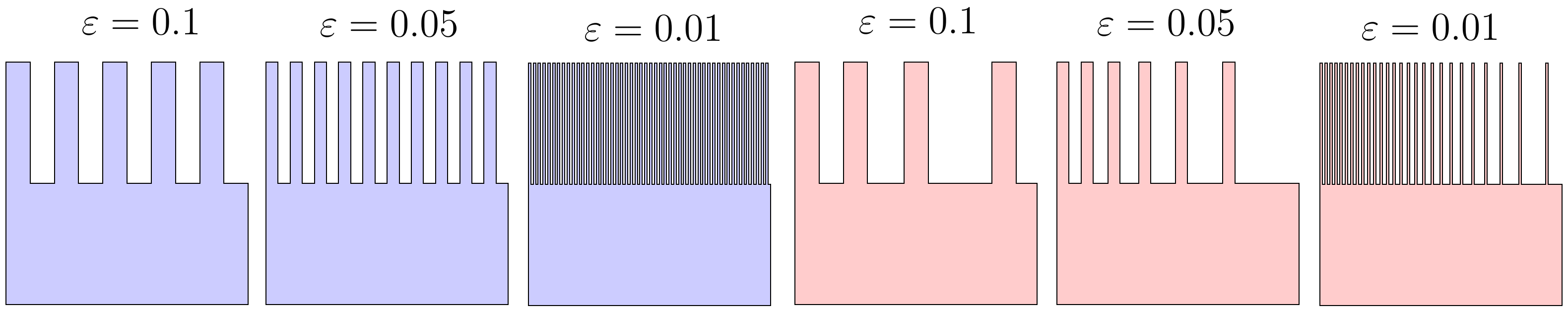}	
	\caption{\small In the figure, two examples of 2-dimensional domains $\Omega_\varepsilon$ are shown for different values of $\varepsilon$. The first three figures at the left, in blue, represent a domain with periodically distributed teeth. In this case, the limit density $\theta$ takes a constant value.
		The other three figures at the right, in red, a domain with non-periodically distributed teeth is shown. In this case, the space between teeth increases linearly (the separation between the n-th and (n+1)-th teeth is n times the separation of the first and second teeth). In this case, the limit density has the form $\theta(x) = (1 - x)/2$.
	}
\end{figure}

\subsection{Main result}
\label{sec:main}
Here we will state the main theorem of the paper, which characterizes the limit problem of \eqref{eqn:main}. As the domains $\Omega_\varepsilon$ of the problem vary with $\varepsilon$, it is not clear what the domain of the limit problem should be. In the lower part, $\Omega^b$, the domain is fixed, so, it is reasonable that the restriction of the limit solution to $\Omega^b$, that we denote by $u^b$ will belong to $H^1(\Omega^b)$. In the upper part, $\Omega^a_\eps= \cup_{n=1}^{N_\varepsilon} Y^n_\varepsilon$, if the number of teeth is very large, we can view $\Omega_\eps^a$ as a rapidly varying boundary. As it is customarily in many problems in homogenization, we will need to ``unfold'' this rapidly varying domain defining the set
$$W=\{(x,\xi,y) : x\in \Omega', (\xi,y)\in Y\}=\Omega'\times Y.$$
See Figure \ref{fig:unfoldeddomain} for a representation of the unfolded domain $W$ given a particular example of 2-dimensional domain $\Omega_\varepsilon$. Note that, although $\Omega_\varepsilon$ depends on $\varepsilon>0$, $W$ does not.

\begin{figure}[H]
	\centering
	\includegraphics[width=1\textwidth]{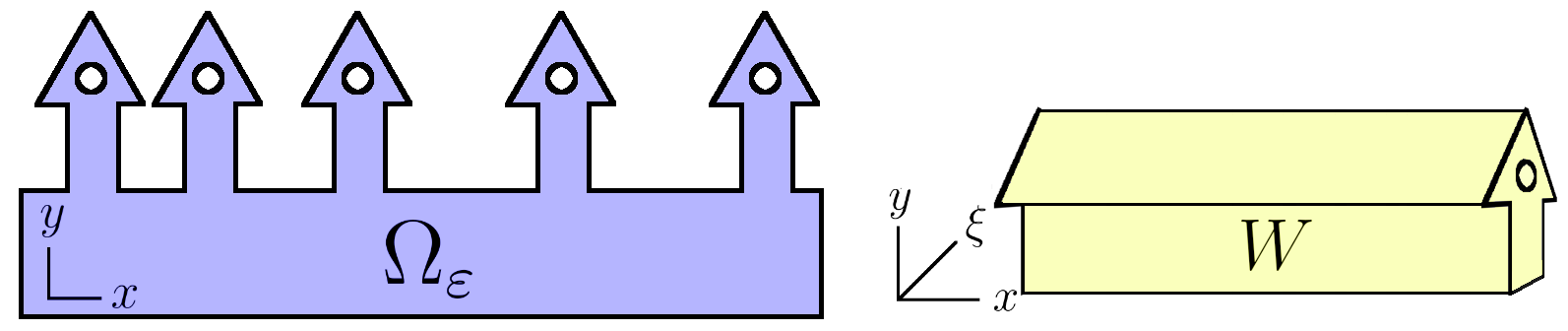}	
	\caption{\small On the left, in blue, a particular domain $\Omega_\varepsilon$ is shown and on the right, in yellow, its corresponding unfolded domain $W$.
	}
	\label{fig:unfoldeddomain}
\end{figure}

The solution of the limit problem in the upper part, that we denote by $u^a$ (the superscript $a$ stands for above), will be defined in the set $W$ and will satisfy $u^a\in L^2(\Omega',\theta, H^1(Y))$, that is $u^a(x)\in H^1(Y)$ a.e. $x\in \Omega'\setminus \Theta_0$ and 
$$\int_{\Omega'} \theta(x)\|u(x,\cdot)\|_{H^1(Y)}^2dx<\infty$$

For this space, the value of $u^a$ is undefined in $\Theta_0$ as the weight $\theta$ vanishes. We refer to Section~\ref{sec:functional} for a discussion of this type of ``weighted Bochner spaces".

\par\medskip

In order to state properly the limit unfolded problem, we need to consider the space of functions 
$$H(\theta)=\{(u^a, u^b): u^a\in L^2(\Omega',\theta; H^1(Y)),  u^b\in H^1(\Omega^b), \hbox{ satisfying (1) and (2) below} \} $$

\begin{enumerate}
	\item[(1)] $\nabla_{\xi}u^a(x)(\xi,y)=0$,  a.e. $x\in \Omega'\setminus \Theta_0$, $(\xi,y)\in Y$, that is $u^a$ does not depend on $\xi$ in the connected components of the horizontal sections of $Y$.
	
	\item[(2)] $u^a(x)(\xi,0)=u^b(x,0)$,   a.e. $x\in \Omega'\setminus  \Theta_0$ and a.e. $\xi\in \omega$ (interpreted in the sense of traces, see Section~\ref{sec:functional} below).
\end{enumerate}
A more detailed analysis of the space of solutions $H(\theta)$ will be given in Section~\ref{sec:spacesols}.

\bigskip

With the notation explained above, we can state the 

\par\medskip \noindent {\bf Limiting ``unfolded'' problem}: given $f\in L^2(\R^{N+1})$, find a pair of functions $(u^a, u^b)\in H(\theta)$, satisfying for every $(\varphi^a,\varphi^b)\in H(\theta)$,
\begin{equation}
	\label{eqn:thmmaineq3}
	\int_{\Omega^b}\left(\nabla u^b \nabla\varphi^b+u^b\varphi^b -f\varphi^b\right)+\int_{\Omega'\times Y}\theta\left( \frac{\partial u^a}{\partial y}\frac{\partial \varphi^a}{\partial y}+u^a \varphi^a-f^*\varphi^a\right)=0,
\end{equation}
where $f^*(x,\xi,y)=f(x,y)$ for $x\in \Omega'$ and $(\xi,y)\in Y$.

\par\bigskip  Existence and uniqueness of solutions of the unfolded problem can be proved via Lax-Milgram theorem.  We will see this later on. The main result is the following,

\begin{theorem}
	\label{thm:main}
	Let $u_\varepsilon$ be the solution of \eqref{eqn:main} and let $(u^a, u^b)\in H(\theta)$ be the solution of the unfolded problem described above. Redefine $u^a(x)=0$ when $x\in\Theta_0$. Then, $u^a$ belongs to $L^2(\Omega';H^1(Y))$ and, defining for $\varepsilon>0$
	\begin{equation}
		\label{eqn:thmmaineq1}
		\bar{u}^a_\varepsilon (x,y)\defeq 
			\fint_{\omega_\varepsilon^n} u^a\left(z,\frac{x-{\bar x^n_\varepsilon}}{l_\varepsilon^n},y\right)dz \qquad  (x,y)\in Y_\varepsilon^n,\quad n=1,\ldots N_\eps
	,
	\end{equation}
	we have $\bar{u}^a_\varepsilon \in H^1(\Omega_\varepsilon)$ and
	\begin{equation}
		\label{eqn:thmmaineq2}
		\lim_{\varepsilon\to 0}\norm{u_\varepsilon-\bar{u}^a_\varepsilon }_{H^1(\Omega^a_\varepsilon)}=0, \qquad	\lim_{\varepsilon\to 0}\norm{u_\varepsilon-u^b}_{H^1(\Omega^b)}=0.
	\end{equation}

\end{theorem}

The proof of this theorem will be given in Section \ref{sec:energy} as a combination of all the results that we will obtain through the following sections.
\par\medskip

\begin{remark}
	The expression \eqref{eqn:thmmaineq1} makes sense for $u^a\in L^2(\Omega';H^1(Y))$ via the identification of $L^2(\Omega';H^1(Y))$
	with the space of functions belonging to $L^2(W)$ such that their derivatives with respect to $y$ and $\xi$ belong to $L^2(W)$ too, done in Lemma \ref{lemma:bochnerproduct}. Through that identification, \eqref{eqn:thmmaineq1} is equivalent to 
		$\bar{u}^a_\varepsilon (x,y)\defeq 
		\left(\fint_{\omega_\varepsilon^n} u^a(z)dz\right)\left(\frac{x-{\bar x^n_\varepsilon}}{l_\varepsilon^n},y\right)$.
\end{remark}

\begin{remark} Note that the limit functions $u^a$, $u^b$ are independent of $\varepsilon$. However, to prove the convergence in $\Omega_\varepsilon$, we need to define $\bar{u}^a_\varepsilon$ which depends irremediably on $\varepsilon$ but is constructed directly from $u^a$ and $u^b$.
\end{remark}

\begin{remark}
	The convergence \eqref{eqn:thmmaineq2} can be improved to $C^{1,\alpha}$ locally in $\adh{\Omega^b}\setminus \{\Omega'\times\{0\}\}$ if $\Omega^b$ and $f$ has better regularity properties. For instance, if $f\in L^\infty(\R^{N+1})$ and $U\subset \subset \Omega^b\cup T$ where $T$ is a $C^{1,1}$ boundary portion of $\Omega'\times\{0\}$, as $-\Lap u_\varepsilon + u_\varepsilon = f$ in $\Omega^b$ and $\frac{\partial u}{\partial n}=0$ on $T$, by standard bootstrap regularity arguments (see e.g. \cite[Corollary B.9]{mythesis}), $\norm{u_\varepsilon}_{C^{1,\alpha}(\adh{U})}\leq C(\norm{f}_\infty+\norm{u_\varepsilon}_{H^1(\Omega^b)})$ for any $0<\alpha<1$. Then, the uniform boundedness of $u_\varepsilon$ in $H^1(\Omega^b)$ implies the uniform boundedness in $C^{1,\alpha}(\adh{U})$, and using Ascoli-Arzelà as well as \eqref{eqn:thmmaineq2}, we obtain $u_\varepsilon \to u^b$ in $C^{1,\alpha}(\adh{U})$.
\end{remark}

\begin{remark}
	Note that the coupling condition (2) has to be checked only in $\Omega'\setminus \Theta_0$, that is, at the points where $\theta>0$. In particular, for the singular but very important case in which $\theta\equiv 0$, the limit problem 
 \eqref{eqn:thmmaineq3} becomes  
\begin{equation}
	\int_{\Omega^b}\left(\nabla u^b \nabla\varphi^b+u^b\varphi^b -f\varphi^b\right)=0, \, \forall \varphi^b\in H^1(\Omega^b)
\end{equation}
which is the weak formulation of the problem $-\Lap u^b+u^b=f$ in $\Omega^b$ with $\frac{\partial u}{\partial n}=0$ on $\partial \Omega^b$. 
Moreover, the only information that Theorem \ref{thm:main} provides on the behaviour of the solutions in $\Omega_\eps^a$ is that $\|u_\eps\|_{H^1(\Omega_\eps^a)}\to 0$, which is somehow expected since the Lebesgue measure of $\Omega_\eps^a$ satisfies $|\Omega_\eps^a|\to 0$. We analyze this case, obtaining a more refined convergence result of the solutions in $\Omega_\eps^a$,  in \cite{jtxtheta0}. See also \cite{mythesis}.  The techniques for this case are different from those of the present work and they are inspired in techniques of thin and dumbbell domains, see \cite{dumbell,AS}
\end{remark}

\subsection{Overview of the Method}
\label{sec:jtxoverview}

The goal of this subsection is to provide  the reader an idea of the underlying general approach, in order to facilitate the understanding of why certain results are obtained or the intuition behind certain definitions. We do not intend to develop here an extensive, rigorous, or detailed description of the method

We will employ the so-called ``unfolding operator method'' (see references~\cite{librocioranescu,tesismanuel}), which is nowadays a well known tool to treat homogenization problems. 

We begin with the problem described in~\eqref{eqn:mainnonvariational} in its variational form,
\begin{equation}
	\int_{\Omega_\varepsilon}\nabla u_\varepsilon \nabla \varphi + \int_{\Omega_\varepsilon}u_\varepsilon\varphi = \int_{\Omega_\varepsilon} f\varphi,  \qquad \forall \varphi\in H^1(\Omega_\varepsilon).
\end{equation}
As previously justified, it is reasonable to split our analysis into the upper and lower parts. When considering the variational formulation, this can be done in a simple way,
\begin{equation}
	\int_{\Omega^b}\left(\nabla u_\varepsilon \nabla \varphi + u_\varepsilon\varphi -f\varphi\right)+\int_{\Omega_\varepsilon^a}\left(\nabla u_\varepsilon \nabla \varphi + u_\varepsilon\varphi -f\varphi\right)= 0\qquad \forall \varphi\in H^1(\Omega_\varepsilon).
\end{equation}
Then, in Section~\ref{sec:unfolding}, we will define the so-called \textit{unfolding operator} $\te$. This operator transforms Lebesgue-measurable functions defined on $\Omega_\varepsilon^a$ into Lebesgue-measurable functions defined on the \textit{unfolded domain} $W$, which is independent of $\varepsilon$. Moreover, $\te$ enjoys very useful properties such as linearity, preservation of integrals, etc. We use the unfolding operator to tackle the upper part of our problem, while the lower part remains unchanged. Without going into much detail, this operator allows us to transform our original problem into the \textit{unfolded problem}:
\begin{equation}
	\begin{aligned}
		\int_{\Omega^b}\left(\nabla u_\varepsilon \nabla \varphi +u_\varepsilon\varphi-f\varphi\right)+\int_{W}\left(\te(\nabla u_\varepsilon)\te(\nabla \varphi)+\te(u_\varepsilon)\te(\varphi) - \te(\varphi) \te(f)\right)
		= 0, \ \ \ \ \ \forall \varphi\in H^1(\Omega_\varepsilon).
	\end{aligned}
\end{equation}
Notice that an important simplification we achieve is that now we have two fixed domains $W$ and $\Omega^b$ while in the original problem we have a fixed domain $\Omega^b$ and a variable domain $\Omega_\varepsilon^a$. The unfolding operator essentially moves the complexity of the problem from the domain to the solutions and the test functions.

The next step is then studying the convergence of the \textit{unfolded problem}. By using some a-priori estimates and compactness results on the solutions, and choosing adequate test functions, it is possible to take the limit $\varepsilon\to 0$ in the \textit{unfolded problem} to obtain the \textit{limit problem} previously described
\begin{equation}
	\int_{\Omega^b}\left(\nabla u^b \nabla\varphi^b+u^b\varphi^b -f\varphi^a\right)+\int_{W}\theta\left( \frac{\partial u^a}{\partial y}\frac{\partial \varphi^a}{\partial y}+u^a \varphi^a-f^*\varphi^a\right)=0, \qquad \forall (\varphi^a,\varphi^b)\in H(\theta),
\end{equation}
where $(u^a,u^b)\in H(\theta)$. The existence and uniqueness of the problem above then follows by Lax–Milgram theorem, and to obtain convergence of the original solutions $u_\varepsilon$ stated in Theorem \ref{thm:main}, we proceed with a reverse argument using the definition of the unfolding operator.

This is, without going into detail, the general scheme of the method we will follow, and which will be seen clearly in the proof of Proposition~\ref{prop:main}. Afterwards, we will obtain results concerning improved convergence, regularity, or interpretations of the limit problem in terms of graphs.

\subsection{Notation}
\label{sec:notations}

This paper involves a high density of notation. Moreover, many of the proofs require lengthy computations and technical details. Therefore, in order to simplify the reading without compromising mathematical rigour, we will systematically adopt the notational conventions that follow.

$\bullet$ The letters $C$ or $c$ may denote different constants, possibly changing from line to line. When the dependence of such constants is not clear from the context, it will be explicitly specified.

$\bullet$ In order to reduce notational burden, we omit the integration variables when no confusion arises. For instance, if the function $g \in L^1(A)$, we may write
$
\int_{A} g =  \int_A g(x)dx.
$
The fully expanded version with all variables will generally be used when the context does not make it clear which variables are being integrated, or when we wish to emphasize the role of specific variables due to their relevance in the proof.

$\bullet$ We sometimes use the same notation for functions and their restrictions. For example, if $g \in L^1(\mathbb{R}^N)$ and $\Omega \subset \mathbb{R}^N$, we may write $g \in L^1(\Omega)$ without explicitly mentioning that we are referring to $\restr{g}{\Omega}$, the restriction of $g$ to $\Omega$.

$\bullet$ When working with functions in product spaces $A\times B$, given a function $g: A \to \mathbb{R}$, we denote by the same symbol its constant extension to $A\times B$,
\begin{equation}
	\begin{aligned}
		g: \ & A\times B && \longrightarrow &&&   \R \  \ \\
		& \ (a,b) && \mapsto &&&  g(a) 
	\end{aligned}
\end{equation}
In this way, if for instance $g \in L^1(A)$ and $\abs{B}<\infty$, then it is meaningful to write
$
	\int_{A\times B} g = \abs{B} \int_A g.
$
However, for main results or whenever there is potential for confusion, we will make a point of clarifying or use alternative notation if necessary, as it is the case with the definition of $ f^*$ in \eqref{eqn:thmmaineq3}, since the extension of $f\in L^2(\R^{N+1})$ to $L^2(W)$ may not be immediately clear.

\section{Tools and functional framework}
\label{sec:jtxtools}
We have divided this section in two subsections. In the first one, we define rigorously the unfolding operator which we have mentioned briefly in the previous section. We will state their main properties and prove some preliminary results we will need later.  In the second subsection we will define the space of solutions for the limit unfolded problem $H(\theta)$ which have been mentioned above in Subection \ref{sec:main}. To accomplish this we need first to introduce some functional spaces, like Weighted Lebesgue Bochner spaces and analyse some of their properties, including traces operators in these spaces and a subclass of spaces that we have denoted ``partially constant Sobolev Spaces". To go over this functional framework is a needed technical first step to properly define the setting of our problem.

\subsection{Unfolding operator}
\label{sec:unfolding}

We define the main tool we are using to prove the convergence. As stated in the introduction, we will denote by $x,\xi$ points in $\RN$ while $y$ represents a point in $\R$ so $(x,y)\in \R^{N+1}$ and $(x,\xi ,y)\in \R^{2N+1}$. When necessary, each of the coordinates of $x$ and $\xi$ will be denoted $x_j$ and $\xi_j$. In the same way, $\nabla_x u=(u_{x_1},\dots, u_{x_N})$ and $\nabla_{\xi} u=(u_{\xi_1},\dots, u_{\xi_N})$.

Inspired by the unfolding operator method, extensively used in the theory of homogenization (see the classical reference book \cite{librocioranescu}, or  \cite{tesismanuel} for an introduction to the unfolding operator in the context of thin domains with oscillatory boundaries), we define a {\sl non-conventional} unfolding operator $\te$ adapted to our problem. This unfolding operator will act only in $\Omega_\varepsilon^a$, the upper part of the brush domain.  
Although it was already mentioned in the introduction, let us now proceed to define the unfolded domain and its base.
\begin{definition}
	We define the unfolded domain $W\subset \R^{2N+1}$ as
	$$
	W=\{(x,\xi,y): x\in \Omega', \ (\xi,y)\in Y \}
	$$
	We will also use the following notation for its base,
	$$
	W_0=\{(x,\xi,0): x\in \Omega', \ \xi\in\omega\}
	$$
\end{definition}
Let us define now the \textit{unfolding operator} $\tau_\varepsilon$.

\begin{definition}
	Let $\varphi$ be a Lebesgue-measurable function in $\Omega_\varepsilon^a$. The unfolding operator $\te$ acting on $\varphi$ is defined as the function $\te(\varphi)$ defined for $(x,\xi,y)\in W$ as follows,
	\begin{equation}
		\label{eqn:defunfolding}
		\te(\varphi)(x,\xi,y)=\left\{
		\begin{array}{cll}
			\varphi({\bar x^n_\varepsilon}+l^n_\varepsilon \xi,y) \qquad  & \text{if}   & x\in \omega^n_\varepsilon \hbox{ for some }n=1,\ldots,N_\eps,\\ \\ 
			0 \qquad & \text{if}   & x\notin \omega_\varepsilon=\cup_{n=1}^{N_\eps}\omega_\eps^n, 		\end{array}
		\right. 
	\end{equation}	
\end{definition}
\begin{remark}
	\label{rem:z}
	\ 
	\begin{enumerate}
		\item Notice that, given $(x,\xi,y)\in W$, with $x\in \omega_\varepsilon$, we have
		\begin{equation}
			\label{eqn:z}
			\te(\varphi)(x,\xi,y)=\varphi(x+z_\eps(x,\xi),y)
		\end{equation} where $z_\eps(x,\xi)=\bar x_\varepsilon^n+l_\varepsilon^n\xi-x$, for $x\in\omega_\eps^n$, which satisfies
		\begin{equation}
			\label{eqn:zeps}
			\abs{z_\eps(x,\xi)}\leq C\varepsilon
		\end{equation}
		where $C>0$ is independent of $\eps$, $x$ and $\xi$.
		\item Note that the operator $\te$ transforms Lebesgue-measurable functions defined on
		$\Omega_\varepsilon^a$ into Lebesgue-measurable functions defined on the set $W$ which are
		piecewise constant with respect to $x$.
		\item Observe also that the operator $\tau_\eps$ assigns the value $0$ to any point $(x,\xi,y)\in W$ with $x\not\in \omega_\eps$. Although this may not be the common procedure to define the unfolding operator, we will see it is very much adapted to our geometrical situation. 
		\item As in classical homogenization, the unfolding operator reflects two scales:
		the \textit{macroscopic} scale $x$ gives the position in $\Omega'$ and the \textit{microscopic}
		scale $(\xi,y)$ gives it in the cell $Y$.
		\item We say that this operator is a \textit{non-conventional} unfolding operator because the unfolding operators typically considered in the literature are applied to periodic systems and do not include the zero-extension property described in (iii). 
	\end{enumerate}
\end{remark}

The following proposition, analogous to \cite[Proposition 1.1.4]{tesismanuel}, presents several basic properties
of the unfolding operator.
\begin{prop}
	\label{prop:prop}
	The unfolding operator $\te$ has the following properties:
	\begin{enumerate}[label=(\roman*)]
		\item $\te$ is a linear operator, i.e., for every $\varphi_1, \varphi_2: \Omega_\varepsilon^a \to \R$ measurable and $\lambda_1,\lambda_2\in \R$,
		\begin{equation}
			\te(\lambda_1\varphi_1+\lambda_2\varphi_2)=\lambda_1\te(\varphi_1)+\lambda_2\te(\varphi_2).
		\end{equation}
		\item For every $\varphi_1, \varphi_2: \Omega_\varepsilon^a \to \R$ measurable functions,
		\begin{equation}
			\label{eqn:propprod}
			\te(\varphi_1 \varphi_2)=\te(\varphi_1)\te(\varphi_2).
		\end{equation}	
		\item For every $\varphi\in L^p(\Omega_\varepsilon^a)$, where $1\leq p \leq \infty$, we have $\te(\varphi)\in L^p(W)$. In addition,
		\begin{equation}
			\label{eqn:propL2}
			\norm{\te(\varphi)}_{L^p(W)}=\norm{\varphi}_{L^p(\Omega_\varepsilon^a)}.
		\end{equation}
		More precisely, for every $n=1,\ldots,N_\varepsilon$,
		\begin{equation}
			\label{eqn:propL22}
			\norm{\te(\varphi)}_{L^p(\omega_\varepsilon^n \times Y)}=\norm{\varphi}_{L^p(Y_\varepsilon^n)}.
		\end{equation}
		\item For every $\varphi\in H^1(\Omega_\varepsilon^a)$, we have that $\frac{\partial \te(\varphi)}{\partial y}\in L^2(W)$, $\frac{\partial \te(\varphi)}{\partial \xi_i}\in L^2(W)$ for every $i=1,\ldots, N$ and, for a.e. $(x,\xi,y)\in W$
		\begin{equation}
			\label{eqn:propdery}
			\frac{\partial \te(\varphi)}{\partial y}(x,\xi,y)=\te\left(\frac{\partial \varphi}{\partial y}\right)(x,\xi,y)
		\end{equation}
		\begin{equation}
			\label{eqn:propderxi}
			\frac{\partial \te(\varphi)}{\partial \xi_i}(x,\xi,y)=
			\left\{
			\begin{aligned}
				& l_\varepsilon^n\te\left(\frac{\partial \varphi}{\partial x_i}\right)(x,\xi,y) \qquad && x\in \omega_\varepsilon^n, \,n=1,\ldots,N_\eps,\\
				& 0 \qquad && x\not \in \omega_\varepsilon.
			\end{aligned}
			\right.	
		\end{equation}
	\end{enumerate}
\end{prop}
\begin{proof}
	(i) and (ii) are a simple consequence of definition of the unfolding operator.
	
	\noindent (iii) We first do the case $p=1$ of \eqref{eqn:propL22}. Using the change of variables $x=\bar x_\varepsilon^n + l_\varepsilon^n \xi$ and the fact that $\abs{\omega_\varepsilon^n}=(l_\varepsilon^n)^{N}$,
	\begin{equation*}
		\int_{Y^n_\varepsilon}\varphi(x,y)dxdy=  (l^n_\varepsilon)^N\int_{Y}\varphi({ \bar x^n_\varepsilon}+l^n_\varepsilon \xi,y)d\xi dy =  \frac{(l^n_\varepsilon)^N}{\abs{\omega^n_\varepsilon}}\int_{\omega^n_\varepsilon\times Y}\te(\varphi)(x,\xi ,y)dxd\xi dy=\int_{\omega_\varepsilon^n\times Y} \te(\varphi)
	\end{equation*}
	Now, to obtain \eqref{eqn:propL2}, we express $\Omega_\varepsilon^a$ as the disjoint union of the sets $Y_\varepsilon^n$ and use that $\te(\varphi)(x,\xi, y)\equiv 0$ when $x\not \in \omega_\varepsilon$,
	\begin{equation}
		\int_{\Omega^a_\varepsilon}\varphi(x,y)dxdy = \sum_{n=1}^{N_\varepsilon} \int_{Y^n_\varepsilon}\varphi(x,y)dxdy= \sum_{n=1}^{N_\varepsilon} \int_{\omega^n_\varepsilon\times Y}\te(\varphi)(x,\xi ,y)dxd\xi dy=\int_{W} \te(\varphi)	
	\end{equation}
	The case $p=\infty$ is straightforward. If $p<\infty$, by using the result for $p=1$ and the fact that $\te(\varphi)^p=\te(\varphi^p)$,
 $$\norm{\te(\varphi)}^p_{L^p(W)}=\norm{(\te(\varphi))^p}_{L^1(W)}=\norm{\te(\varphi^p)}_{L^1(W)}= \norm{\varphi^p}_{L^1(\Omega_\varepsilon^a)}=\norm{\varphi}^p_{L^p(\Omega_\varepsilon^a)}.$$
	
	\noindent (iv)
	Assume first that $\varphi\in C^\infty(\Omega^a_\varepsilon)$. In that case, by using the chain rule, from \eqref{eqn:defunfolding} one obtains \eqref{eqn:propdery} and \eqref{eqn:propderxi} when $x\in \omega_\varepsilon$. When $x\not\in \omega_\varepsilon$, we just obtain that both derivatives are $0$, which is expressed explicitly in \eqref{eqn:propderxi} and implicitly in \eqref{eqn:propdery}, because $\te(\varphi_y)(x,\xi,y)=0$ when $x\not\in \omega_\varepsilon$. 
	If $\varphi$ is not smooth, the result follows by an approximation argument and \eqref{eqn:propL2}.

\end{proof}

Finally, the following lemma shows how the unfolding operator applied to an $L^2(\R^{N+1})$ function converges, when $\varepsilon\to 0$, to the same function. Recall that, by the conventions from Section \ref{sec:notations}, without risk of confusion we denote $f_{|_{\Omega^a_\eps}}$ the restriction of $f$ to $\Omega_\varepsilon^a$ by simply $f$.
\begin{lemma}
	\label{lemma:fconv}
	If $f\in L^2(\R^{N+1})$, we have
	\begin{equation}\label{convergence1}
		\lim_{\varepsilon\to 0}\int_W\abs{\te(f)(x,\xi,y)-f(x,y)\chi_{\omega_\varepsilon}(x)}^2dxd\xi dy=0.
	\end{equation}
\end{lemma}
\begin{proof}   Assume first that $f\in C^\infty_c(\R^{N+1})$. Recalling the notation of \eqref{eqn:z}, for any $(x,\xi,y)\in W$ such that $x\in \omega_\varepsilon$, we have $\te(f)(x,{\xi},y)=f(x+z_\eps(x,\xi),y)$. Then, due to \eqref{eqn:zeps} and using the regularity of $f$, we obtain
	\begin{equation}
		\label{eqn:fconveq1}
		\abs{\te(f)(x,{\xi},y)-f(x,y)}\leq C\varepsilon.
	\end{equation}
	In addition, when $(x,\xi,y)\in W$ with $x\not \in \omega_\varepsilon$, we have that $\te(f)(x,\xi,y)=0$, so $\te(f)(x,{\xi},y)-\chi_{\omega_\varepsilon}(x)f(x,y)=0$, so, combining this with \eqref{eqn:fconveq1},
	\begin{equation}
		\label{eqn:fconveq2}
		\abs{\te(f)(x,{\xi},y)-\chi_{\omega_\varepsilon}(x)f(x,y)}\leq C\varepsilon \qquad (x,\xi,y)\in W.
	\end{equation}
	Therefore, integrating in $W$, we obtain the result. If $f$ is not smooth, the result follows by an approximation argument and \eqref{eqn:propL2}.
	
\end{proof}
\subsection{The space of solutions and some functional framework}
\label{sec:functional}
At the end of this Subsection we will define rigorously the space of solutions of the limit unfolded problem. But to reach to this point we need to include some definitions and some results on the functional framework we are using.

\subsubsection{Weighted Lebesgue-Bochner Spaces}
\label{sec:weightedbochner}

This subsection is devoted to the description of the functional spaces we are going to use, which consists mainly in weighted Lebesgue-Bochner spaces.
A good introduction to these spaces and their main properties can be found in \cite{diestel1977vector}, \cite[Chapter 1]{banachspacesI}, \cite[Section V.5]{yosida2012functional} and \cite[Section 1.1.]{arendt2011vector}. We also refer to \cite[Appendix D]{mythesis} for an introduction adapted to this particular problem.

Let ${\hil}$ be a separable Hilbert space and $U\subset \RN$ a domain. Let $\mu:U \longrightarrow [0,\infty)$ bounded and measurable, and consider the associated measure $d\mu\defeq \mu dx$ where $dx$ represents the Lebesgue measure. We will denote by
$$
L^2(U,\mu; {\hil})
$$
the space of (equivalence classes of) Bochner $\mu$-measurable\footnote{
 The measure space $(U, \Sigma, \mu)$, where $\Sigma$ denotes the Lebesgue $\sigma$-algebra, is not necessarily a complete measure space (because of the non-measurable sets of $U$ where $\mu$ vanishes). However, it can easily completed adding the necessary sets to the $\sigma$-algebra, $(U, \Sigma_\mu, \mu)$. We will refer to $\mu$-measurability, to distinguish from standard (Lebesgue) measurability, to the measurability of functions with respect to that $\sigma$-algebra.
}
 functions $u: U\longrightarrow {\hil}$ such that
\begin{equation}
	\int_U \mu(x) \norm{u(x)}_{\hil}^2dx<\infty.
\end{equation}
In the particular case in which $\hil=\R$ we will simply use the notation $L^2(U,\mu)$. As in the case of the standard $L^2$ spaces, $L^2(U,\mu; {\hil})$ is a Hilbert space with the scalar product
\begin{equation}
	\label{eqn:bochnerescalarproduct2}
	\braket{u,v}_{L^2(U,\mu; {\hil})}= \int_U \mu(x)\braket{u(x),v(x)}_H dx.
\end{equation}
Again, as in the standard theory of Lebesgue spaces, $L^2(U,\mu; {\hil})$ is made of equivalence classes and without risk of confusion, whenever we refer to $u\in L^2(U,\mu; {\hil})$, we mean that $u$ is a representative of the equivalence class. It is important to emphasize that the value of $u$ on the set where $\mu$ vanishes is not meaningful, as it depends on the chosen representative. That is, two functions belonging to the same equivalence class may differ completely on the set where $\mu$ vanishes.

An important result regarding these spaces is the following.
\begin{lemma}
	\label{lemma:l2dense}
	Let ${\hil}$, $U$ and $\mu$ as above. Then, $L^2(U; {\hil})$ is canonically embedded in $L^2(U,\mu;{\hil})$ and
	\begin{equation}
		\qquad \norm{g}_{L^2(U,\mu;{\hil})}\leq C_{\mu}\norm{g}_{L^2(U;{\hil})}, \qquad \forall g\in L^2(U;{\hil}),
	\end{equation}
	where $C_{\mu}=\sqrt{\sup_{x\in U} \mu(x)}$. In addition, $L^2(U;{\hil})$ is dense in $L^2(U,\mu;{\hil})$.
\end{lemma}
\begin{proof}
		Let $g\in L^2(U;{\hil})$. As $g$ is Bochner measurable, it can be approximated by simple functions a.e., so it can be approximated by simple functions $\mu-$a.e. and therefore it is Bochner $\mu-$measurable.
	The embedding result is straightforward:
	\begin{equation}
		\norm{g}_{L^2(U,\mu;{\hil})}=\left(\int_U \mu(x) \norm{g}_{\hil}^2dx\right)^{1/2} \leq C_\mu \left(\int_U \norm{g}_{\hil}^2dx\right)^{1/2}=C_\mu\norm{g}_{L^2(U;{\hil})}.
	\end{equation}
	Now, let us prove the density of $L^2(U;{\hil})$ in $L^2(U,\mu;{\hil})$. Take $g\in L^2(U,\mu;\hil)$ and let us approximate it by functions in $L^2(U;\hil)$. Define, for each $n\in \mathbb{N}$, $U_n=\left\{x\in U: \frac{1}{n+1}\leq\mu(x)<\frac{1}{n}\right\}$
	and $U_0=\{x\in U: 1\leq\mu(x)\}$. These sets are disjoint and its union is $\cup_{n=0}^\infty U_n = U\setminus M_0$, where $M_0=\{x\in U:\mu(x)=0\}$. But then,
	\begin{equation}
		\sum_{n=0}^\infty\int_{U_n} \mu(x)\norm{g(x)}^2_{\hil}dx = 
		\int_{U\setminus M_0} \mu(x)\norm{g(x)}^2_{\hil}dx = \int_{U} \mu(x)\norm{g(x)}^2_{\hil}dx <\infty.
	\end{equation}
	Therefore, given $\epsilon>0$, there exists $n_0\in \mathbb{N}$ such that
	\begin{equation}
		\label{eqn:l2denseeq1}
		\sum_{n=n_0+1}^\infty\int_{U_n} \mu(x)\norm{g(x)}^2_{\hil}dx < \epsilon
	\end{equation}
	Define then $h(x)\defeq g(x)$ when $x\in \cup_{n=0}^{n_0}U_n$ and $h(x)=0$ otherwise.
	Then, using that $n_0\mu(x)\geq 1$ for every $x\in \cup_{n=0}^{n_0}U_n$, we have
	\begin{equation}
		\norm{h}^2_{L^2(U;{\hil})}=\sum_{n=0}^{n_0}\int_{U_n}\norm{g(x)}_{\hil}^2dx\leq \sum_{n=0}^{n_0}\int_{U_n}(n_0\mu(x))\norm{g(x)}_{\hil}^2dx \leq n_0\norm{g}_{L^2(U,\mu;{\hil})}
	\end{equation}
	so $h\in L^2(U;{\hil})$. In addition,
	\begin{equation}
		\norm{g-h}^2_{L^2(U,\mu;{\hil})}= \sum_{n=n_0+1}^\infty\int_{U_n}\mu(x)\norm{g(x)}_{\hil}^2dx
		+\int_{M_0}\mu(x)\norm{g(x)}_{\hil}^2dx \myleq{\eqref{eqn:l2denseeq1}} \epsilon + 0
	\end{equation}
	because $\mu \equiv 0$ in $M_0$. As $\epsilon>0$ was arbitrary, this concludes the proof.
\end{proof}

\subsubsection{The unfolding operator and Bochner spaces}
As we have already mentioned, we will need to use Bochner spaces to describe the problem in the unfolded domain. This is because, when applying the unfolding operator $\te$ to a function $\varphi \in H^1(\Omega_\varepsilon)$, as we saw in Remark \ref{rem:z}, $\te(\varphi)$ is piecewise constant with respect to $x$ so we lose the weak differentiability in the $x$ variable. However, as shown in Proposition \ref{prop:prop}, the function $\te(\varphi)$ does have derivatives in the $\xi$ and $y$ directions, that is, it belongs to the following space:
$$
\left\{ \psi \in L^2(W): \frac{\partial \psi}{\partial y}\in L^2(W), \ \nabla_{\xi} \psi \in (L^2(W))^N\right\}
$$
It turns out that this space coincides with the Bochner space
$$
L^2(\Omega';H^1(Y)).
$$
This is a natural though not trivial fact, and to avoid overloading this section, we include the proof in Appendix \ref{app:bochner}, Lemma \ref{lemma:bochnerproductb}. In fact, due to equations \eqref{eqn:propL22}, \eqref{eqn:propdery} and \eqref{eqn:propderxi}, the following norm inequality holds:
\begin{equation}
	\label{eqn:propnorm}
	\norm{\te(\varphi)}_{L^2(\Omega';H^1(Y))}\leq C\norm{\varphi}_{H^1(\Omega_\varepsilon^a)}, \qquad  \varphi\in H^1(\Omega_\varepsilon), \ 0<\varepsilon<1.
\end{equation}
where $C$ is independent of $\varepsilon>0$. Note that \eqref{eqn:propnorm} together with the linearity of $\te$ implies the continuity of $\te$ as an operator from $H^1(\Omega_\varepsilon^a)$ to $L^2(\Omega';H^1(Y))$.

Thanks to \eqref{eqn:propnorm}, having a priori estimates in the $H^1(\Omega_\varepsilon^a)$ norm, in Section \ref{sec:conv} we will prove the weak convergence of the functions $\te(\varphi)$. When obtaining such convergence, the limit solution will be expressible as $\theta u^a$ where $u^a$ belongs to a weighted space of the form
\begin{equation}
	L^2(\Omega',\theta;H^1(Y))
\end{equation}
where $\theta$ is the limit density from \eqref{eqn:ass2}.

\subsubsection{Traces}
\label{subsubsec:traces}

As we already mentioned in the introduction, the solutions to the limit problem in the upper part will belong to a weighted Bochner space of the form
$L^2(\Omega',\theta; H^1(Y))$ and we will also need to work with the unweighted space $L^2(\Omega'; H^1(Y))$. 
In both cases, we will need to deal with traces of their functions in $W_0$, so in this section we will present some results concerning traces. In order to unify the proofs, we will consider $\mu$ as in Section \ref{sec:weightedbochner} so $L^2(\Omega',\mu; H^1(Y))$ will represent the spaces with and without the weight $\theta$.

First of all, we observe that,
since $\partial Y \cap (\RN\times\{0\}) = \adh{\omega}\times \{0\}$,
from \eqref{eqn:delta0}, $\omega\times(0,\delta_0)\subset Y$ and $\omega\times(0,\delta_0)$ is a Lipschitz domain, the space $H^1(Y)$ has a well defined linear and bounded trace operator in $L^2(\omega)$ that we denote by $tr:H^1(Y)\to L^2(\omega)$ (see for example \cite{evansmeasure} Section 4.3). Using this operator we can define now an extended trace operator in the space $L^2(\Omega',\theta;H^1(Y))$

\begin{definition}
	\label{def:trace}
	Given a function $\varphi\in L^2(\Omega',\mu; H^1(Y))$, we define the ``above trace'' as a function defined in $\Omega'\times \omega$ given by
	\begin{equation}
		\label{eqn:traceeq1}
		{Tr}^a(\varphi)(x,\xi)=tr(\varphi(x))(\xi) \qquad \mu-a.e. \ x\in \Omega', \ a.e. \ \xi\in \omega.
	\end{equation}
\end{definition}
Note that, without risk of confusion, we will denote ${Tr}^a$ the trace of an $L^2(\Omega',\theta;H^1(Y))$ or an $L^2(\Omega';H^1(Y))$ function.
The defined trace operators are in fact a bounded linear operators as the following lemma shows

\begin{lemma}
	\label{lemma:boundtra}
	Equation \eqref{eqn:traceeq1} defines a bounded and linear operator
	\begin{equation}
		\label{eqn:boundtraeq1}
		\begin{aligned}
			{Tr}^a: \  & L^2(\Omega',\mu;H^1(Y)) && \to &&& L^2(W_0,\mu) \\
			& \varphi && \mapsto &&& {Tr}^a(\varphi).
		\end{aligned}
	\end{equation}
	Moreover,
	\begin{equation}
		\norm{{Tr}^a}_{\mathscr{L}(L^2(\Omega',\mu;H^1(Y)),L^2(W_0,\mu))}\leq \norm{tr}_{\mathscr{L}(H^1(Y),L^2(\omega))}.
	\end{equation}
\end{lemma}
Note that, in equation \eqref{eqn:boundtraeq1}, $\mu$ is interpreted as a function defined in $W_0=\Omega'\times \omega$ so we are using the conventions from Section \ref{sec:notations}.

\begin{proof}
	The linearity is immediate. The boundedness follows from the identification $L^2(W_0,\mu)=L^2(\Omega'\times\omega,\mu)=L^2(\Omega',\mu;L^2(\omega))$ (see Lemma \ref{lemma:bochnerproduct} and Remark \ref{remark:bochnerproduct}) and from the following calculation
	\begin{equation}
		\begin{array}{l} 
			\displaystyle 
			\|{Tr}^a(\varphi)\|_{L^2(\Omega',\mu; L^2(\omega))}=\int_{\Omega'}\mu(x)\int_{\omega}\abs{tr(\varphi(x))({\xi})}^2d\xi dx \\ \\ \displaystyle \qquad \leq \int_{\Omega'}\mu(x)\norm{tr}_{\mathcal{L}(H^1(Y), L^2(\omega))}^2\norm{\varphi(x)}^2_{H^1(Y)} 
			\displaystyle =\norm{tr}_{\mathcal{L}(H^1(Y), L^2(\omega))}^2 \norm{\varphi}^2_{L^2(\Omega,\mu;H^1(Y))}.
		\end{array}
	\end{equation}
\end{proof}

In the lower part, as $\Omega^b$ is a Lipschitz domain, we also have a trace operator 
$${Tr}^b: H^1(\Omega^b)\to L^2(\Omega')$$ defined in the usual standard way. This is what we will call the ``below trace''.

The following result states the expected behaviour of the traces when applying the unfolding operator.

\begin{prop}
	\label{prop:tr}
	Let $\varepsilon>0$ and $\varphi\in H^1(\Omega_\varepsilon)$. Then, for any $n=1,\dots,N_\varepsilon$,
	\begin{equation}
		\label{eqn:proptreq10}
		{Tr}^a(\te(\varphi))(x,\xi)={Tr}^b(\varphi)(\bar x_\varepsilon^n+l_\varepsilon^n\xi) \qquad a.e. \ x\in \omega_\varepsilon^n, \ a.e. \ \xi\in \omega.
	\end{equation}
\end{prop}

\begin{proof}
	If $\varphi$ is smooth, the result is immediate from the definition of the unfolding operator, \eqref{eqn:defunfolding}.
	
	Now, consider the following operator
	\begin{equation}
		\begin{aligned}
			T: \  & H^1(\Omega_\varepsilon) && \to &&& \R \qquad\qquad\qquad\qquad \qquad\quad \\
			& \varphi && \mapsto &&& \int_{\omega_\varepsilon^n \times \omega}\abs{{Tr}^a(\te(\varphi))(x,\xi)-{Tr}^b(\varphi)(\bar x_\varepsilon^n+l_\varepsilon^n\xi)}^2dxd\xi.
		\end{aligned}
	\end{equation}
	Due to the fact that $\te$ is bounded and linear from $H^1(\Omega_\varepsilon)$ to $L^2(\Omega';H^1(Y))$, ${Tr}^a$ is bounded and linear from $L^2(\Omega';H^1(Y))$ to $L^2(W_0)$ and ${Tr}^b$ is bounded and linear from $H^1(\Omega_\varepsilon)$ to $L^2(\Omega')$, we obtain that $T$ is bounded and linear. As $T(\varphi)=0$ for any $\varphi\in C^\infty(\Omega_\varepsilon)\cap H^1(\Omega_\varepsilon)$, and this set is dense in $H^1(\Omega_\varepsilon)$, $T(\varphi)=0$ for any $\varphi\in H^1(\Omega_\varepsilon)$ and the proof is finished.
\end{proof}

\subsubsection{Partially Constant Sobolev Spaces}

Some spaces needed for the analysis will require that some partial derivatives are identically zero.
Concretely, the derivatives in the directions of the components of ${\xi}$. We will use the following notation
\begin{definition}
	\label{def:functspaces1}  Let $S\subset\R^{N+1}$ be an open set. Use the notation $(\xi,y)$ to denote a point in $\R^{N+1}$ where $\xi\in \RN$ and $y\in \R$. We define the space
	\begin{equation}
		H^1_ {\left<\xi\right>}(S)\defeq \{u\in H^1(S): \nabla_{{\xi}} u(\xi,y)=0 \ a.e. \ (\xi,y)\in S\} \subset H^1(S). 
	\end{equation}
\end{definition}

	As this space is a closed subspace of the Hilbert space $H^1(S)$, it is a Hilbert space with the inherited scalar product.	

Note that functions belonging to these spaces are not necessarily independent of $\xi$. They may depend on $\xi$ but they will be locally constant on $\xi$, that is constant in connected horizontal components, as the following lemma shows.
\begin{lemma}
	\label{lemma:constxi1}
	Let $S\subset\R^{N+1}$ be an open set and let $\varphi\in H^1_{\braket{\xi}}(S)$. Then, up to a redefinition in a set of measure $0$ of $S$, for every $(\xi_0,y_0),(\widetilde{\xi_0},y_0)\in S$ such that
	$\{((1-t)\xi_0+t\widetilde{\xi_0},y):t\in [0,1]\}\subset S$, we have
	$\varphi(\xi_0,y)=\varphi(\widetilde{\xi_0},y)$.
\end{lemma}
\begin{proof}	
	The proof for a product of intervals, $\prod_{i=1}^{N+1} (a_i,b_i)$, can be found in \cite[Lemma 2.3]{prizzi2001effect}. Let us prove the general case. First, we can cover $S$ with a countable set of product of intervals $I_m=\prod_{i=1}^{N+1} (a^m_i,b^m_i)$.Then, we can redefine $\varphi$ in a set of measure zero such that $\varphi$ is constant in $\xi$ in each $I_m$ (note that we are using that the countable union of sets of measure $0$ is a set of measure $0$). Now, take any $\{((1-t)\xi_0+t\widetilde{\xi_0},y):t\in [0,1]\}\subset S$ as in the statement of the Lemma. Assume $\varphi$ is not constant in $\{((1-t)\xi_0+t\widetilde{\xi_0},y):t\in [0,1]\}$. Let $t_0=\inf\{t\in[0,1]:\varphi(\xi_0,y)=\varphi((1-t)\xi_0+t\widetilde{\xi_0},y)\}$. Then, for some $m_0\in \mathbb{N}$ we have $((1-t_0)\xi_0+t_0\widetilde{\xi_0},y)\in I_{m_0}$. But then, as $I_{m_0}$ is an open set, for some $\delta>0$
	\begin{equation}
		((1-t)\xi_0+t\widetilde{\xi_0},y)\in I_{m_0} \qquad t_0-\delta<t<t_0+\delta,
	\end{equation}
	and consequently
	\begin{equation}
		\varphi((1-t)\xi_0+t\widetilde{\xi_0},y)=\varphi((1-t_0)\xi_0+t_0\widetilde{\xi_0},y) \qquad t_0-\delta<t<t_0+\delta,
	\end{equation}
	contradicting the definition of $t_0$.
\end{proof}

As a consequence of the previous lemma, we have that, if $S\subset \R^{N+1}$ is a domain such that, for every $y\in \R$, $\{\xi\in\RN: (\xi,y)\in S\}$ is connected (this is what we will call to have connected horizontal sections, see Definition \ref{def:connhorsec}), then, any $\varphi\in H^1_{\braket{\xi}}(S)$ is independent of $\xi$. The following proposition states this fact.
\begin{prop}
	\label{prop:constxi}
	Let $S\subset \R^{N+1}$ be a domain and $\varphi\in H^1_{\braket{\xi}}(S)$. Assume that, for every $y\in \R$, the section of $S$ at height $y$,
	$$S_y=\{\xi\in\RN: (\xi,y)\in S\}$$
	is connected. Let $(a,b)=\{y\in \R: S_y\neq \emptyset\}$ and define
	\begin{equation}
		\label{eqn:constxieq1}
		p_S(y)=\abs{\{\xi\in\RN: (\xi,y)\in S\}}\qquad y\in (a,b).
	\end{equation}
	Then, there exists $\widetilde{\varphi}\in H^1((a,b),p_S)$ such that
	\begin{equation}
		\label{eqn:prizzieq1bis}
		\varphi(\xi,y)=\widetilde{\varphi}(y), \quad \frac{\partial \varphi}{\partial y}(\xi,y)=\frac{\partial \widetilde{\varphi}}{\partial y}(y), \qquad a.e. \ (\xi,y)\in S.
	\end{equation}
	In addition
	\begin{equation}
		\label{eqn:prizzieq2bis}
		\norm{\varphi}_{H^1(S)}=\norm{\widetilde{\varphi}}_{H^1((a,b),p_S)}.
	\end{equation}
	Conversely, if $\widetilde{\varphi}\in H^1((a, b),p_S)$, it induces the definition of the function $\varphi\in H^1_{\braket{\xi}}(S)$ satisfying \eqref{eqn:prizzieq1bis} and \eqref{eqn:prizzieq2bis}.
\end{prop}
\begin{proof}
	Again, the proof of this result is based on the work \cite{prizzi2001effect}. As $S$ is an open set, so is $S_y\subset \RN$ for every $y\in \R$. As $S_y$ is open and connected set, it is polygonally connected (i.e. any two points of $S_y$ can be connected by a piecewise linear path in $S_y$), so, using Lemma \ref{lemma:constxi1}, we obtain that $\varphi$ is independent of $\xi$. Define then
	\begin{equation}
		\widetilde{\varphi}(y)=\varphi(\cdot,y) \qquad y\in (a,b)
	\end{equation}
	where $\varphi(\cdot,y)$ denotes $\varphi(\xi,y)$ for any $(\xi,y)\in S$. Then,
	\begin{equation}
		\norm{\widetilde{\varphi}}^2_{L^2((a,b),p_S)}=\int_a^b\abs{\widetilde{\varphi}(y)}^2p_{S}(y)dy=\int_a^b\int_{S_y}\abs{\varphi(\xi,y)}^2d\xi dy=\norm{\varphi}^2_{L^2(S)}
	\end{equation}
	so $\norm{\widetilde{\varphi}}_{L^2((a,b),p_S)}=\norm{\varphi}_{L^2(S)}$.
	
	Now, in the distributional sense, we have that $\nabla_{\xi}\frac{\partial \varphi}{\partial y}=\frac{\partial }{\partial y}\nabla_{\xi}\varphi=0$. Therefore, repeating the argument above one obtains that there exists $w\in L^2((a,b),p_S)$ such that $\frac{\partial \varphi}{\partial y}(\xi,y)=w(y)$.
	We want to prove that $w=\frac{\partial \widetilde{\varphi}}{\partial y}$. As the differentiability is a local property, we can restrict our analysis to a product of intervals, so let us assume $S_y=P=\prod_{i=1}^{N}(c_i,d_i)$ for every $y\in (a,b)$. Given any $\eta\in C^\infty_c(a,b)$ and $\psi\in C^\infty_c(P)$ such that $\int\psi=1$,
	\begin{equation}
		\begin{aligned}
			\int_a^b\frac{\partial \eta}{\partial y}(y)\tilde{\varphi}(y)dy= & \int_a^b\fint_P\psi(\xi)\frac{\partial \eta}{\partial y}(y)\varphi(\xi,y)d\xi dy = \int_a^b\fint_P\psi(\xi)\eta(y)\frac{\partial \varphi}{\partial y}(\xi,y)d\xi dy
			\\ = & \int_a^b\fint_P\psi(\xi)\eta(y)w(y)d\xi dy  = \int_a^b\eta(y)w(y)dy
		\end{aligned}
	\end{equation}
	so $w=\frac{\partial \widetilde{\varphi}}{\partial y}$.\\
	
	Let us prove the converse result. First, define
	\begin{equation}
		\varphi(\xi,y)=\widetilde{\varphi}(y), \quad \psi(\xi,y)=\frac{d \widetilde{\varphi}}{dy}(y) \qquad a.e. \ (\xi,y)\in S,
	\end{equation}
	Let us prove that $\frac{\partial \varphi}{\partial y}=\psi$ in the distributional sense. Take $\rho\in C^\infty_c(S)$ and extend it to $\rho\in C^\infty_c(\RN)$, then
	\begin{equation}
		\begin{aligned}
			\int_S \varphi(\xi,y)\frac{\partial \rho}{\partial y}(\xi,y)d\xi dy & = \int_{a}^{b}\widetilde{\varphi}(y)\left(\int_{\RN}\frac{\partial \rho}{\partial y}(\xi,y)d\xi\right)dy=\int_{a}^{b}\widetilde{\varphi}(y)\left(\frac{d}{dy}\int_{\RN}\rho(\xi,y)d\xi\right)dy
			\\
			& = \int_{a}^{b}\frac{d\widetilde{\varphi}}{dy}(y)\left(\int_{\RN}\rho(\xi,y)d\xi\right)dy = \int_{S}\psi(\xi,y)\rho(\xi,y)d\xi dy
		\end{aligned}
	\end{equation}
	so $\frac{\partial \widetilde{\varphi}}{\partial y}=\psi$. Then, \eqref{eqn:prizzieq2bis} is obtained in a straightforward way using Fubini's theorem.
\end{proof}

Another consequence of Lemma \ref{lemma:constxi1} is that the trace of functions of $L^2(\Omega',\theta;H^1_{\braket{\xi}}(Y))$ is constant in $\xi$. Note that this implies that they will belong to the space $L^2(\Omega',\theta)$. The following lemma summarizes and proves this fact.

\begin{lemma}
	\label{lemma:traces}
	Given $\varphi\in H^1_{\braket{\xi}}(Y)$, we have that $tr(\varphi)$ is a constant function in $\omega$. In other words, the trace operator restricted to $H^1_{\braket{\xi}}(Y)$, denoted by $tr_{\braket{\xi}}$ is a linear and bounded operator from $H^1_{\braket{\xi}}(Y)$ to $\R$. As a consequence of this fact, given $\varphi \in L^2(\Omega',\theta;H^1_{\braket{\xi}}(Y))$, we have that
	${Tr}^a(\varphi)$ is independent of $\xi$ and can be interpreted as a function in $L^2(\Omega',\theta)$. Therefore, the restriction of ${Tr}^a$ to $L^2(\Omega',\theta;H^1_{\braket{\xi}}(Y))$, denoted by ${Tr}^a_{\braket{\xi}}$ is a bounded operator from $L^2(\Omega',\theta;H^1_{\braket{\xi}}(Y))$ to $L^2(\Omega',\theta)$. The same applies to $L^2(\Omega';H^1_{\braket{\xi}}(Y))$.
\end{lemma}
\begin{proof}
	Let us prove that $tr$ maps $H^1_{\braket{\xi}}(Y)$ into $\R$. First of all, as $\omega\times(0,\delta_0)\subset Y$, we can restrict our analysis to $\omega^{\delta_0}=\omega\times(0,\delta_0)$. In this case, as $\varphi\in H^1_{\braket{\xi}}(\omega^{\delta_0})$ and $\omega$ is connected, by Proposition \ref{prop:constxi}, there exists $\tilde{\varphi}\in H^1((0,\delta_0))$ such that
	\begin{equation}
		\varphi(\xi,y)=\tilde{\varphi}(y), \quad \frac{\partial \varphi}{\partial y}(\xi,y)=\frac{\partial \tilde{\varphi}}{\partial y}(y), \qquad (\xi,y)\in Y.
	\end{equation}
	Now, given $\delta>0$, we can find an approximating sequence of smooth functions $\{\tilde{\varphi}_n\}\subset C^\infty(\R)\cap H^1((0,\delta_0))$ such that $\tilde{\varphi}_n\to \tilde{\varphi}$ in $H^1((0,\delta_0))$ with $n\to\infty$. Then, defining $\varphi_n(\xi ,y)=\tilde{\varphi}_n(y)$, we obtain $\varphi_n\to \varphi$ in $H^1(\omega^{\delta_0})$. As the trace operator is continuous we also have $tr(\varphi_n)\to tr(\varphi)$ when $n\to\infty$. Now, we know that $tr(\varphi_n)$ are constant functions so let us denote them $tr(\varphi_n)(\xi)=C_n$. Then, we have
	\begin{equation}
		\abs{C_n-C_m}=\int_{\omega}\abs{tr(\varphi_n)(\xi)-tr(\varphi_m)(\xi)}d\xi \leq \norm{tr(\varphi_n)-tr(\varphi_m)}_{L^2(\omega)}\to 0
	\end{equation}
	when $n,m\to \infty$. Therefore $\{C_n\}$ is a Cauchy sequence and there exists $C_0\in\R$ such that $C_n\to C_0$ when $n\to 0$. Then
	\begin{equation}
		\int_{\omega}\abs{tr(\varphi_n)-C_0}^2d\xi=\abs{C_n-C_0}^2\to 0
	\end{equation}
	when $n\to \infty$. That is, $\{tr(\varphi_n)\}$ converges to a constant function. By uniqueness of the limit in $L^2(\omega)$, this proves that $tr(\varphi)$ is a constant function.
\end{proof}
\subsubsection{The space of solutions for the limit unfolded problem}
\label{sec:spacesols}
\par\medskip 
Finally we define the space of solutions of the limit unfolded problem \eqref{eqn:thmmaineq3}. We already introduced this space in Section \ref{sec:statmentjtx} without going into detail. Now that we have described the weighted Bochner spaces, their trace operators and the partially constant Sobolev spaces $H^1_{\left<\xi\right>}(Y)$, we can precisely define our Hilbert space of solutions of the limit problem.
\begin{definition}
	We define the space of solutions as
	\begin{equation}
		\label{def:spaceofsolY}
		\begin{aligned}
			H(\theta)\defeq \bigg\{(u^a,u^b) : \ & u^a\in L^2(\Omega',\theta; H^1_{\left<\xi\right>}(Y)), \ u^b\in H^1(\Omega^b), \\
			& {Tr}^a_{\braket{\xi}}(u^a)(x)=Tr^b(u^b)(x) \quad  \theta-a.e. \ x\in\Omega'\bigg\}
		\end{aligned}	
	\end{equation}
	Recall that ${Tr}^a_{\braket{\xi}}$ is defined in Lemma \ref{lemma:traces} as the restriction of ${Tr}^a$, from Definition \ref{def:trace}, to the space $L^2(\Omega',\theta; H^1_{\left<\xi\right>}(Y))$, whose range belongs to $L^2(\Omega',\theta)$.
\end{definition}

Now, we include a lemma which proves that $H(\theta)$ is a Hilbert space and define its correspondent scalar product

\begin{lemma}
	The set $H(\theta)$ is a Hilbert space endowed with the scalar product
	\begin{equation}
		\label{eqn:vpht}
		\braket{(u^a, u^b)|(v^a,v^b)}=\braket{u^a|v^a}_{L^2(\Omega',\theta; H^1_{\left<\xi\right>}(Y))}+\braket{u^b|v^b}_{H^1(\Omega^b)}
	\end{equation}
\end{lemma}
\begin{proof}
	$L^2(\Omega',\theta; H^1_{\left<\xi\right>}(Y))$ and $H^1(\Omega^b)$ are Hilbert spaces, so their Cartesian product $L^2(\Omega',\theta; H^1_{\left<\xi\right>}(Y))\allowbreak\times H^1(\Omega^b)$ is a Hilbert space too with the scalar product \eqref{eqn:vpht}. Now, let us prove that $H(\theta)$ is a closed subspace of $L^2(\Omega',\theta; H^1_{\left<\xi\right>}(Y))\times H^1(\Omega^b)$. $Tr^b$ is a continuous linear operator from $H^1(\Omega^b)$ to $L^2(\Omega')$. But $L^2(\Omega')$, by Lemma \ref{lemma:l2dense}, is continuously embedded in $L^2(\Omega',\theta)$, so $Tr^b$ is a continuous linear operator from $H^1(\Omega^b)$ to $L^2(\Omega',\theta)$. In addition,
	${Tr}^a_{\braket{\xi}}$ is a continuous linear operator from $L^2(\Omega',\theta; H^1_{\left<\xi\right>}(Y))$ to $L^2(\Omega',\theta)$. Therefore, ${Tr}^a_{\braket{\xi}}-{Tr}^b$ can be interpreted as a continuous linear operator from $L^2(\Omega',\theta; H^1_{\left<\xi\right>}(Y))\times H^1(\Omega^b)$ to $L^2(\Omega',\theta)$, so $H(\theta)$, which is by definition, the preimage of $0$ of ${Tr}^a_{\braket{\xi}}-{Tr}^b$, is a closed subspace of $L^2(\Omega',\theta; H^1_{\left<\xi\right>}(Y))\times H^1(\Omega^b)$, hence, a Hilbert space endowed with its scalar product.
\end{proof}

\section{Derivation of the results}
\label{sec:jtxresults}

This section is devoted to proving the main result Theorem \ref{thm:main} stated in Section \ref{sec:statmentjtx}. The organization of the section is as follows.
Firstly, in subsection \ref{sec:conv} we will prove that, given an arbitrary bounded family of functions $u_\varepsilon\in H^1(\Omega_\varepsilon)$, we can obtain convergence of $u_\varepsilon\to u^b$ and $\te(u_\varepsilon)\to \theta u^a$ at least in some weak sense.
Although these results apply to arbitrary families of functions, we will apply them later to the specific case of the solutions of equation \eqref{eqn:main}.
Then, in subsection \ref{sec:test} we will describe the test functions we are going to use in \eqref{eqn:main} to identify the limit functions $(u^a,u^b)$. 
In subsection \ref{sec:limit} we will identify $(u^a,u^b)$ as the solutions of \eqref{eqn:thmmaineq3}. 
Finally, in subsection \ref{sec:energy} we will prove the convergence of energies, improving, as a consequence, the convergence of $u_\varepsilon$ and $\te(u_\varepsilon)$ to some strong convergence. Then, we will prove Theorem \ref{thm:main} as a consequence of the previous results.

\subsection{Convergences}
\label{sec:conv}
Let us see that given a uniformly bounded family of functions $u_\varepsilon\in H^1(\Omega_\varepsilon)$, after the application of the unfolding operator, it converges at least through a subsequence of $\varepsilon\to 0$.
\begin{prop}\label{prop:convergences}
	Let $u_\varepsilon$ be a family of functions satisfying $\|u_\eps\|_{H^1(\Omega_\eps)}\leq C_1$ for some $C_1>0$ independent of $\eps$. Then, $\te(u_\varepsilon)\in L^2(\Omega';H^1(Y))$ and there exists $(u^a, u^b)\in H(\theta)$ such that, through a subsequence of $\varepsilon\to 0$,
	\begin{enumerate}
		\item $u_\varepsilon \to u^b$ strongly in $L^2(\Omega^b)$ and $u_\varepsilon \wto u^b$ weakly in $H^1(\Omega^b)$
		\item $\te(u_\varepsilon)\wto \theta u^a$ weakly in $L^2(\Omega';H^1(Y))$.
	\end{enumerate}
	In addition,
	\begin{equation}
		\norm{u^b}_{H^1(\Omega^b)}\leq C_1, \qquad \norm{u^a}_{L^2(\Omega',\theta;H^1_{\braket{\xi}}(Y))}\leq 2C_1.
	\end{equation}
\end{prop}
\begin{proof}  We divide the proof in three steps.

\medskip	\noindent \textbf{Step 1: }
	Using Banach-Alaoglu theorem, we know that there exists $u^b\in H^1(\Omega^b)$ and a subsequence of $\varepsilon\to 0$ such that $u_\varepsilon\wto u^b$ weakly in $H^1(\Omega^b)$. Then, we trivially obtain $\norm{u^b}_{H^1(\Omega^b)}\leq C_1$. As $H^1(\Omega^b)$ is compactly embedded in $L^2(\Omega^b)$ (see \cite[Theorem 6.3]{adams}), we also obtain $u_\varepsilon\to u^b$ strongly in $L^2(\Omega^b)$. In the same way, as the space of traces of $H^1(\Omega^b)$ functions is compactly embedded in $L^2(\Omega')$ (see \cite[Theorem 6.3]{adams}), we have ${Tr}^b(u_\varepsilon^b)\to {Tr}^b(u^b)$ strongly in $L^2(\Omega')$.	
	
	Furthermore, by \eqref{eqn:propnorm}, we have that $\{\te(u_\varepsilon)\}_{\varepsilon>0}$ are uniformly bounded in $L^2(\Omega';H^1(Y))$, so, again by Banach-Alaoglu theorem, there exists $\tilde{u}\in L^2(\Omega';H^1(Y))$ and we can choose the subsequence such that  $\te(u_\varepsilon)\wto \widetilde{u}$ weakly in $L^2(\Omega';H^1(Y))$. In particular, from \eqref{eqn:propderxi}, we have $\norm{\nabla_\xi u_\varepsilon}_{L^2(W)}\leq C\varepsilon$ and therefore we obtain $\norm{\nabla_\xi u}_{L^2(W)}=0$ so
	$\widetilde{u}\in L^2(\Omega';H^1_ {\left<\xi\right>}(Y))$. In addition, as, by Lemma \ref{lemma:traces}, ${Tr}^a$ is a bounded linear operator from $L^2(\Omega';H^1(Y))$ to $L^2(W_0)$, we have ${Tr}^a(\te(u_\varepsilon))\wto {Tr}^a(\tilde{u})$ weakly in $L^2(W_0)$. Note that the Proposition states that $\te(u_\varepsilon)\wto \theta u^a$, so, in the next step, our goal will be to prove that $\tilde{u}=\theta u^a$ for some $u^a\in L^2(\Omega',\theta;H^1_{\braket{\xi}}(Y))$.

	\noindent\textbf{Step 2: } Let us see that $\widetilde{u}=\theta u^a$ for some $u^a\in L^2(\Omega',\theta;H^1_{\braket{\xi}}(Y))$.
	Define the following functional
	\begin{equation}
		\begin{aligned}
			& T:  && L^2(W) &&& \to &&&& \R \ \  \\
			& \ && \ \ \ \varphi &&& \mapsto &&&& \int_W \tilde{u}\varphi .
		\end{aligned}
	\end{equation}
	For any $\varphi \in L^2(W)$ we have
	\begin{equation}
		\label{eqn:wellpeq1}
		\abs{\int_{W} \widetilde{u} \varphi} = \lim_{\varepsilon\to 0}\abs{\int_{W} \te(u_\varepsilon)\varphi} \leq  \limsup_{\varepsilon\to 0} \left(\norm{\te(u_\varepsilon)}_{L^2(W)}\left(\int_{W} \chi_{\omega_\varepsilon} \varphi^2\right)^{1/2}\right)\myleq{\eqref{eqn:ass2}, \eqref{eqn:propL2}} C_1\left(\int_{W} \theta \varphi^2\right)^{1/2}
	\end{equation}
	Hence, $\abs{T(\varphi)}\leq C_1\norm{\varphi}_{L^2(W,\theta)}$ for every $\varphi\in L^2(W)$. As $L^2(W)$ is dense in $L^2(W,\theta)$ (see Lemma \ref{lemma:l2dense}),
	we can extend $T$ to $\tilde{T}\in (L^2(W,\theta))'$. Therefore, by the Riesz representation theorem, there exists $u^a\in L^2(W,\theta)$ satisfying $\norm{u^a}_{L^2(W,\theta)}\leq C_1$ such that, for every $\varphi\in L^2(W,\theta)$, we have $\tilde{T}(\varphi)=\int_W \theta u^a\varphi$. In particular, for any $\varphi \in L^2(W)$,
	\begin{equation}
		\int_W \widetilde{u}\varphi = T(\varphi) = \tilde{T}(\varphi) = \int_W \theta u^a \varphi,
	\end{equation}
	so $\widetilde{u}=\theta u^a$ for some $u^a\in L^2(W,\theta)$. One can repeat the argument with $\frac{\partial \widetilde{u}}{\partial y}$ to prove that $\frac{\partial \widetilde{u}}{\partial y}=\theta \frac{\partial u^a}{\partial y}$ and $\frac{\partial u^a}{\partial y}\in L^2(W,\theta)$ satisfying $\norm{\frac{\partial u^a}{\partial y}}_{L^2(W,\theta)}\leq C_1$. It is straightforward to prove that, for $\theta-$a.e. $x\in \Omega'$, $\nabla_{\xi} u^a(x)\equiv 0$ because, for any $\varphi\in C^\infty_c(\Omega')$ and $\psi\in C^\infty_c(Y)$,
	$$
	\int_{\Omega'}\theta \varphi \int_{Y}\frac{\partial \psi}{\partial \xi_i} u^a = \int_{W} \varphi\frac{\partial \psi}{\partial \xi_i} \tilde{u}=\int_{W}\varphi\psi\frac{\partial \widetilde{u}}{\partial \xi_i} = 0.
	$$
	Hence,  $u^a\in L^2(\Omega',\theta; H^1_{\braket{\xi}}(Y))$ and satisfy $\norm{u^a}_{L^2(\Omega',\theta;H^1_{\braket{\xi}}(Y))}\leq \norm{u^a}_{L^2(W,\theta)}+\norm{\frac{\partial u^a}{\partial y}}_{L^2(W,\theta)}\leq 2C_1$. In fact, by being more precise in the estimates, the constant $2$ can be removed, but we have not done it in order to keep the proof easier to read.
	
	\noindent \textbf{Step 3: }To finish the proof, we need to show that
	${Tr}^a_{\braket{\xi}}(u^a)(x)={Tr}^b(u^b)(x)$ for $\theta$-a.e. $x\in\Omega '$. As $\widetilde{u}=\theta u^a$, this is equivalent to proving
	${Tr}^a_{\braket{\xi}}(\widetilde{u})(x)=\theta(x){Tr}^b(u^b)(x)$ for a.e. $x\in\Omega '$,
	which is equivalent to proving that, for every $\psi\in C^\infty_c(\Omega')$,
	\begin{equation}
		\label{eqn:convergenceseq1}
		\int_{\Omega'}{Tr}^a_{\braket{\xi}}(\widetilde{u})(x) \psi(x) dx=\int_{\Omega'}\theta(x){Tr}^b(u^b)(x)\psi(x)dx.
	\end{equation}
	Let us prove \eqref{eqn:convergenceseq1}. First, as $Tr^b(u_\varepsilon^b)\to {Tr}^b(u^b)$ in $L^2(\Omega')$ and $\chi_{\omega_\varepsilon}\wsto \theta$ in $L^\infty(\Omega')$,  we have 
	\begin{equation}
		\label{eqn:convergenceseq2}
		\lim_{\varepsilon\to 0}\int_{\Omega'}\chi_{\omega_\varepsilon}(x){Tr}^b(u_\varepsilon^b)(x)\psi(x)dx=\int_{\Omega'}\theta(x){Tr}^b(u^b)(x)\psi(x)dx.
	\end{equation}
	On the other hand, as we have that $Tr^a(\te(u_\varepsilon))\wto {Tr}^a(\widetilde{u})$ weakly in $L^2(W_0)$,
	\begin{equation}
		\label{eqn:convergenceseq3}
		\lim_{\varepsilon\to 0}\int_{W^0} {Tr}^a(\te(u_\varepsilon))(x,\xi)\psi(x)d\xi dx=\int_{W^0}{Tr}^a(\widetilde{u})(x,\xi) \psi(x)d\xi dx
	\end{equation}
	and, as $\widetilde{u}\in L^2(\Omega',H^1_{\braket{\xi}}(Y))$ and $\abs{\omega}=1$, we have
	\begin{equation}
		\label{eqn:convergenceseq3b}
		\int_{W^0}{Tr}^a(\widetilde{u})(x,\xi) \psi(x)d\xi dx=\int_{\Omega'}{Tr}^a_{\braket{\xi}}(\widetilde{u})(x) \psi(x) dx.
	\end{equation} 	
	Hence, combining \eqref{eqn:convergenceseq2}, \eqref{eqn:convergenceseq3} and \eqref{eqn:convergenceseq3b}, we reduce the proof of \eqref{eqn:convergenceseq1} to proving
	\begin{equation}
		\label{eqn:convergenceseq1b}
		\lim_{\varepsilon\to 0}\int_{\Omega'}\chi_{\omega_\varepsilon}(x){Tr}^b(u_\varepsilon)(x)\psi(x)dx=\lim_{\varepsilon\to 0}\int_{W^0} {Tr}^a(\te(u_\varepsilon))(x,\xi)\psi(x)d\xi dx.
	\end{equation}
	Let us prove then \eqref{eqn:convergenceseq1b}. We expand the second term in \eqref{eqn:convergenceseq1b} using Proposition \ref{prop:tr} with $\varphi=u_\varepsilon$,
	\begin{equation}
		\label{eqn:convergenceseq5}
		\int_{W^0} {Tr}^a(\te(u_\varepsilon))(x,\xi)\psi(x)d\xi dx= \sum_{n=1}^{N_\varepsilon} \int_{\omega \times \omega_\varepsilon^n}{Tr}^b(u_\varepsilon)(\bar x_\varepsilon^n+l_\varepsilon^n \xi)\psi(x) d\xi dx.
	\end{equation}
	Now, from Remark \ref{rem:z} and the regularity of $\psi$, we have, for every $x\in \omega_\varepsilon^n$ and $\xi\in \omega$,
	\begin{equation}
		\label{eqn:convergenceseq6bis}
		\abs{\psi(\bar x^n_\varepsilon + l_\varepsilon^n \xi)-\psi(x)}\leq C\varepsilon
	\end{equation}
	where $C$ is independent of $x$, $\xi$ and $\varepsilon$. Therefore, using \eqref{eqn:convergenceseq6bis} as well as the change of variables $\widetilde{x}=\bar x^n_\varepsilon + l_\varepsilon^n \xi$,

	\begin{equation}
		\label{eqn:convergenceseq6}
		\begin{aligned}
			&\abs{ \int_{\omega \times \omega_\varepsilon^n}{Tr}^b(u_\varepsilon)(\bar x_\varepsilon^n+l_\varepsilon^n \xi)\psi(x) d\xi dx - \int_{\omega \times \omega_\varepsilon^n}{Tr}^b(u_\varepsilon)(\bar x_\varepsilon^n+l_\varepsilon^n \xi)\psi(\bar x_\varepsilon^n+l_\varepsilon^n \xi) d\xi dx}
			\\
			&\myleq{\eqref{eqn:convergenceseq6bis}}  C\varepsilon\int_{\omega \times \omega_\varepsilon^n}\abs{{Tr}^b(u_\varepsilon)(\bar x^n_\varepsilon + l_\varepsilon^n \xi)}d\xi dx
			= C\varepsilon\frac{\abs{\omega_\varepsilon^n}}{(l_\varepsilon^n)^{N}}\int_{ \omega_\varepsilon^n}\abs{{Tr}^b(u_\varepsilon)(\widetilde{x})}d\widetilde{x} \leq C\varepsilon \norm{{Tr}^b(u_\varepsilon)}_{L^1(\omega_\varepsilon^n)}
		\end{aligned}	 
	\end{equation}
	Now, using the same change of variables, we obtain
	\begin{equation}
		\label{eqn:convergenceseq7}
		\begin{aligned}
			\int_{\omega \times \omega_\varepsilon^n}{Tr}^b(u_\varepsilon)(\bar x_\varepsilon^n+l_\varepsilon^n \xi)\psi(\bar x_\varepsilon^n+l_\varepsilon^n \xi) d\xi dx & = \int_{\omega_\varepsilon^n \times \omega_\varepsilon^n}(l_\varepsilon^n)^{-N}{Tr}^b(u_\varepsilon)(\widetilde{x})\psi(\widetilde{x}) d\widetilde{x} dx \\
			& =\int_{\omega_\varepsilon^n}{Tr}^b(u^b)(\widetilde{x})\psi(\widetilde{x})d\widetilde{x}.
		\end{aligned}
	\end{equation}
	Hence, combining \eqref{eqn:convergenceseq5}, \eqref{eqn:convergenceseq6} and \eqref{eqn:convergenceseq7} we obtain

	\begin{equation}
		\label{eqn:convergenceseq9}
		\begin{aligned}
			&\abs{\int_{W^0} {Tr}^a(\te(u_\varepsilon))(x,\xi)\psi(x)d\xi dx-\sum_{n=1}^{N_\varepsilon}\int_{\omega_\varepsilon^n}{Tr}^b(u^b)(\widetilde{x})\psi(\widetilde{x})d\widetilde{x}}\leq 
			C\varepsilon \sum_{n=1}^{N_\varepsilon}\norm{{Tr}^b(u_\varepsilon)}_{L^1(\omega_\varepsilon^n)}
			\\
			= \ &C\varepsilon \norm{{Tr}^b(u_\varepsilon)}_{L^1(\omega_\varepsilon)}\leq C\varepsilon\abs{\omega_\varepsilon}\norm{{Tr}^b(u_\varepsilon)}_{L^2(\omega_\varepsilon)}\leq C\varepsilon \abs{\Omega'} \norm{u_\varepsilon}_{H^1(\Omega_\varepsilon)}\leq C\varepsilon.
		\end{aligned}
	\end{equation}
	or equivalently,
	\begin{equation}
		\abs{\int_{W^0} {Tr}^a(\te(u_\varepsilon))(x,\xi)\psi(x)d\xi dx-\int_{\Omega'}\chi_{\omega_\varepsilon}(x){Tr}^b(u^b)(x)\psi(x)dx}\leq 
		C\varepsilon .
	\end{equation}
	Therefore, by taking the limit when $\varepsilon\to 0$, we obtain \eqref{eqn:convergenceseq1b} and finish the proof.

\end{proof}

\subsection{Our test functions}
\label{sec:test}

Let us define the test functions we are going to use. As presented in the introduction, the limit problem is different in the upper and the lower part. In the lower part we simply have a limit problem in $\Omega^b$ while, in the upper part, we have the unfolded domain $W=\Omega'\times Y$. Our test functions will then have this behaviour too. 

\begin{wrapfigure}[14]{r}{0.3\textwidth}  
	\centering
	\vspace{-0.045\textheight}
	\includegraphics[width=0.3\textwidth]{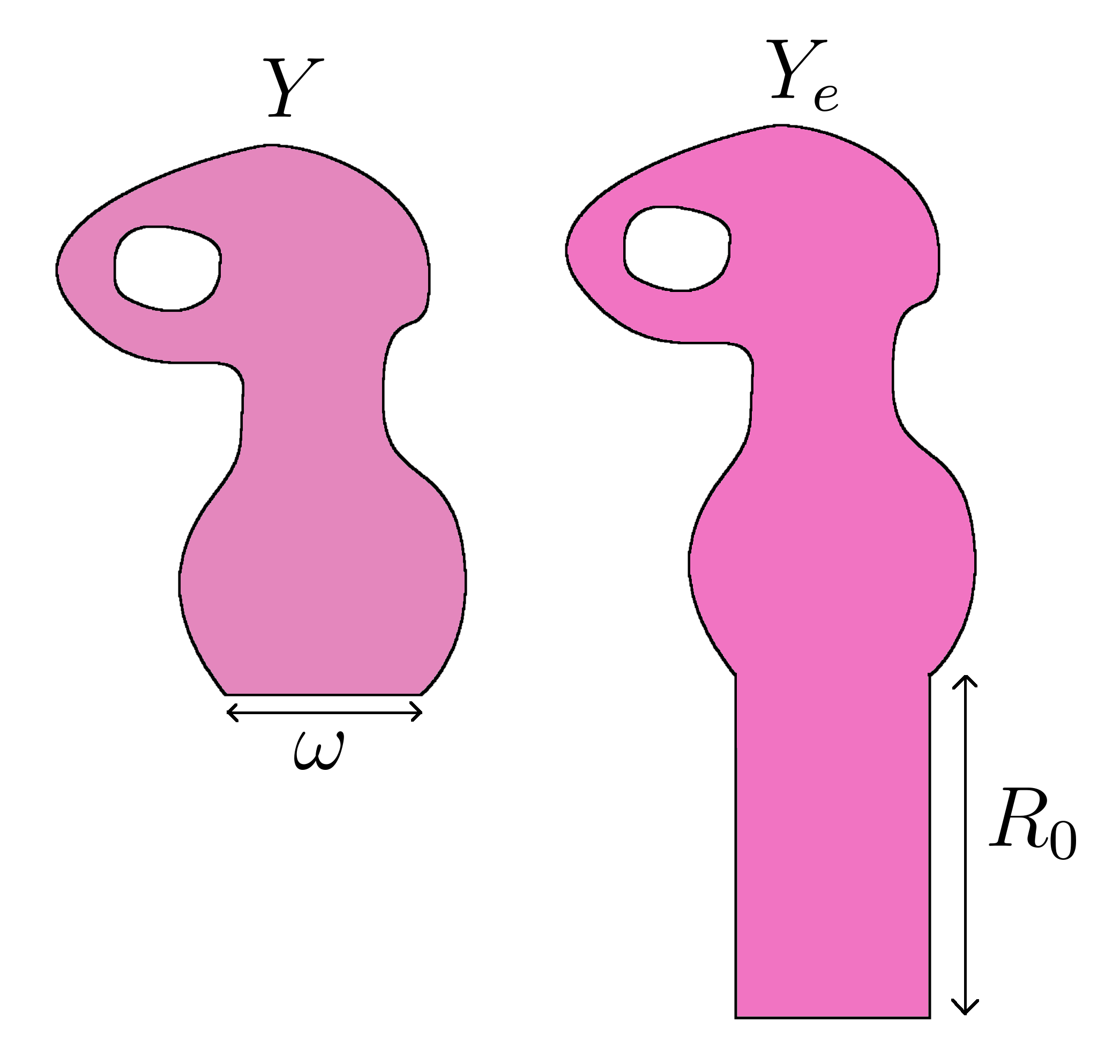}
	\caption{\small Example of an extended cell.}
	\label{fig:extendedcell}
\end{wrapfigure}

First of all, we define \textit{the extended cell}
$$
Y_e=Y\cup (\omega\times(-R_0,0])
$$
which is an open set due to assumption \eqref{eqn:delta0}. Recall that $R_0>0$ is such that $\Omega^b\subset B_{\R^{N+1}}(0,R_0)$. See Figure \ref{fig:extendedcell} for an example of an extended cell.

We need to consider test functions in $W$ as well as in $\Omega^b$, which are compatible. To do this, we will separate the dependence in $x$ from the dependence in $\xi$ and $y$. Our test functions will be of the form 
$$\phi(x)\psi(\xi,y)$$
with $\phi\in C^\infty(\RN)$ and $\psi\in H^1_{\braket{\xi}}(Y_e)$. Note that, by Proposition \ref{prop:constxi}, $\psi$ is then independent of $\xi$ when $y<0$ and there exists $\psi^b\in H^1((-R_0,0))$ such that 
\begin{equation}
	\label{eqn:psib}
	\psi^b(y)=\psi(\xi,y), \qquad \xi\in \omega, \ y\in (-R_0,0).
\end{equation} 
Analogously, we will denote as 
\begin{equation}
	\label{eqn:psia}
	\psi^a(\xi,y)=\psi(\xi,y), \qquad (\xi,y)\in Y,
\end{equation}
so $\psi^a\in H^1_{\braket{\xi}}(Y)$, although we will sometimes omit the super-index $a$ without risk of confusion, as it is just a restriction to a subdomain. Let us see how this function adapts to our domain $\Omega_\varepsilon$.

\begin{definition}[Adapted test function]
	\label{def:adapatedtest}
	Let $\phi \in C^\infty(\RN)$, $\psi\in H^1_{\braket{\xi}}(Y_e)$. Using the notation $\psi^a$ and $\psi^b$ from above, we define our $\varepsilon-$adapted test function:
	\begin{equation}
		U_\varepsilon(\phi,\psi)(x,y)= \left\{
		\begin{aligned}
			& \phi(x)\psi^a\left(\frac{x-\bar{x}^n_\varepsilon}{l^n_\varepsilon},y\right) \qquad & \text{if} \ \ & (x,y)\in Y^n_\varepsilon \\
			& \phi(x)\psi^b(y) \qquad & \text{if} \ \ & (x,y)\in \Omega^b
		\end{aligned}	
		\right.
	\end{equation}
\end{definition}
A consequence of the definition is that
\begin{equation}
	\te(U_\varepsilon(\phi,\psi))(x,\xi,y)=
	\left\{
	\begin{aligned}
		& \phi(\bar x_\varepsilon^n+l_\varepsilon^n \xi)\psi^a(\xi,y) \qquad  && (x,\xi,y)\in W, \ x\in \omega_\varepsilon^n \\
		& 0 \qquad && (x,\xi,y)\in W, \ x\notin \omega_\varepsilon.
	\end{aligned}
	\right.
\end{equation}

Now, let us see that $U_\varepsilon(\phi,\psi)$ is an adequate test function in \eqref{eqn:main}.
\begin{lemma}
	\label{lemma:adaptedtest}
	Let $\phi$ and $\psi$ be as in Definition \ref{def:adapatedtest} and denote
	$\varphi_\varepsilon= U_\varepsilon(\phi,\psi)$. Then $\varphi_\varepsilon\in H^1(\Omega_\varepsilon)$.
\end{lemma}
\begin{proof}
	By the chain rule, we trivially obtain $\varphi_\varepsilon$ belongs to  $H^1(\Omega_\varepsilon^a)$ and $H^1(\Omega_\varepsilon^b)$.
	Now, take $x\in \omega_\varepsilon^n$, and consider $r>0$ such that $B_{\R^{N+1}}((x,0),r)\subset \Omega_\varepsilon$. In particular, $B_{\RN}(x,r)\subset \omega_\varepsilon^n$. Then, we have that
	$$\varphi_\varepsilon(x,y)=\phi(x)\psi\left(\frac{x-\bar x_\varepsilon^n}{l^n_\varepsilon},y\right) \qquad (x,y)\in B_{\R^{N+1}}((x,0),r)$$
	because, when $y\leq 0$, $\varphi$ is independent of $\xi$. Therefore, $\varphi_\varepsilon\in H^1(B_{\R^{N+1}}((x,0),r))$. As $\omega_\varepsilon^n$ and $x\in \omega_\varepsilon^n$ were arbitrary, $\varphi\in H^1 (\Omega_\varepsilon)$.
\end{proof}

The following proposition shows how the test functions $ U_\varepsilon(\phi, \psi)$, constructed from $\phi$ and $\psi$, converge in the unfolded domain to the product $\phi \psi^a$ after applying the unfolding operator.
\begin{prop}
	\label{prop:test}
	Let $\phi$ and $\psi$ be as in Definition \ref{def:adapatedtest}. Let $\varphi_\varepsilon= U_\varepsilon((\phi,\psi))$. Then,
	\begin{equation}
		\label{eqn:testeq3}
		\lim_{\varepsilon\to 0}\int_{\Omega'}\norm{\te(\varphi_\varepsilon)(x)-\chi_{\omega_\varepsilon}(x)\phi(x)\psi^a}_{H^1(Y)}^2dx=0
	\end{equation}
\end{prop}
\begin{proof}
	In this proof we omit the super-index in $\psi^a$. We have
	\begin{equation}
		\label{eqn:testeq1}
		\te(\varphi_\varepsilon)(x,\xi,y)=\phi(\bar x_\varepsilon^n+l_\varepsilon^n \xi)\psi(\xi,y) \qquad \forall (x,\xi,y)\in W, \ x\in \omega_\varepsilon^n.
	\end{equation}
	By assumption \eqref{eqn:ass1}, $\abs{x-\bar x_\varepsilon^n-l_\varepsilon^n\xi}\leq C\varepsilon$, so we have
	\begin{equation}
		\label{eqn:testeq2}
		\abs{\phi(\bar x_\varepsilon^n+l_\varepsilon^n \xi)-\phi(x)}\leq C\varepsilon \norm{\phi}_{C^1(\adh{\Omega})}
	\end{equation}
	Combining \eqref{eqn:testeq1} and \eqref{eqn:testeq2}, we obtain	
	\begin{equation}
		\norm{\te(\varphi_\varepsilon)(x)-\phi(x)\psi}_{H^1(Y)}
		\leq C\varepsilon\norm{\phi}_{C^1(\adh{\Omega})}\norm{\psi}_{H^1(Y)}
	\end{equation}
	so, integrating in $x$ we obtain the desired result.
\end{proof}

The following corollary will be useful when passing to the limit in \eqref{eqn:main} with our test functions. We remind the reader of conventions in Section~\ref{sec:notations}, which allow us to interpret equations~\eqref{eqn:testeq4} and~\eqref{eqn:testeq5}. Note that, as $\phi \in C^\infty(\mathbb{R}^N)$ and $\psi \in H^1_{\braket{\xi}}(Y)$, it makes sense to consider $\phi$, $\psi$ or their product $\phi \psi$ as a functions $\mathbb{R}^N \times Y$, and therefore in $W \subset \mathbb{R}^N \times Y$.

\begin{corollary}
	\label{cor:test}
	Let $\phi$ and $\psi$ as in Definition \ref{def:adapatedtest}, and denote $\varphi_\varepsilon= U_\varepsilon((\phi,\psi))$. If we have $v_\varepsilon\wto v$ in $L^2(W)$ and $\chi_{\omega_\varepsilon} v_\varepsilon = v_\varepsilon$ then,
	\begin{equation}
		\label{eqn:testeq4}
		\int_{W} v_\varepsilon\te(\varphi_\varepsilon)\to \int_{W}v\phi\psi^a
	\end{equation}
	and
	\begin{equation}
		\label{eqn:testeq5}
		\int_{W} v_\varepsilon\te\left(\frac{\partial \varphi_\varepsilon}{\partial y}\right)\to \int_{W}v\phi\frac{\partial \psi^a}{\partial y}
	\end{equation}
\end{corollary}
\begin{proof}
	In this proof we omit the super-index in $\psi^a$. By the weak convergence of $v_\varepsilon$
	\begin{equation}
		\label{eqn:cortesteq1}
		\int_{W} (v_\varepsilon-v)\phi\psi\to 0
	\end{equation}
	when $\varepsilon\to 0$. In addition, by using Proposition \ref{prop:test},
	\begin{equation}	
		\label{eqn:cortesteq2}
		\begin{aligned}
			\abs{\int_{W}v_\varepsilon\te(\varphi_\varepsilon)-v_\varepsilon\phi\psi}&=\abs{\int_{W}v_\varepsilon\te(\varphi_\varepsilon)-v_\varepsilon\chi_\varepsilon\phi\psi}\leq
			\int_{W}\abs{v_\varepsilon}\abs{\te(\varphi_\varepsilon)-\chi_{\omega_\varepsilon}\phi\psi} \\
			&\leq \norm{v_\varepsilon}_{L^2(W)}\norm{\te(\varphi_\varepsilon)-\phi\psi\chi_{\omega_\varepsilon}}_{L^2(W)}\leq C\norm{\te(\varphi_\varepsilon)-\phi\psi\chi_{\omega_\varepsilon}}_{L^2(W)}\to 0
		\end{aligned}
	\end{equation}
	when $\varepsilon\to 0$. Hence, combining \eqref{eqn:cortesteq1} and \eqref{eqn:cortesteq2},
	\begin{equation}
		\begin{aligned}
			\abs{\int_{W}v_\varepsilon\te(\varphi_\varepsilon)-v\phi\psi}\leq 
			\abs{\int_{W}v_\varepsilon\te(\varphi_\varepsilon)-v_\varepsilon\phi\psi}	+\abs{\int_{W}(v_\varepsilon-v)\phi\psi}\to 0
		\end{aligned}
	\end{equation}
	The proof for the $y$-derivative is analogous.
\end{proof}
\subsection{Limit problem}
\label{sec:limit}

The following proposition is the first result where convergence to the limiting unfolded problem is obtained. The weak convergences established here will later be improved in order to prove the main Theorem \ref{thm:main}.
\begin{prop}
	\label{prop:main}
	Let $u_\varepsilon$ be the solution of \eqref{eqn:main}. Then, there exists $(u^a,u^b)\in H(\theta)$ such that $u_\varepsilon \to u^b$ strongly in $L^2(\Omega^b)$ and weakly in $H^1(\Omega^b)$, and $\te(u_\varepsilon)\wto \theta u^a$ weakly in $L^2(\Omega';H^1(Y))$. In addition, $(u^a,u^b)$ is the unique solution of
	\begin{equation}
		\label{eqn:limpro}
		\int_{\Omega^b}\left(\nabla u^b \nabla\varphi^b+u^b\varphi^b \right)+\int_{W}\theta\left( \frac{\partial u^a}{\partial y}\frac{\partial \varphi^a}{\partial y}+u^a \varphi^a\right)=\int_{\Omega^b}f\varphi^b+\int_{W}\theta f^*\varphi^a \qquad \forall (\varphi^a,\varphi^b)\in H(\theta)
	\end{equation}
	where $f^*(x,\xi,y)=f(x,y)$.
\end{prop}

\begin{proof}
	The proof of this result follows the method briefly outlined in Section \ref{sec:jtxoverview}. To facilitate the presentation, we divide the proof into a series of steps.
	
	\textbf{Step 1 (Convergence): }
	By using $u_\varepsilon\in H^1(\Omega_\varepsilon)$ as a test function in \eqref{eqn:main}, we obtain
	\begin{equation}
		\norm{u_\varepsilon}^2_{H^1(\Omega_\varepsilon)}=\int_{\Omega_\varepsilon}\abs{\nabla u_\varepsilon}^2 + \int_{\Omega_\varepsilon}\abs{u_\varepsilon}^2 \myeq{\eqref{eqn:main}} \int_{\Omega_\varepsilon} fu_\varepsilon  \leq \norm{f}_{L^2(\R^{N+1})}\norm{u_\varepsilon}_{H^1(\Omega_\varepsilon)}.
	\end{equation}
	Hence, we have the a priori estimate $\norm{u_\varepsilon}_{H^1(\Omega_\varepsilon)}\leq \norm{f}_{L^2(\R^{N+1})}$,
	which in particular implies by Proposition \ref{prop:prop}
	\begin{equation}
		\label{eqn:propmaineq8}
		\norm{u_\varepsilon}_{H^1(\Omega^b)}\leq \norm{f}_{L^2(\R^{N+1})}, \qquad \norm{\te(u_\varepsilon)}_{L^2(\Omega';H^1(Y))}\leq  C\norm{f}_{L^2(\R^{N+1})}.
	\end{equation}
	By Proposition \ref{prop:convergences}, there exists $(u^a, u^b)\in H(\theta)$ satisfying
	\begin{equation}
		\label{eqn:propmaineq7}
		\norm{u^b}_{H^1(\Omega^b)}\leq \norm{f}_{L^2(\R^{N+1})}, \qquad \norm{u^a}_{L^2(\Omega',\theta;H^1_{\braket{\xi}}(Y))}\leq C\norm{f}_{L^2(\R^{N+1})}
	\end{equation}
	such that, through a subsequence of $\varepsilon\to 0$, 
	\begin{equation}
		\label{eqn:propmaineq9}
		u_\varepsilon\to u^b \quad s-L^2(\Omega^b), \qquad u_\varepsilon\wto u^b \quad w-H^1(\Omega^b), \qquad \te(u_\varepsilon)\wto \theta u^a \quad w-L^2(\Omega';H^1(Y)).
	\end{equation}
	
	\noindent\textbf{Step 2 (Weak Formulation): }Now, let $\phi\in C^\infty_c(\RN)$ and $\psi\in H^1_{\braket{\xi}}(Y_e)$ and consider $\varphi_\varepsilon=U_\varepsilon(\phi,\psi)$ as in Definition \ref{def:adapatedtest}. By Lemma \ref{lemma:adaptedtest}, $\varphi_\varepsilon\in H^1(\Omega_\varepsilon)$, so we can use it as a test function in the variational formulation \eqref{eqn:main} to obtain
	\begin{equation}
		\int_{\Omega_\varepsilon} \nabla u_\varepsilon \nabla \varphi_\varepsilon+\int_{\Omega_\varepsilon}u_\varepsilon \varphi_\varepsilon = \int_{\Omega_\varepsilon} \varphi_\varepsilon f
	\end{equation}
	Now, we split the integral in $\Omega^b$ and $\Omega_\varepsilon^a$, apply the unfolding operator on the integrand of the $\Omega_\varepsilon^a$ term and obtain, using Proposition \ref{prop:prop},
	\begin{equation}
		\label{eqn:propmaineq1}
		\begin{aligned}
			\int_{\Omega^b}\left(\nabla u_\varepsilon \nabla \varphi_\varepsilon +u_\varepsilon\varphi_\varepsilon\right)+\int_{W}\left(\te(\nabla_x u_\varepsilon)\te(\nabla_x \varphi_\varepsilon)+\te(\frac{\partial u_\varepsilon}{\partial y})\te(\frac{\partial \varphi_\varepsilon}{\partial y})+\te(u_\varepsilon)\te(\varphi_\varepsilon)\right) \\
			= \int_{\Omega^b}\varphi_\varepsilon f + \int_{W}\te(\varphi_\varepsilon) \te(f)
		\end{aligned}
	\end{equation}
	
	\noindent\textbf{Step 3 (Taking limits): }Let us study each integral term in \eqref{eqn:propmaineq1}. First, using the convergences of \eqref{eqn:propmaineq9}, we have
	\begin{equation}
		\label{eqn:propmaineq2}
		\int_{\Omega^b}\left(\nabla u_\varepsilon \nabla \varphi_\varepsilon +u_\varepsilon\varphi_\varepsilon\right)=\int_{\Omega^b}\left(\nabla u_\varepsilon \nabla (\phi\psi^b) +u_\varepsilon\phi\psi^b\right)\to \int_{\Omega^b}\left(\nabla u^b \nabla(\phi\psi^b)+u^b\phi\psi^b \right)
	\end{equation}
	when $\varepsilon\to 0$. In addition, due to \eqref{eqn:propdery} and \eqref{eqn:propmaineq9}, $\te(\frac{\partial u_\varepsilon}{\partial y})=\frac{\partial\te( u_\varepsilon)}{\partial y}\wto \theta\frac{\partial u^a}{\partial y}$ in $L^2(W)$ and $\te(\frac{\partial u_\varepsilon}{\partial y})=\te(\frac{\partial u_\varepsilon}{\partial y})\chi_{\omega_\varepsilon}$. Hence, we can use Corollary \ref{cor:test} so that
	\begin{equation}
		\label{eqn:propmaineq3}
		\int_{W}\te\left(\frac{\partial u_\varepsilon}{\partial y}\right)\te\left(\frac{\partial \varphi_\varepsilon}{\partial y}\right)\to \int_{W}\theta\frac{\partial u^a}{\partial y} \phi\frac{\partial\psi^a}{\partial y}
	\end{equation}
	The same argument is applied to $\te(u_\varepsilon)$ and we obtain
	\begin{equation}
		\label{eqn:propmaineq4}
		\int_{W}\te( u_\varepsilon)\te(\varphi_\varepsilon)\to \int_{W}\theta u^a\phi\psi^a
	\end{equation}
	About the term
	\begin{equation}
		\int_{W}\te(\nabla_{x} u_\varepsilon)\te(\nabla_x \varphi_\varepsilon),
	\end{equation}
	it can be neglected when $\varepsilon\to 0$. The proof is a bit tedious, so we add it later in Lemma \ref{lemma:derivativex}.
	
	Finally, let us study the source term. We have, due to Lemma \ref{lemma:fconv}, that
	\begin{equation}
		\lim_{\varepsilon\to 0}\int_{W}\abs{\te\left(f\right)-f^*\chi_{\omega_\varepsilon}}^2=0
	\end{equation}
	where $f^*(x,\xi,y)=f(x,y)$.
	Furthermore, $f^*\chi_{\omega_\varepsilon}\wto \theta f^*$ weakly in $L^2(W)$ so $\te(f)\wto \theta f$ weakly in $L^2(W)$. Hence, we can use Corollary \ref{cor:test} and obtain
	\begin{equation}
		\label{eqn:propmaineq5}
		\int_{W}\te(f)\te(\varphi)\to \int_{W} f^*\theta \phi\psi.
	\end{equation}
	
	Combining \eqref{eqn:propmaineq1}, \eqref{eqn:propmaineq2}, \eqref{eqn:propmaineq3}, \eqref{eqn:propmaineq4}, \eqref{eqn:propmaineq5} and Lemma \ref{lemma:derivativex}, we obtain
	\begin{equation}
		\label{eqn:limiteq1}
		\int_{\Omega^b}\left(\nabla u^b \nabla(\phi\psi^b)+u^b\phi\psi^b \right)+\int_{W}\theta\left( \frac{\partial u^a}{\partial y}\frac{\partial (\phi\psi^a)}{\partial y}+u^a\phi\psi\right)=\int_{\Omega^b}f\phi\psi^b+\int_{W}\theta f^*\phi\psi^a.
	\end{equation}
	\noindent\textbf{Step 4 (Limit problem): }
	Let us obtain \eqref{eqn:limpro} from \eqref{eqn:limiteq1}. As \eqref{eqn:limiteq1} is linear in $\phi\psi$, for any linear combination of such functions, $\sum_{i=1}^k \phi_i\psi_i$ we obtain
	\begin{equation}
		\label{eqn:limiteq1b}
		\sum_{i=1}^k\left(\int_{\Omega^b}\left(\nabla u^b \nabla(\phi_i\psi_i^b)+u^b\phi_i\psi_i^b-f\phi_i\psi_i^b \right)+\int_{W}\theta\left( \frac{\partial u^a}{\partial y}\frac{\partial (\phi_i\psi_i^a)}{\partial y}+u^a\phi_i\psi_i-f^*\phi_i\psi_i^a\right)\right)=0.
	\end{equation}
	It turns out that the linear combination of functions of the form $\phi_i\psi_i$ is dense in $H(\theta)$. This is proved in Appendix \ref{app:density}. Then, let $(\varphi^a,\varphi^b)\in H(\theta)$. By Corollary \ref{cor:density}, we can find a finite number of functions $\{\phi_i\}_{i=1}^k\subset C^\infty_c(\RN)$ and $\{\psi_i\}_{i=1}^k\subset H^1_{\braket{\xi}}(Y_e)$ such that
	\begin{equation}
		\label{eqn:propmaineq6}
		\bnorm{\varphi^a-\sum_{i=1}^k \phi_i\cdot \psi^a_i}_{L^2(\Omega',\theta;H^1_{\braket{\xi}}(Y))}+\bnorm{\varphi^b-\sum_{i=1}^k \phi_i\cdot \psi^b_i}_{H^1(\Omega^b)}\leq \delta
	\end{equation}
	Then, by using Cauchy-Schwartz inequality, \eqref{eqn:propmaineq7} and \eqref{eqn:propmaineq6},
	\begin{equation}
		\label{eqn:limiteq8}
		\begin{aligned}
			\abs{\int_{\Omega^b}\left(\nabla u^b \nabla \varphi^b+u^b\varphi^b-f\varphi^b \right)-
				\sum_{i=1}^k\int_{\Omega^b}\left(\nabla u^b \nabla(\phi_i\psi_i^b)+u^b\phi_i\psi_i^b-f\phi_i\psi_i^b \right)	
			} 
			\\
			\leq \left(2\norm{u^b}_{H^1(\Omega^b)}+\norm{f}_{L^2(\Omega^b)}\right)\bnorm{\varphi^b-\sum_{i=1}^k \phi_i\cdot \psi^b_i}_{H^1(\Omega^b)}\leq C\delta \norm{f}_{L^2(\R^{N+1})}.
		\end{aligned}
	\end{equation}
	In an analogous way,
	\begin{equation}
		\label{eqn:limiteq9}
		\begin{aligned}
			\abs{\int_{W}\theta\left( \frac{\partial u^a}{\partial y}\frac{\partial \varphi^a}{\partial y}+u^a \varphi^a-f^*\varphi^a \right)-\sum_{i=1}^k\int_{W}\theta\left( \frac{\partial u^a}{\partial y}\frac{\partial (\phi_i\psi_i^a)}{\partial y}+u^a\phi_i\psi_i-f^*\phi_i\psi_i^a\right)}
			\\
			\leq
			\left(2\norm{u^a}_{L^2(\Omega',\theta;H^1_{\braket{\xi}}(Y))}+\norm{f^*}_{L^2(W,\theta)}\right)\bnorm{\varphi^b-\sum_{i=1}^k \phi_i\cdot \psi^b_i}_{L^2(\Omega',\theta;H^1_{\braket{\xi}}(Y))}\leq C\delta \norm{f}_{L^2(\R^{N+1})}.
		\end{aligned}
	\end{equation}
	Therefore, by combining \eqref{eqn:limiteq1b}, \eqref{eqn:limiteq8} and \eqref{eqn:limiteq9}, we obtain
	\begin{equation}
		\abs{
			\int_{\Omega^b}\left(\nabla u^b \nabla\varphi^b+u^b\varphi^b-f\varphi^b \right)+\int_{W}\theta\left( \frac{\partial u^a}{\partial y}\frac{\partial \varphi^a}{\partial y}+u^a \varphi^a-f^*\varphi^a\right)}\leq C\delta\norm{f}_{L^2(\R^{N+1})},
	\end{equation}
	which, as $\delta>0$ was arbitrary, leads to \eqref{eqn:limpro}.
	
	\noindent\textbf{Step 5 (Existence and uniqueness): }The existence and uniqueness of solutions therefore follows from Lax-Milgram theorem because 
	\begin{equation}
		\label{eqn:bilinear}
		a(u,\varphi)=\int_{\Omega^b}\left(\nabla u^b \nabla\varphi^b+u^b\varphi^b \right)+\int_{W}\theta\left(  \frac{\partial u^a}{\partial y}\frac{\partial \varphi^a}{\partial y}+u^a\varphi^a\right)
	\end{equation}
	is a bilinear coercive form in $H(\theta)$ and $\mathcal{L}_{f}(\varphi)=\int_{\Omega^b}f\varphi^b+\int_{W}\theta f^*\varphi^a$ is a linear form in $H(\theta)$. As $(u^a,u^b)\in H(\theta)$ is determined uniquely by \eqref{eqn:limpro}, the convergence of $u_\varepsilon$ is independent of the chosen subsequence of $\varepsilon>0$, so it is true for the whole sequence.
\end{proof}
Now, we have proved that our solutions converge to $(u^a,u^b)\in H(\theta)$. However, this convergence is not completely satisfactory as it is not strong in the upper part. We will improve this convergence in the following section. But first, let us improve the integrability of $u^a$. In principle, $u^a$ belongs to $L^2(\Omega', \theta;H^1_ {\left<\xi\right>}(Y))$, but we can improve this and show that, in fact, $u^a\in L^2(\Omega';H^1_ {\left<\xi\right>}(Y))$. However, note that,
although $u^a$, as a representative of its class of the Bochner space $L^2(\Omega',\theta;H^1_{\braket{\xi}}(Y))$, is defined in $\Theta_0$, its value there is not relevant because the class $u^a$ represents has not a well defined value in $\Theta_0$ because the weight $\theta$ vanishes. Therefore, we must redefine the value of $u^a$ in $\Theta_0$ to have an $L^2(\Omega';H^1_ {\left<\xi\right>}(Y))$ function.
\begin{prop}
	\label{prop:uaisnice}
	Let $u^a \in L^2(\Omega',\theta;H^1_ {\left<\xi\right>}(Y))$ be the function defined in Proposition \ref{prop:main}. If we redefine $u^a(x)=0$ for every $x\in \Theta_0$, then $u^a$ belongs to $L^2(\Omega';H^1_ {\left<\xi\right>}(Y))$.
\end{prop}
\begin{proof}
	First of all, since $(u^a,u^b)\in H(\theta)$, we have that ${Tr}^a_{\braket{\xi}}(u^a)(x)={Tr}^b(u^b)(x)$ $\theta$-a.e. $x\in \Omega'$. Hence, ${Tr}^a_{\braket{\xi}}(u^a)(x)={Tr}^b(u^b)(x)$ a.e. $x\in \Omega'\setminus \Theta_0$. As ${Tr}^b(u^b)\in L^2(\Omega')$, we have ${Tr}^a_{\braket{\xi}}(u^a)\in L^2(\Omega'\setminus \Theta_0)$, but, as $u^a(x)=0$ when $x\in \Theta_0$, we obtain ${Tr}^a_{\braket{\xi}}(u^a)\in L^2(\Omega')$. Therefore, if we define
	\begin{equation}
		\begin{aligned}
			&\adh{u}(x,\xi,y)\defeq u^a(x,\xi,y)-{Tr}^a_{\braket{\xi}}(u^a)(x) \ && a.e. \ (x,\xi,y)\in W, \\
			& \adh{f}(x,\xi,y)=f(x,y)-{Tr}^a_{\braket{\xi}}(u^a)(x) \ && a.e. \ (x,\xi,y)\in W,
		\end{aligned}
	\end{equation}
	we have $\adh{f}\in L^2(W)$ and the problem is reduced to proving that $\adh{u}$ and its derivatives are in $L^2(W)$ too.
	
	Now, let $0<\delta<1$ and define $\theta_\delta(x)\defeq \max(\delta,\theta(x))$. Consider $(\frac{\adh{u}^a}{\theta_\delta},0)\in H(\theta)$.
	Hence, we can use \eqref{eqn:limpro} with $(\frac{\adh{u}^a}{\theta_\delta},0)$ to obtain
	\begin{equation}
		\int_{W}\frac{\theta}{\theta_\delta} \left( \abs{\frac{\partial \adh{u}}{\partial y}}^2+\abs{\adh{u}}^2\right)=
		\int_{W}\frac{\theta}{\theta_\delta} \adh{f} \adh{u}\leq \left(\int_{W}\frac{\theta}{\theta_\delta}\abs{\adh{f}}^2\right)^{1/2} \left(\int_{W}\frac{\theta}{\theta_\delta}\abs{\adh{u}}^2\right)^{1/2}
	\end{equation}
	so then, as $\frac{\theta}{\theta_\delta}\leq 1$, we obtain
	\begin{equation}
		\bnorm{\sqrt{\frac{\theta}{\theta_\delta}}\adh{u}}_{L^2(\Omega';H^1_ {\left<\xi\right>}(Y))}\leq \norm{\adh{f}}_{L^2(W)}.
	\end{equation}
	Hence, as $\frac{\theta}{\theta_\delta}\adh{u}^2\to \adh{u}^2$ and $\frac{\theta}{\theta_\delta}\left(\frac{\partial\adh{u}}{\partial y}\right)^2\to \left(\frac{\partial\adh{u}}{\partial y}\right)^2$ pointwisely and monotonically when $\delta\to 0$, using the monotone convergence theorem we obtain $\norm{\adh{u}}_{L^2(\Omega';H^1_ {\left<\xi\right>}(Y))}\leq \norm{\adh{f}}_{L^2(W)}$.
\end{proof}
To finish the section, we state the auxiliary lemma we used in Proposition \ref{prop:main}.
\begin{lemma}
	\label{lemma:derivativex}
	Let $\phi\in C^\infty(\RN)$, $\psi\in H^1_{\braket{\xi}}(Y_e)$ and denote $\varphi_\varepsilon=U_\varepsilon(\phi,\psi)$ as in Definition \ref{def:adapatedtest}. Let $u_\varepsilon$ be the solutions of \eqref{eqn:main}. Then,
	\begin{equation}
		\int_{W}\te(\nabla_x u_\varepsilon)\te(\nabla_x \varphi_\varepsilon)\to 0
	\end{equation}
	when $\varepsilon\to 0$.
\end{lemma}

\begin{proof}	
	The idea of this proof is inspired by the techniques introduced in \cite{brizzichalot} and its use in \cite[Section 7]{gaudiello}.
	First of all, take $\eta\in C^\infty_c((0,\infty))$ and define
	\begin{equation}
		w_\varepsilon(x,y)=\left\{
		\begin{aligned}
			& \eta(y)\sum_{j=1}^N \frac{\partial \varphi_\varepsilon }{\partial x_j}(x,y)\cdot (x-\bar x_\varepsilon^n)_j \qquad && (x,y)\in Y_\varepsilon^n \\
			& 0 \qquad && y\leq 0
		\end{aligned}
		\right.
	\end{equation}
	where $(x-\bar x_\varepsilon^n)_j$ represents the $j$-th component of $(x-\bar x_\varepsilon^n)$. Then, $w_\varepsilon\in H^1(\Omega_\varepsilon)$ and its support is concentrated in $\Omega_\varepsilon^a$. We use then the weak formulation \eqref{eqn:main} with $w_\varepsilon$ and apply the unfolding operator
	\begin{equation}
		\label{eqn:derivativexeq2}
		\int_{W}\left(\te(\nabla_{x} u_\varepsilon)\te(\nabla_x w_\varepsilon)+\te(\frac{\partial u_\varepsilon}{\partial y})\te(\frac{\partial w_\varepsilon}{\partial y})+\te( u_\varepsilon)\te(w_\varepsilon)\right)
		= \int_{W}\te(w_\varepsilon)\te(f).
	\end{equation}	
	
	First of all, by \eqref{eqn:ass2}, $\abs{x-\bar x_\varepsilon^n}\leq C\varepsilon$ in $Y_\varepsilon^n$, so
	\begin{equation}
		\label{eqn:derivativexeq1}
		\norm{\te(w_\varepsilon)}_{L^2(W)}\leq C\varepsilon \sum_{j=1}^N \bnorm{\te\left(\frac{\partial \varphi_\varepsilon }{\partial x_j}\right)}_{L^2(W)}.
	\end{equation}
	Now, given $(x,y)\in Y_\varepsilon^n$, we have
	$
		\frac{\partial \varphi_\varepsilon }{\partial x_j}(x,y)=\frac{\partial \phi}{\partial x_j}(x)\psi\left(\frac{x-\bar x_\varepsilon^n}{l_\varepsilon^n},y\right),
	$
	just because $\nabla_{{\xi}}\psi\equiv 0$. Therefore, given $(x,{\xi},y)\in W$, $x\in \omega_\varepsilon^i$,
	\begin{equation}
		\label{eqn:derivativexeq5}
		\te\left(\frac{\partial \varphi_\varepsilon }{\partial x^j}\right)(x,{\xi},y)=\frac{\partial \phi}{\partial x_j}(\bar x_\varepsilon^n+l_\varepsilon^n \xi)\psi(\xi,y),
	\end{equation}
	so then $\bnorm{\te\left(\frac{\partial \varphi_\varepsilon }{\partial x^j}\right)}_{L^2(W)}\leq \norm{\phi}_{C^\infty(\adh{\Omega'})}\norm{\psi}_{L^2(Y_e)}$, and, by \eqref{eqn:derivativexeq1}, we obtain 
		$
		\te(w_\varepsilon)\to 0 \qquad \text{in } L^2(W)
		$
	when $\varepsilon\to 0$. In an analogous way we obtain 
	$
		\te\left(\frac{\partial w_\varepsilon}{\partial y}\right)\to 0 \qquad \text{in } L^2(W)
	$
	when $\varepsilon\to 0$. Then, as from \eqref{eqn:propmaineq8}, we have  $\norm{\te(u_\varepsilon)}_{L^2(\Omega';H^1(Y))}\leq C\norm{f}_{L^2(\R^{N+1})}$, taking the limit when $\varepsilon\to 0$ in \eqref{eqn:derivativexeq2}, we obtain
	\begin{equation}
		\label{eqn:derivativexeq3}
		\int_{W}\te(\nabla_{x} u_\varepsilon)\te(\nabla_x w_\varepsilon)\to 0.
	\end{equation}
	Now, let us study $\te(\nabla_{x}w_\varepsilon)$. For any $k=1, \dots, N$,
	\begin{equation}
		\frac{\partial w_\varepsilon}{\partial x_k}(x,y)=\eta(y)\frac{\partial \varphi_\varepsilon }{\partial x_k}(x,y)+\eta(y)\sum_{j=1}^N (x-\bar x_\varepsilon^n)_j\cdot \frac{\partial^2  \phi}{\partial x_j\partial x_k}(x)\psi\left(\frac{x-\bar x_\varepsilon^n}{l_\varepsilon^n},y\right).
	\end{equation}
	Hence,
	\begin{equation}
		\te\left(\frac{\partial w_\varepsilon}{\partial x_k}\right)(x,{\xi},y)=\eta(y)\te\left(\frac{\partial \varphi_\varepsilon }{\partial x_k}\right)(x,{\xi},y)+\eta(y)\sum_{j=1}^N l_\varepsilon^n {\xi}_j\frac{\partial^2  \phi}{\partial x_j\partial x_k}(\bar x_\varepsilon^n+l_\varepsilon^n\xi)\psi(\xi,y)
	\end{equation}
	Now, using $l_\varepsilon^n\leq C\varepsilon$ from assumption \eqref{eqn:ass1}, the fact that $\phi\in C^\infty(\adh{\Omega'})$ and $\psi\in H^1_ {\left<\xi\right>}(Y_e))$, we have
	\begin{equation}
		\label{eqn:derivativexeq4}
		\norm{\te(\nabla_x w_\varepsilon)-\eta\te(\nabla_x\varphi_\varepsilon )}_{L^2(W)}\leq C\varepsilon\norm{\phi}_{C^2(\adh{\Omega'})}\norm{\psi}_{H^1_ {\left<\xi\right>}(Y_e)}.
	\end{equation}
	Therefore, from \eqref{eqn:derivativexeq3} and \eqref{eqn:derivativexeq4} we obtain
	$
		\int_{W}\eta\te(\nabla_{x} u_\varepsilon)\te(\nabla_x \varphi_\varepsilon )\to 0
	$
	for any $\eta\in C^\infty_c(0,\infty)$. Now, take $0\leq \eta\leq 1$ such that $\eta(y)= 1$ for every $y\geq \delta$. Then
	\begin{equation}
		\begin{aligned}
			\bigg|\int_{W}\te(\nabla_{x} u_\varepsilon) & \te(\nabla_x \varphi_\varepsilon )\bigg|  \leq \abs{\int_{W}\eta \te(\nabla_{x} u_\varepsilon)\te(\nabla_x \varphi_\varepsilon )}+ \abs{\int_{W}(1-\eta)\te(\nabla_{x} u_\varepsilon)\te(\nabla_x \varphi_\varepsilon )}
			\\
			& \leq \abs{\int_{W}\eta \te(\nabla_{x} u_\varepsilon)\te(\nabla_x \varphi_\varepsilon )}+ \norm{\te(\nabla_{x} u_\varepsilon)}_{L^2(W)}\norm{\te(\nabla_x \varphi_\varepsilon )}_{L^2(W_\delta)}
		\end{aligned}
	\end{equation}
	where $W_\delta\defeq \{(x,{\xi},y)\in W: y\leq \delta\}$, so we just need to prove that $\norm{\te(\nabla_x \varphi_\varepsilon )}_{L^2(W_\delta)}\to 0$ when $\delta\to 0$ uniformly in $\varepsilon$ to conclude the proof. From \eqref{eqn:derivativexeq5}, we have that $\norm{\te(\nabla_x \varphi_\varepsilon )}_{L^2(W_\delta)}\leq \norm{\phi}_{C^\infty(\adh{\Omega'})}\norm{\psi}_{L^2(Y_\delta)}$ where $Y_\delta=\{(\xi,y)\in Y \ : \ y\leq \delta\}$ so, as $\norm{\psi}_{L^2(Y_\delta)}\to 0$ when $\delta\to 0$, we have $\norm{\te(\nabla_x \varphi_\varepsilon )}_{L^2(W_\delta)}\to 0$ when $\delta\to 0$ uniformly in $\varepsilon$ and we conclude the proof.
\end{proof}

\subsection{Energy and strong convergence}
\label{sec:energy}

With the previous result, Theorem \ref{thm:main}, we only get strong $L^2$ and weak $H^1$ convergence of solutions $u_\varepsilon$ in the lower part $\Omega^b$. In this section, we will improve the convergences in Proposition \ref{prop:main} to obtain strong $H^1(\Omega^b)$ convergence of solutions $u_\varepsilon$ in the lower part and strong $L^2(\Omega';H^1(Y))$ convergence of the unfolded solutions $\te(u_\varepsilon)$ in the unfolded domain. We will obtain these results via the energy estimates techniques of \cite{gaudiello}. 

In the following proposition, we define the notion of energy for the problem \eqref{eqn:main} and the limit problem \eqref{eqn:limpro}, and then prove the convergence when $\varepsilon\to 0$.
\begin{prop}[\textbf{Energy convergence}]
	\label{prop:convenergy}
	Let $u_\varepsilon$ be the solution of \eqref{eqn:main} and let $(u^a,u^b)\in H(\theta)$ be as in Proposition \ref{prop:main}. Define the following energy quantities,
	\begin{equation}
		E_\varepsilon(u_\varepsilon)\defeq \int_{\Omega^b}\left(\abs{\nabla u_\varepsilon}^2  +\abs{u_\varepsilon}^2\right)+\int_{W}\left(\abs{\te(\nabla_{x} u_\varepsilon)}^2+\abs{\te\left(\frac{\partial u_\varepsilon}{\partial y}\right)}^2+\abs{\te(u_\varepsilon)}^2\right)
	\end{equation}
	and
	\begin{equation}
		E(u^a,u^b)\defeq \int_{\Omega^b}\left(\abs{\nabla u^b}^2 + \abs{u^b}^2\right)+\int_{W}\theta\left(\abs{\frac{\partial u^a}{\partial y}}^2+\abs{u^a}^2\right).
	\end{equation}
	We have that
	\begin{equation}
		\lim_{\varepsilon\to 0}E_\varepsilon(u_\varepsilon)= E(u^a,u^b).
	\end{equation}
\end{prop}
\begin{proof}
	By using the weak formulation of our problem we have
	\begin{equation}
		E_\varepsilon(u_\varepsilon) = \int_{\Omega^b} u_\varepsilon f+\int_{W}\te(u_\varepsilon)\te(f)
	\end{equation}
	By Proposition \ref{prop:convergences}, $u_\varepsilon\to u^b$ strongly in $L^2(\Omega^b)$ and $\te(u_\varepsilon)\wto \theta u^a$ weakly in $L^2(W)$, by Lemma \ref{lemma:fconv}, $\int_{W}\abs{\te(f)-\chi_{\omega_\varepsilon} f^*}^2\to 0$ when $\varepsilon\to 0$ and by \eqref{eqn:ass2}, $\chi_{\omega_\varepsilon}f^*\wto \theta f^*$ in $L^2(\Omega)$, so we have
	\begin{equation}
		E_\varepsilon(u_\varepsilon)=\int_{\Omega^b} u_\varepsilon f+\int_{W}\te(u_\varepsilon)\te(f)\to \int_{\Omega^b} u^b f+\int_{W}\theta u^a f^*
	\end{equation}
	But, choosing $(\varphi^a,\varphi^b)=(u^a,u^b)$ in \eqref{eqn:limpro} we have 
	\begin{equation}
		\int_{\Omega^b} u^b f+\int_{W}\theta u^a f^*=E(u^a,u^b)
	\end{equation}
	so we are done.
\end{proof}
Now, using the convergence of the energies, we will prove that the convergences of Proposition \ref{prop:main} are strong.
\begin{theorem}
	\label{thm:strongconv}
	Let $u_\varepsilon$ be the solution of \eqref{eqn:main} and $(u^a,u^b)\in H(\theta)$ as in Proposition \ref{prop:convergences}. We have strong convergence,
	\begin{equation}
		\label{eqn:strongconveq1}
		\lim_{\varepsilon\to 0}\norm{\te(u_\varepsilon)-\chi_{\omega_\varepsilon} u^a}_{L^2(\Omega';H^1(Y))}=0, \qquad \lim_{\varepsilon\to 0}\norm{u_\varepsilon-u^b}_{H^1(\Omega^b)}\to 0.
	\end{equation}
	In particular, $\te\left(\nabla_{x} u_\varepsilon\right)\to 0$ in $L^2(W)$.
\end{theorem}
\begin{proof}
	First, we write \eqref{eqn:strongconveq1} in terms of the energies from Proposition \ref{prop:convenergy}.
	\begin{equation}
		\label{eqn:strongconveq6}
		\begin{aligned}
			& \int_{\Omega^b}\left(\abs{\nabla u_\varepsilon-\nabla u^b}^2  +\abs{u_\varepsilon-u^b}^2\right)+\int_{W}\left(\abs{\te(\nabla_{x} u_\varepsilon)}^2+\abs{\te(\frac{\partial u_\varepsilon}{\partial y})-\chi_{\omega_\varepsilon}\frac{\partial u^a}{\partial y}}^2+\abs{\te(u_\varepsilon)-\chi_{\omega_\varepsilon} u^a}^2\right)  \\
			& = E_\varepsilon(u_\varepsilon)+\adh{E}_\varepsilon(u^a,u^b)-2\left(\int_{\Omega^b}\left(\nabla u_\varepsilon\nabla u^b +u_\varepsilon u^b\right)+\int_{W}\left(\te(\frac{\partial u_\varepsilon}{\partial y})\frac{\partial u^a}{\partial y}+\te(u_\varepsilon) u^a\right)\right)\\
		\end{aligned}	
	\end{equation}
	where 
	\begin{equation}
		\label{eqn:strongconveq4}
		\adh{E}_\varepsilon(u^a,u^b)=\int_{\Omega^b}\left(\abs{\nabla u^b}^2 +\abs{u^b}^2\right)+\int_{W}\left(\abs{\frac{\partial u^a}{\partial y}}^2\chi_{\omega_\varepsilon}+\abs{u^a}^2\chi_{\omega_\varepsilon}\right)\to E(u^a, u^b)
	\end{equation}
	when $\varepsilon\to 0$ because $\chi_{\omega_\varepsilon}\wsto \theta$ in $L^\infty(\Omega')$. In addition, by the weak convergences of Proposition \ref{prop:convergences},
	\begin{equation}
		\label{eqn:strongconveq5}
		\begin{aligned}
			\int_{\Omega^b}\left(\nabla u_\varepsilon\nabla u^b +u_\varepsilon u^b\right)+\int_{W}\left(\te(\frac{\partial u_\varepsilon}{\partial y})\frac{\partial u^a}{\partial y}+\te(u_\varepsilon) u^a\right) \\ \to 
			\int_{\Omega^b}\left(\abs{\nabla u^b}^2 + \abs{u^b}^2\right)+\int_{W}\theta\left(\abs{\frac{\partial u^a}{\partial y}}^2+\abs{u^a}^2\right)=E(u^a,u^b)
		\end{aligned}
	\end{equation}
	Finally, using Proposition \ref{prop:convenergy}, we have $E_\varepsilon(u_\varepsilon)\to E(u^a,u^b)$, which, in combination with \eqref{eqn:strongconveq5}, \eqref{eqn:strongconveq4} and \eqref{eqn:strongconveq6}, leads to the desired result.
\end{proof}

Now that we have shown the strong convergence we can prove Theorem \ref{thm:main} as a corollary of the obtained results.\\

\begin{proof}[\textbf{Proof of Theorem \ref{thm:main}: }]
	By Proposition \ref{prop:main},
	there exist $(u^a,u^b)\in H(\theta)$ such that
	they are solutions of \eqref{eqn:thmmaineq3}, $u_\varepsilon\wto u^b$ weakly in $H^1(\Omega^b)$ and $\te(u_\varepsilon)\wto \theta u^a$ weakly in $L^2(W)$. By Theorem~\ref{thm:strongconv}, the convergence is improved so $u_\varepsilon\to u^b$ strongly in $H^1(\Omega^b)$. Hence, by definition of $\bar{u}_\varepsilon$ in $\Omega^b$, we obtain the right hand side convergence in \eqref{eqn:thmmaineq2}.
	
	Let us prove the left hand side convergence of \eqref{eqn:thmmaineq2}. By Proposition \ref{prop:uaisnice}, redefining $u^a(x)=0$ when $x\in \Theta_0$, we improve the integrability of $u^a$ so $u^a\in L^2(\Omega';H^1_{\braket{\xi}}(Y))$.
	By Theorem \ref{thm:strongconv} we improve the weak convergence of $\te(u_\varepsilon)$ to a strong convergence,
	\begin{equation}
		\label{eqn:mainproofeq0}
		\lim_{\varepsilon\to 0}\norm{\te(u_\varepsilon)-\chi_{\omega_\varepsilon} u^a}_{L^2(\Omega';H^1(Y))}=0.
	\end{equation}
	To obtain the convergence of \eqref{eqn:thmmaineq2} in $\Omega_\varepsilon^a$, note that
	$$
	\te(\bar{u}_\varepsilon)(x,{\xi},y)=\fint_{\omega_\varepsilon^n}u^a(z,{\xi},y)dz \qquad x\in \omega_\varepsilon^n, \ ({\xi},y)\in Y.
	$$
	Then,
	\begin{equation}
		\label{eqn:mainproofeq0bis}
		\begin{aligned}
			\norm{u_\varepsilon-\bar{u}_\varepsilon^a}^2_{L^2(\Omega_\varepsilon^a)} &\myeq{\eqref{eqn:propL2}} \norm{\te(u_\varepsilon)-\te(\bar{u}_\varepsilon^a)}^2_{L^2(W)}=\sum_{n=1}^{N_\varepsilon}\int_{\omega_\varepsilon^n}\int_Y\abs{u_\varepsilon(\bar{x}_\varepsilon^n+l_\varepsilon^n\xi,y)-\fint_{\omega_\varepsilon^n} u^a(z,\xi,y)dz}^2d\xi dy dx
			\\
			&=
			\sum_{n=1}^{N_\varepsilon}\abs{\omega_\varepsilon^n}\int_Y\abs{u_\varepsilon(\bar{x}_\varepsilon^n+l_\varepsilon^n\xi,y)-\fint_{\omega_\varepsilon^n} u^a(z,\xi,y)dz}^2d\xi dy
			\\
			&\leq 
			\sum_{n=1}^{N_\varepsilon}\abs{\omega_\varepsilon^n}\fint_{\omega_\varepsilon^n}\int_Y\abs{u_\varepsilon(\bar{x}_\varepsilon^n+l_\varepsilon^n\xi,y)- u^a(z,\xi,y)}^2d\xi dy dz 
			\\
			&=\sum_{n=1}^{N_\varepsilon}\int_{\omega_\varepsilon^n}\int_Y\abs{u_\varepsilon(\bar{x}_\varepsilon^n+l_\varepsilon^n\xi,y)- u^a(z,\xi,y)}^2d\xi dy dz  = \norm{\te(u_\varepsilon)-\chi_{\omega_\varepsilon}u^a}_{L^2(W)}
		\end{aligned}
	\end{equation}
	Therefore, combining \eqref{eqn:mainproofeq0} and \eqref{eqn:mainproofeq0bis} we obtain $\lim_{\varepsilon\to 0}\norm{u_\varepsilon-\bar{u}_\varepsilon^a}_{L^2(\Omega_\varepsilon^a)}$. In an analogous way, and using \eqref{eqn:propdery}, one obtains $\lim_{\varepsilon\to 0}\norm{\frac{\partial u_\varepsilon}{ \partial y}-\frac{\partial \bar{u}_\varepsilon^a}{\partial y}}_{L^2(\Omega_\varepsilon^a)}$. So, to finish the proof we only need to prove that $\lim_{\varepsilon\to 0}\norm{\nabla_{x} u_\varepsilon-\nabla_{x}\bar{u}_\varepsilon^a}_{L^2(W)}= 0$
	
	We have, due to Theorem \ref{thm:strongconv},
	\begin{equation}
		\label{eqn:mainproofeq5}
		\norm{\te\left(\nabla_{x} u_\varepsilon\right)}_{L^2(W)}\to 0,
	\end{equation}
	when $\varepsilon\to 0$. In addition, by the chain rule, for $(x,y)\in Y_\varepsilon^n$, we have
	\begin{equation}
		\label{eqn:mainproofeq6}
		\nabla_x \bar{u}_\varepsilon(x,y)=\frac{1}{l_\varepsilon^n}\left(\fint_{\omega_\varepsilon^n} \nabla_{\xi}u^a(z)dz\right)\left(\frac{x-{\bar x^n_\varepsilon}}{l_\varepsilon^n},y\right)=0,
	\end{equation}
	because $u^a(z)\in H^1_{\braket{\xi}}(Y)$ for every $z\in \Omega'$. Therefore,
	\begin{equation}
		\label{eqn:mainproofeq7}
		\lim_{\varepsilon\to 0}\norm{\nabla_{x}u_\varepsilon-\nabla_{x}\bar{u}_\varepsilon}_{L^2(\Omega_\varepsilon^a)} \myeq{\eqref{eqn:mainproofeq6}} \lim_{\varepsilon\to 0}\norm{\nabla_{x}u_\varepsilon}_{L^2(\Omega_\varepsilon^a)} \myeq{\eqref{eqn:propL2}} \lim_{\varepsilon\to 0}\norm{\te(\nabla_{x}u_\varepsilon)}_{L^2(W)} \myeq{\eqref{eqn:mainproofeq5}} 0.
	\end{equation}
	concluding the proof.
\end{proof}
The convergence of $u_\varepsilon$ in $\Omega^b$ can be improved if $f$ has better integrability properties and $\Omega^b$ is more regular. For example, if $f\in L^\infty(\R^{N+1})$ this is shown with bootstrapping regularity arguments in \cite{jtxtheta0}.

\section{The graph interpretation}
\label{sec:graph}
The solutions of the limit problem for the convergence Theorem \ref{thm:main} is split into two parts:  the lower problem, which is standard, and the convergence in the upper part, which requires to consider an auxiliary function constructed from the limit function $u^a$ defined in the unfolded domain $W$. This convergence may be non-intuitive, and in addition, it may not be clear which the classical interpretation of the limit problem \eqref{eqn:thmmaineq3} is. However, if our base teeth $Y$ is good enough, the limit problem has a beautiful interpretation in terms of graphs. This is the goal of the section, explaining this interpretation under reasonable assumptions. To make the graph interpretation more accessible, we will use images and specific examples to guide the reader. We will not focus on finding the optimal regularity conditions needed for a graph-based interpretation. Instead, we will simply adopt reasonable assumptions that allow us to obtain an interesting interpretation of the problem in terms of graph problems, without delving into overly technical details. This section is entirely dedicated to facilitate the understanding of the limit problem in terms of graphs and it is just connected to the rest of the paper via Theorem \ref{thm:graphprob}. Therefore, the reader can skip the whole section and still follow the main result of the paper. This section is strongly inspired by Section 6 of \cite{prizzi2001effect}.\\

First, we start defining what we call the projection in the $y$-axis.

\begin{wrapfigure}[14]{r}{0.4\textwidth}
	\centering
	\vspace{-1.4em}
	\includegraphics[width=0.4\textwidth]{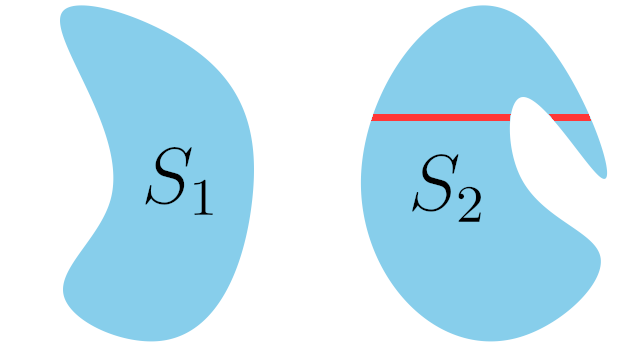}
	\caption{\small The above figure has two 2-dimensional domains. ${S}_1$ has connected horizontal sections, while ${S}_2$ does not. The red lines represent a horizontal section which is not connected.}
	\label{fig:domains}
\end{wrapfigure}
\begin{definition}
	\label{def:connhorsec}
	Given $S\subset \R^{N+1}$, we define $P_{S}:{S}\longrightarrow \R$ as $P_{S}((\xi,y))=y$, where $\xi\in \RN$ and $y\in \R$. In addition, we will say that ${S}$ has connected horizontal sections if for all $y\in \R$,  the section at height $y$, $P^{-1}_{S}(y)\subset \R^{N+1}$ is a connected set (which may be empty).

\end{definition}

Now, we define $p_{S}$ which is measuring the thickness of each connected horizontal section, which is the same as the one in \eqref{eqn:constxieq1}.
\begin{definition}
	Given ${S}\subset \R^{N+1}$ with connected horizontal sections, we define, for $y\in P_{S}({S})$,
	\begin{equation}
		\label{eqn:psigma}
		p_{S}(y)=\abs{\{\xi\in \RN: (\xi,y)\in P^{-1}_{S}(y)\}},
	\end{equation}
	where $\abs{\cdot}$ is the $N$-dimensional Lebesgue measure. The above expression is equivalent to the one used in Proposition \ref{prop:constxi},
	\begin{equation}
		\label{eqn:psigma2}
		p_{S}(y)=\abs{\{\xi\in \RN: (\xi,y)\in {S}\}}.
	\end{equation}
\end{definition}

\begin{prop}
	Let ${S}\subset \R^{N+1}$ be an open set with connected horizontal sections. Then,
	\begin{enumerate}
		\item $p_{S}$ is lower semicontinuous.
		\item $P_{S}({S})$ is an open bounded interval $(a_{S}^-,a_{S}^+)$ for certain $a_{S}^-,a_{S}^+\in \R$ with $a_{S}^-<a_{S}^+$.
	\end{enumerate}
\end{prop}
\begin{proof}
	(i) Let $P_{S}(y_0)=l$. Given $\epsilon>0$, choose a compact set $K\subset \{\xi\in \RN: (\xi,y)\in P^{-1}_{S}(y)\}$ such that $\abs{K}\geq l-\epsilon$. Now, as $K\times\{y_0\}\subset {S}$ is a compact set, we can find $\delta>0$ such that $K\times (y_0-\delta,y_0+\delta)\subset {S}$. Therefore, for every $y\in (y_0-\delta,y_0+\delta)$ we have $p_{S}(y)>l-\epsilon$.\\
	
	\noindent (ii) $p_{S}$ can be extended to $\R$ with the same equation \eqref{eqn:psigma}, which is equivalent to extend by zero, and it is still a lower semicontinuous function. Then, due to the lower semicontinuity, $P_{S}({S})=p_{S}^{-1}((0,\infty))$ is an open set. In addition $P_{S}$ is a continuous function, so, as ${S}$ is connected, $P_{S}({S})$ is too.
\end{proof}
Regarding these domains with connected horizontal sections, it is worth recalling the results obtained in Proposition \ref{prop:constxi}, which we will use later. Adapted to the notation we are using in this section, the proposition states that, if we have ${S}\subset \R^{N+1}=\{(\xi,y):\xi\in\RN, \ y\in \R\}$ domain with connected horizontal sections and $u\in H^1_{\braket{\xi}}({S})$, then, there exists $v\in H^1((a_{S}^-, a_{S}^+), p_{S})$ such that
\begin{equation}
	\label{eqn:prizzieq1}
	u(\xi,y)=v(y), \quad \frac{\partial u}{\partial y}(\xi,y)=\frac{d v}{dy}(y) \qquad a.e. \ (\xi,y)\in {S}.
\end{equation}
In addition, 
\begin{equation}
	\label{eqn:prizzieq2}
	\norm{v}_{H^1((a_{S}^-, a_{S}^+), p_{S})}=\norm{u}_{H^1({S})}.
\end{equation}
Conversely, if $v\in H^1((a_{S}^-, a_{S}^+),p_{S})$, it induces the definition of the function $u\in H^1_{\braket{\xi}}({S})$ satisfying \eqref{eqn:prizzieq1} and \eqref{eqn:prizzieq2}.

Now, we define what we call a nicely decomposable $Y$-domain. This definition is inspired by Section 6 of \cite{prizzi2001effect}.
\begin{definition}
	\label{def:nicedomain}
	Let $Y\subset \R^{N+1}$ be a model tooth as in Section \ref{sec:desc}. We will say that $Y$ is nicely decomposed if there exist $0=a_0<a_1<\ldots<a_M$ such that
	\begin{enumerate}
		\item For every $i\in I=\{1, \ldots M\}$,  the set $Y\cap \left(\RN\times(a_{i-1},a_{i})\right)$ is made up by a finite (positive) number $m_i$ of disjoint connected components that we denote $Y_i^{1},Y_i^{2},\dots, Y_i^{m_i}$.
		\item $m_1=1$, that is, $Y_1^1=Y\cap \left(\RN\times(a_0,a_{1})\right)$ is connected.
		\item $Y_i^j$ has connected horizontal sections for $i\in I$, $j\in J_i=\{1,\ldots,m_i\}$. 
		\item For every $i\in I$, $j\in J_i$ and $y\in (a_i,a_{i+1})$, defining 
		\begin{equation}
			\label{eqn:pij}
			p_i^j(y)= \left\{
			\begin{aligned}
				& p_{Y_i^j}(y) = \abs{\{\xi\in \RN: (\xi,y)\in Y_i^j\}}, \qquad && y\in (a_{i-1},a_{i}), \\
				& \lim_{s\to a_{i}} p_{Y_i^j}(s), \qquad && y=a_{i}, \\
				& \lim_{s\to a_{i-1}} p_{Y_i^j}(s), \qquad && y=a_{i-1},
			\end{aligned}
			\right.
		\end{equation}
		we will assume that $p_i^j\in C([a_i,a_{i+1}])\cap C^{1}((a_i,a_{i+1}))$ and $p_i^j>0$. As a consequence,
		\begin{equation}
			\label{eqn:pijbound}
			p_i^j(y)>c_0, \qquad y\in (a_{i-1}, a_i).
		\end{equation}
		for some constant $c_0>0$ independent of $i\in I$ and $j\in J_i$.
	\end{enumerate}
\end{definition}
In Figure \ref{fig:example} we present a particular example of a 2-dimensional nicely decomposable $Y$-domain.

\begin{figure}[h]
	\centering
	\includegraphics[width=0.8\textwidth]{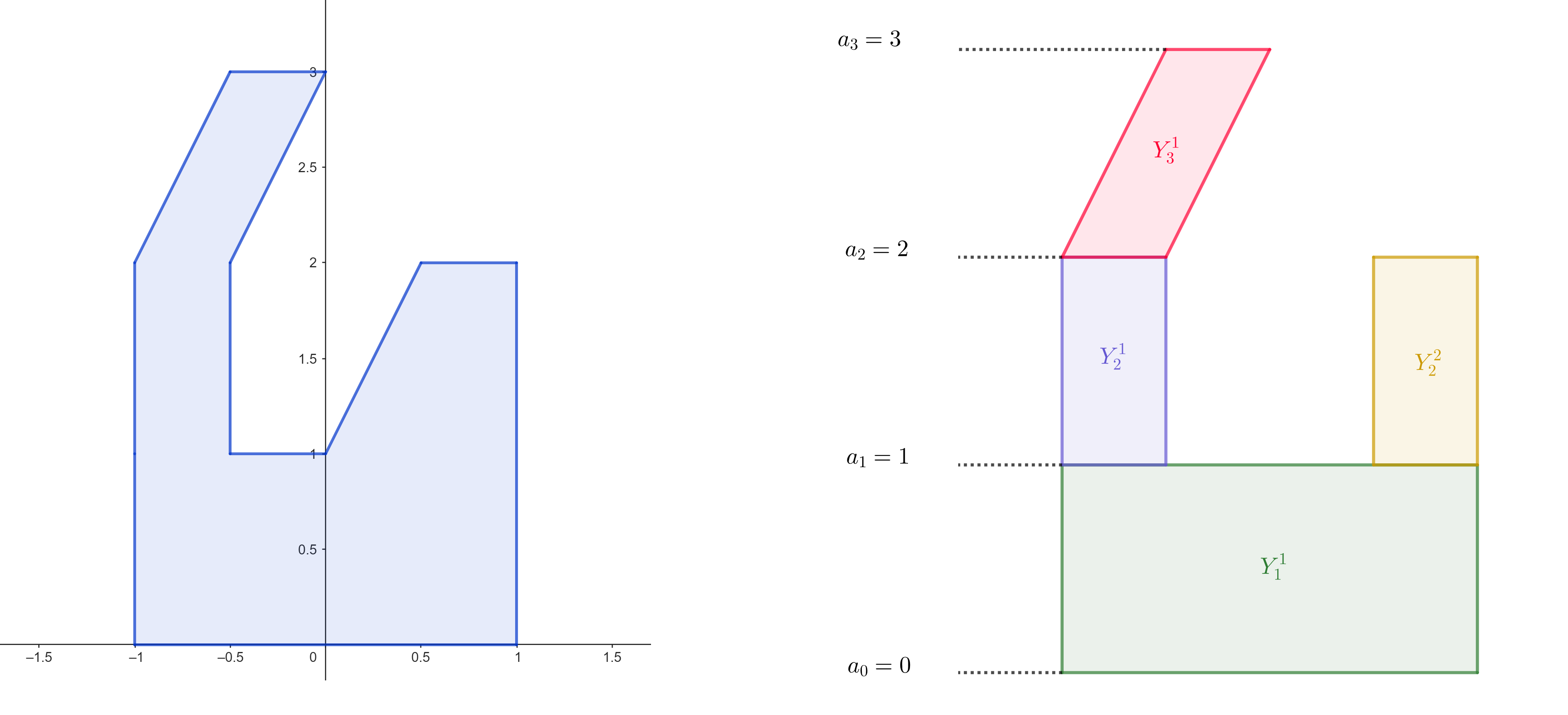}
	\caption{\small An example of a nicely decomposable domain along with a possible decomposition shown beside it. Note that each $Y_i^j$ has connected horizontal sections. We have $p_1^1 \equiv 2$, $p_2^1 \equiv 1/2$, $p_2^2(y) = (3 - y)/2$, and $p_3^1 \equiv 1$. Observe that $p_2^1(3) = 1$, not $0$, due to the definition in \eqref{eqn:pij}.}
	\label{fig:example}
\end{figure}

\par\medskip
An straightforward consequence of the first assumption is the following lemma.
\begin{lemma}
	\label{lemma:ballsjoint}
	Let $Y$ be as in Definition \ref{def:nicedomain}. Given a ball $B\subset Y$, $i\in I$ and $j,j'\in J_i$, if $B\cap Y_i^j$ and $B\cap Y_i^{j'}$ are non-empty, then $j=j'$.
\end{lemma}
\begin{proof}
	Let $B_i=B\cap \{(\xi,y): y\in (a_{i-1},a_i)\}$. $B_i$ is a subset of the union $\cup_{j\in J_i}Y_i^j$ of disjoint connected components. Therefore, by the connectivity of $B_i$, there exists $j\in J_i$ such that $B_i\subset Y_i^j$. As for $j'\neq j$ we have $Y_i^j\cap Y_i^{j'}= \emptyset$, the result follows.
\end{proof}

\begin{wrapfigure}[21]{r}{0.4\textwidth} 
	\centering
	\vspace{-0.04\textheight}
	\includegraphics[width=0.4\textwidth]{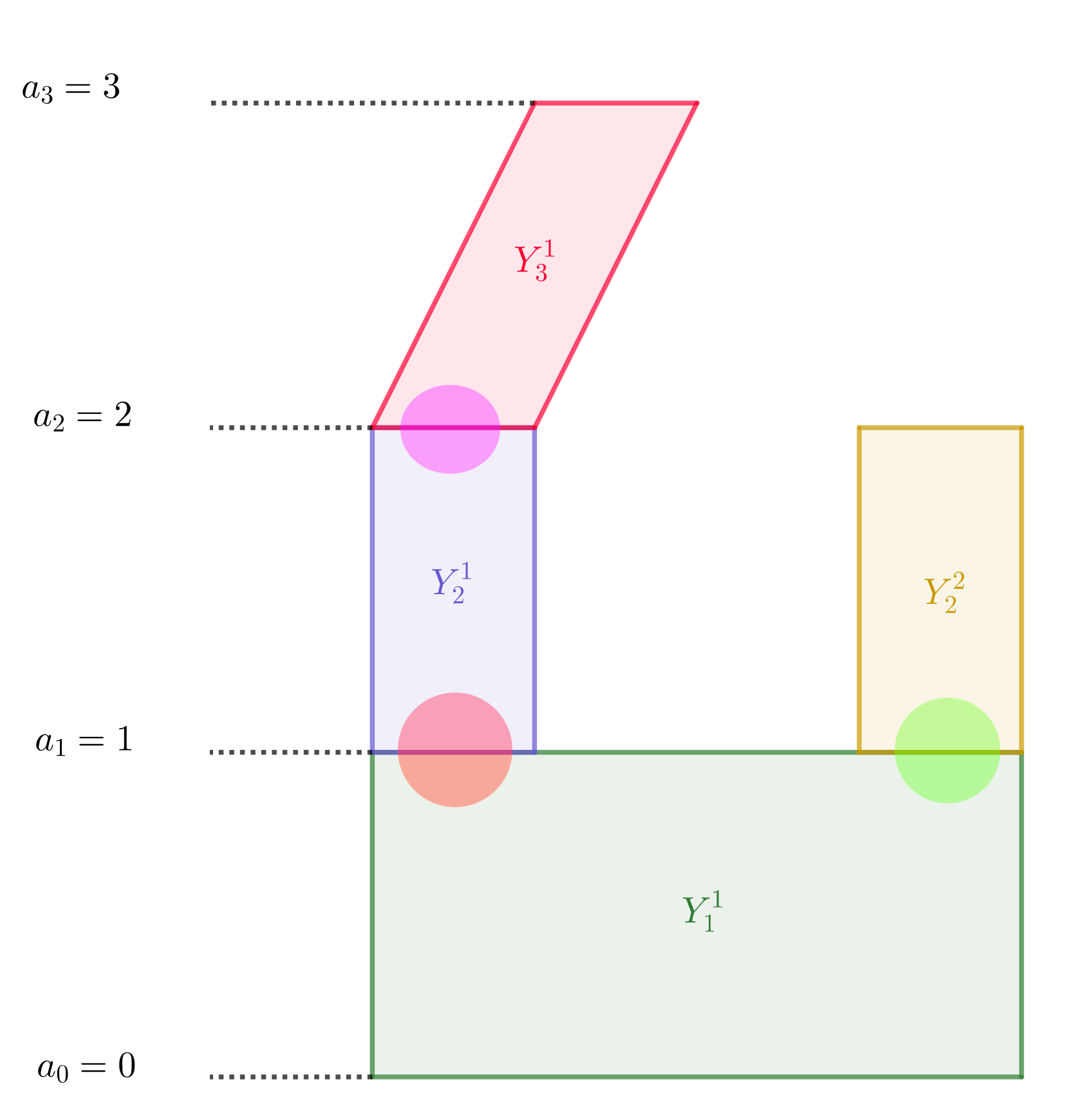}
	\caption{\small In the example shown in the figure, we can observe that $Y_1^1 \stackrel{i}{\sim} Y_2^1$, as indicated by the red ball; $Y_1^1 \stackrel{i}{\sim} Y_2^2$, as indicated by the green ball; and $Y_2^1 \stackrel{i}{\sim} Y_3^1$, as indicated by the pink ball.
	}
	\label{fig:example2}
\end{wrapfigure}

Now, we define what it means that $Y_{i}^j$ joins $Y_{i+1}^{j'}$.
\begin{definition}
	\label{def:join}
	Let $Y$ be as in Definition \ref{def:nicedomain}. Given $i\in I$, $j\in J_i$ and $j'\in J_{i+1}$, we will say that $Y_{i}^j$ joins $Y_{i+1}^{j'}$ at the $i$-th stage,
	and we will denote it like
	\begin{equation}
		Y_{i}^j \ \stackrel{i}{\sim} \ Y_{i+1}^{j'} \quad \text{or} \quad  Y_{i+1}^{j'} \ \stackrel{i}{\sim} \ Y_{i}^j.
	\end{equation}
	if there exists a ball $B\subset Y$ such that
	\begin{equation}
		\label{eqn:defjoin}
		B\cap Y_{i}^j\neq \emptyset \quad \text{and} \quad B\cap Y_{i+1}^{j'}\neq \emptyset.
	\end{equation}
\end{definition}
See Figure \ref{fig:example2} for a graphical interpretation of Definition \ref{def:join}.

Thanks to Lemma \ref{lemma:ballsjoint}, we can improve the conditions satisfied by the ball of the above definition.
\begin{lemma}
	\label{lemma:improvejoin}
	Let $Y$ be as in Definition \ref{def:nicedomain}. If, for some $i\in I$, $j\in I_j$, $j'\in J_{i+1}$, $Y_i^j \stackrel{i}{\sim} Y_{i+1}^{j'}$ in the sense of Definition \ref{def:join}, there exists $\xi_0\in \RN$ and a ball $B((\xi_0,a_i),r)\subset \mathscr{Y}_i=Y \cap \left(\RN\times (a_{i-1},a_{i+1})\right)$ which satisfies \eqref{eqn:defjoin} and
	\begin{equation}
		\label{eqn:improvejoineq1}
		B_i=\{(\xi,y)\in B: y<a_i\}\subset Y_i^j, \qquad B_{i+1}=\{(\xi,y)\in B: y>a_{i}\}\subset Y_{i+1}^{j'}.
	\end{equation}
\end{lemma}
\begin{proof}
	As $Y_i^j \stackrel{i}{\sim} Y_{i+1}^{j'}$, by Definition \ref{def:join}, there exists a ball $\tilde{B}\subset Y$ satisfying $\tilde{B}\cap Y_{i}^j\neq \emptyset$ and $\tilde{B}\cap Y_{i+1}^{j'}\neq \emptyset$. 
	Therefore, there exists $\xi_0\in \RN$ such that $(\xi_0,a_i)\in \tilde{B}$. Take then $B=B((\xi_0,a_i),r)$ and choose $r>0$ small enough such that $B\subset \{(\xi,y)\in \tilde{B}: a_{i-1}<y<a_{i+1}\}\subset \mathscr{Y}_i$. By Lemma \ref{lemma:ballsjoint}, for any $\tilde{j}\neq j$ we have $\tilde{B}\cap Y_i^{\tilde{j}}=\emptyset$. Therefore, since $B_i\subset\tilde{B}$, for any $\tilde{j}\neq j$, $B_i\cap Y_i^{\tilde{j}}=\emptyset$, so  $B_i\subset Y_i^j$ necessarily. In an analogous way, $B_{i+1}\subset Y_{i+1}^{j'}$.
\end{proof}

Now, we need to define the joints at the i-th stage.
\begin{definition}
	\label{def:joint}
	Let $Y$ be as in Definition \ref{def:nicedomain}. Given $i\in I$, $i\neq M$, consider the connected components of $\mathscr{Y}_i=Y \cap \left(\RN\times (a_{i-1},a_{i+1})\right)$ and denote them as $$T_i^{1},T_i^{2},\dots T_i^{r_i}\subset Y \cap \left(\RN\times (a_{i-1},a_{i+1})\right).$$ Then, given any $Y_{i}^{j}$ for some $j\in J_i$, there exists $k\in K_i=\{1,\dots, r_i\}$ such that $Y_{i}^{j}\subset T_i^k$. In the same way, given any $Y_{i+1}^{j}$ for some $j\in J_{i+1}$, there exists $k\in K_i$ such that $Y_{i+1}^{j}\subset T_i^k$. Therefore, given $k\in K_i$, we can define the set of indices
	\begin{equation}
		B_i^k\defeq \{j\in\{1,\dots,m_{i-1}\}: Y_{i}^j\subset T_i^k\}, \qquad A_i^k\defeq \{j\in\{1,\dots,m_{i}\}: Y_{i+1}^j\subset T_i^k\}.
	\end{equation}
	We will call them the set of indices of the \textbf{lower k-th joint at the i-th stage} and the set of indices of the \textbf{upper k-th joint at the i-th stage}. Note that $A$ and $B$ stand for \textit{above} and \textit{below}.
\end{definition}
In Figure \ref{fig:example3} we present the sets defined above for the example of Figure \ref{fig:example}.
\begin{figure}[H]
	\centering
	\includegraphics[width=1\textwidth]{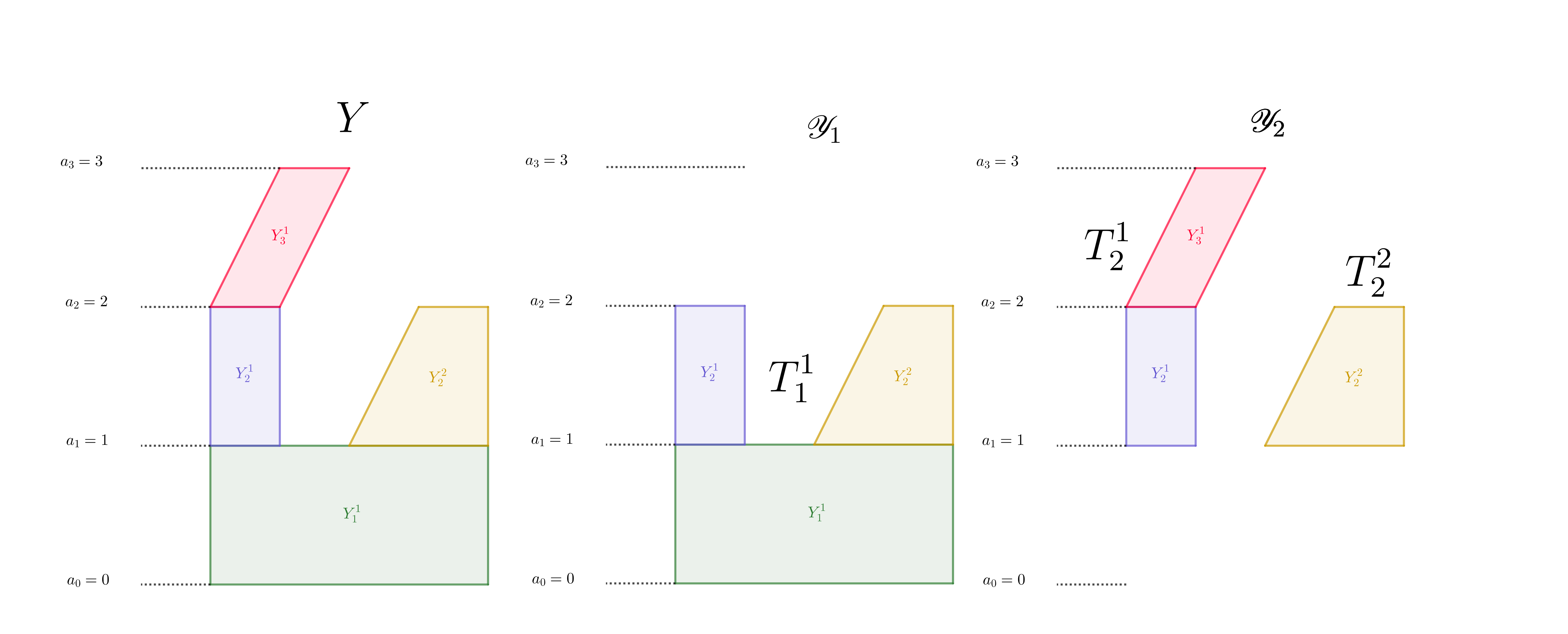}
	\caption{\small In the figure, the cell $Y$ from Figure \ref{fig:example} and its associated sets $\mathscr{Y}_1$ and $\mathscr{Y}_2$ are shown. As can be observed, $\mathscr{Y}_1$ consists of a single connected component $T_1^1$, while $\mathscr{Y}_2$ consists of two connected components, $T_2^1$ and $T_2^2$. Then, $B_1^1 = \{1\}$, $A_1^1 = \{1,2\}$, $B_2^1 = \{1\}$, $A_2^1 = \{1\}$, $B_2^2 = \{2\}$, and $A_2^2 = \emptyset$.
	}
	\label{fig:example3}
\end{figure}

The following lemma gives a relationship between Definitions \ref{def:join} and \ref{def:joint}.
\begin{lemma}
	\label{lemma:continuityjoints}
	Given $i\in I$ and $k\in K_i$, if $j\in B_i^k$ (index of the lower k-th joint at the i-th stage) and $j'\in A_i^k$ (index of the upper k-th joint at the i-th stage). Then, there exists a sequence of indices $j=j_1,\dots , j_l=j'$ such that
	\begin{equation}
		\label{eqn:continuityjointseq1}
		Y_{i}^{j_1} \ \stackrel{i}{\sim} \ Y_{i+1}^{j_2} \ \stackrel{i}{\sim} \ Y_{i}^{j_3} \ \stackrel{i}{\sim} \ Y_{i+1}^{{j}_{4}} \ \stackrel{i}{\sim} \ \dots  \stackrel{i}{\sim} Y_{i}^{j_{l-1}} \ \stackrel{i}{\sim} \ Y_{i+1}^{j_l}.
	\end{equation}
	The converse is also true, that is, if such a chain exists, then $j\in B_i^k$ and $j'\in A_i^k$ for the same $k$.
\end{lemma}

\begin{proof}

	As a reminder, recall that we denote $\mathscr{Y}_i=Y \cap \left(\RN\times (a_{i-1},a_{i+1})\right)$.
	
	\noindent \textbf{Step 1: (converse result)} Let us start by proving that, if \eqref{eqn:continuityjointseq1} is satisfied, then $j\in B_i^k$ and $j'\in A_i^k$. First of all, as $Y_{i}^{j_1} \stackrel{i}{\sim} Y_{i+1}^{j_2}$, we can use Lemma \ref{lemma:improvejoin} to find a ball $B\subset \mathscr{Y}_i$ such that $B\cap Y_i^j$ and $B\cap Y_{i+1}^{j_2}$ are non-empty.
	Then, there exists a continuous path in $\mathscr{Y}_i$ connecting a point in $Y_{i}^j$ to a point in $Y_{i+1}^{j_2}$. As both sets are path-connected, we can connect any point in $Y_{i}^j$ to any point in $Y_{i+1}^{j_2}$ by a continuous path on $\mathscr{Y}_i$. Repeating the argument, as $Y_{i+1}^{j_2} \ \stackrel{i}{\sim} \ Y_{i}^{j_3}$ ,we can connect any point in $Y_{i+1}^{j_2}$ to any point in $Y_{i+1}^{j_3}$ by a continuous path in $\mathscr{Y}_i$. Therefore, by transitivity, we can connect any point in $Y_{i}^{j_1}$ to any point in $Y_{i+1}^{j_3}$ by a continuous path in $\mathscr{Y}_i$. Repeating the argument through the chain, we obtain that we can connect any point in $Y_{i}^{j_1}$ to any point in $Y_{i+1}^{j_l}$ by a continuous path in $\mathscr{Y}_i$. This implies that $Y_{i}^{j}=Y_{i}^{j_1}$ and $Y_{i+1}^{j'}=Y_{i+1}^{j_l}$ belong to the same connected component $T_i^k$ of $\mathscr{Y}_i$ so $j\in B_i^k$ and $j'\in A_i^k$.
	
	\textbf{Step 2: (direct result)} Now we prove that, if $j\in B_i^k$ and $j'\in A_i^k$ then the chain \eqref{eqn:continuityjointseq1} exits. If $j\in B_i^k$ and $j'\in A_i^k$, this means that $Y_i^j$ and $Y_{i+1}^{j'}$ belong to the same connected component $T_i^k$ of $\mathscr{Y}_i$. Therefore, there exists a $C^\infty$ path connecting two points of $Y_i^j$ and $Y_{i+1}^{j'}$, that is, there exists $\gamma:[0,1] \to \mathscr{Y}_i$ such that $\gamma(0)\in Y_i^j$ and $\gamma(1)\in Y_{i+1}^{j'}$. Now, by the compactness of $\gamma([0,1])$ we can find $r>0$ such that $B(\gamma(s),r)\subset \mathscr{Y}_i$ for every $s\in [0,1]$. We also choose $r>0$ smaller if necessary in order to have that $B(\gamma(0),r)\subset Y_i^j$.
	Define then $s_1=0$ and
	\begin{equation}
		s_2 = \inf\{s\in [s_1,1]:B(\gamma(s),r)\cap Y_i^j=\emptyset\}.
	\end{equation}
	Now, let us prove that there exists $j_2$ such that 
	$$\gamma(s_2)\in Y_{i+1}^{j_2}.$$
	To do this, we will prove that $B(\gamma(s_2),r)\subset \{(\xi,y)\in Y: y>a_i\}$. As $B(\gamma(s_2),r)$ and $Y$ are open sets, this is equivalent to proving that $B(\gamma(s_2),r) \cap \{(\xi,y)\in Y: y<a_i\}=\emptyset$, which is equivalent to proving that $B(\gamma(s_2),r)\cap Y_i^{\tilde{j}}=\emptyset$ for every $\tilde{j}\in J_i$. Assume, by contradiction that $B(\gamma(s_2),r)\cap Y_i^{\tilde{j}}$ is non-empty. By the continuity of $\gamma$, there exists $\tilde{s}\in (s_1,s_2)$ such that $B(\gamma(\tilde{s}),r)\cap Y_i^{\tilde{j}}$ is non-empty. But, by using Lemma \ref{lemma:ballsjoint}, we obtain $\tilde{j}=j$, which is a contradiction because, by definition of $s_2$, $B(\gamma(s_2),r)\cap Y_i^{j}=\emptyset$. Then, we have that $\gamma(s_2)\in Y_{i+1}^{j_2}$. 
	
	Now, using the continuity of $\gamma$, we can find $\tilde{s}\in (s_1,s_2)$ such that $B(\gamma(\tilde{s}),r)\cap Y_{i+1}^{j_2}$ is non-empty. In addition, by the definition of $s_2$, we have $B(\gamma(\tilde{s}),r)\cap Y_{i}^{j}$. Therefore, 
	$$Y_{i}^{j} \stackrel{i}{\sim} Y_{i+1}^{j_2}.$$
	
	Repeating the argument, we can find $s_3>s_2$ and $j_3\in 1,\ldots, m_i$ such that $\gamma(s_3)\in Y_i^{j_2}$ and 
	$$Y_{i+1}^{j_2} \stackrel{i}{\sim} Y_i^{j_3}.$$
	Iterating, we generate the chain like \eqref{eqn:continuityjointseq1}. However, we must ensure that the process terminates in a finite number of steps. To do this, we will prove that, $s_{l+1}-s_{l}\geq r/\norm{\gamma'}_\infty$ for any index $l$, so, in a finite number of steps we will have $s_l=1$ and the process terminates. We prove the case $l=2$ as the other cases are analogous. We have 
	$B(\gamma(s_2),r)\subset \{(\xi,y)\in Y: y>a_i\}$. Then, choosing $\delta<r/\norm{\gamma'}_\infty$, we have
	$$
	\abs{\gamma(s_2+\delta)-\gamma(s_2)}\leq \delta \norm{\gamma'}_\infty \leq r
	$$
	so $\gamma(s_2+\delta)\in B(\gamma(s_2),r)\subset \{(\xi,y)\in Y: y>a_i\}$ so $\gamma(s_2+\delta)\not\in Y_i^{j_3}$, so $s_3>s_2+\frac{r}{\norm{\gamma'}_\infty}$ and we finish the proof.
\end{proof}
In Figure \ref{fig:example4} we present a graphical sketch of the proof of Lemma \ref{lemma:continuityjoints}.

\begin{figure}[H]
	\centering
	\includegraphics[width=0.5\textwidth]{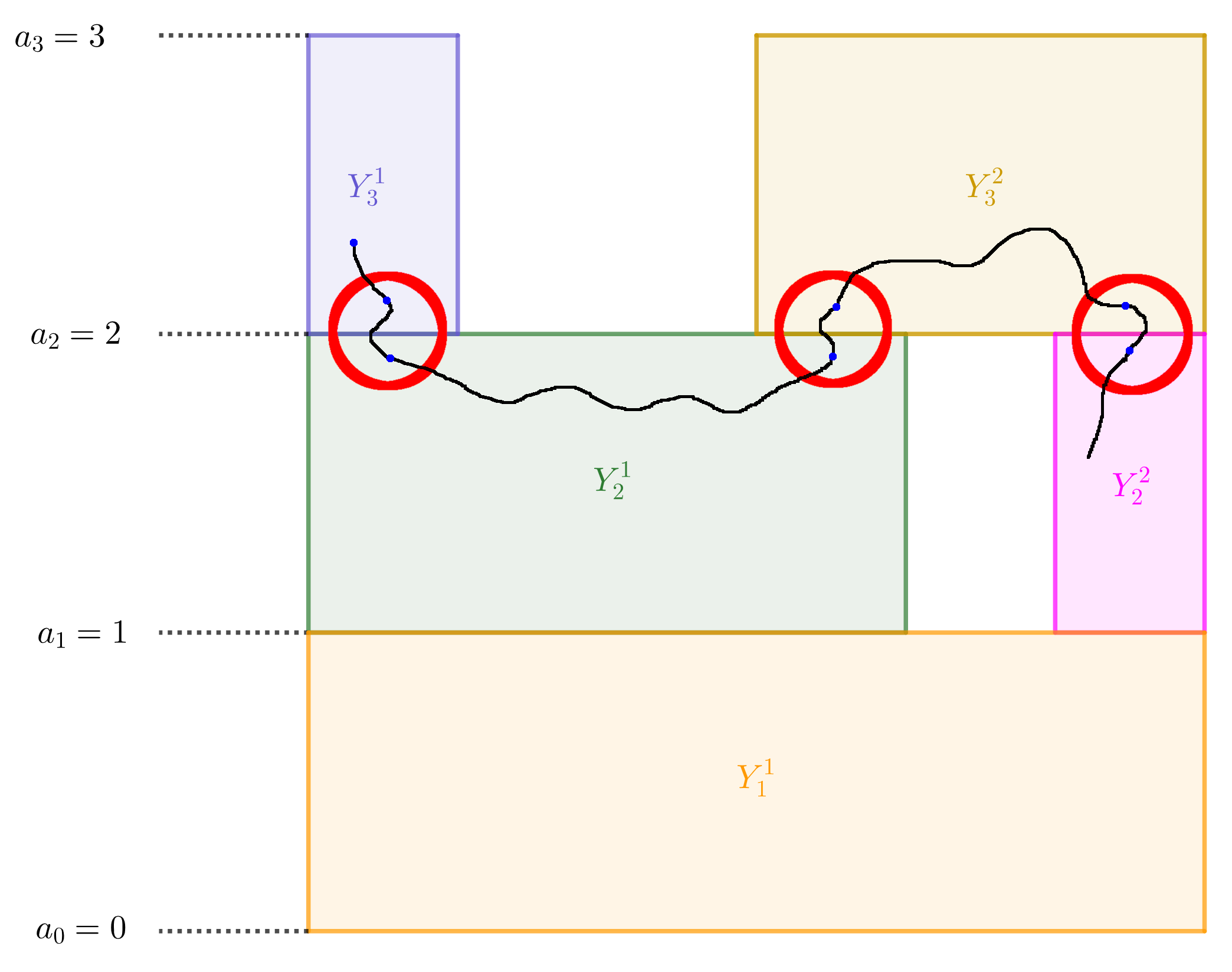}	
	\caption{\small The figure above represents the argument used in the proof of Lemma \ref{lemma:continuityjoints} for a particular example of $Y$. Specifically, it shows the argument for the converse result. The goal is to prove that there exists a path in $\mathscr{Y}_2$ from $Y_3^1$ to $Y_2^2$. To do this, as $Y_3^1 \stackrel{2}{\sim} Y_2^2$, we can find a ball $B_1$ intersecting both $Y_3^1$ and $Y_2^2$. Then, any point $p_1\in Y_3^1$ can be connected to a point $p_2 \in B_1$. Moreover, the point in this ball can be connected by a path to another point $p_3\in Y_2^1$. Since $Y_2^1\stackrel{2}{\sim}Y_3^2$, there exists another ball $B_2$ intersecting $Y_2^1$ and $Y_3^2$. Thus, $p_3$ can be connected to the point $p_4\in B_2$, and so on, until a path is constructed from any point in $Y_3^1$ to any point in $Y_2^2$.
	}
	\label{fig:example4}
\end{figure}

Now, we define a function associated to the generated graph of $Y$
\begin{definition}
	Let $Y$ be as in Definition \ref{def:nicedomain}. We will say that $\boldsymbol{\varphi} =\{\varphi_i^j\}_{i\in I, j\in I_j}$ is a function associated to the graph generated by $Y$ if, for every $i\in I$ and $j\in J_i$, $\varphi_i^j:[a_{i-1},a_{i}]\longrightarrow \R$.
	
	In addition we will say that $\boldsymbol{\varphi}$ satisfy the continuity conditions if, for every $i\in I$, $j\in J_i$ and $j'\in J_{i+1}$ such that $Y_i^{j}\stackrel{i}{\sim} Y_{i+1}^{j'}$, we have that $\varphi_{i}^j$ and $\varphi_{i+1}^{j'}$ are continuous in $a_i$ and, in addition
	\begin{equation}
		\label{eqn:continuitycond2}
		\varphi_{i}^j(a_i)=\varphi_{i+1}^{j'}(a_i).
	\end{equation}
\end{definition}
\begin{remark}
	Note that, if $\bm{\varphi}$ satisfy the continuity conditions, by Lemma \ref{lemma:continuityjoints}, we also have that, for every $i\in I$, $k\in K_i$, $j\in B_i^k$ and $j'\in A_i^k$, we have
	\begin{equation}
		\label{eqn:continuitycond}
		\varphi_{i}^j(a_i)=\varphi_{i+1}^{j'}(a_i).
	\end{equation}
	This is due to the fact that, if $j\in B_i^k$ and $j'\in A_i^k$, we can consider the chain \eqref{eqn:continuityjointseq1} and, by \eqref{eqn:continuitycond2} we obtain
	\begin{equation}
		\varphi_{i}^j(a_i)=\varphi_{i+1}^{j_2}(a_i)=\varphi_{i}^{j_3}(a_i)=\ldots=\varphi_{i}^{j_{l-1}}(a_i)=\varphi_{i+1}^{j'}(a_i).
	\end{equation}
\end{remark}

We now define a Hilbert space consisting of functions associated with the graph generated by $Y$. This space will be useful later on for defining our solution space.
\begin{definition}
	\label{def:HY}
	Let $Y$ be as in Definition \ref{def:nicedomain}. We will call $\mathscr{H}(Y)$ the vector space of functions associated to the graph generated by $Y$, $\boldsymbol{\varphi}=\{\varphi_i^j\}_{i\in I, j\in J_i}$, satisfying the continuity conditions and such that, for every $i\in I$, $j\in J_i$, $\varphi_i^j\in H^1((a_i,a_{i+1}),p_i^j)$. In addition, we will equip this space with the following norm
	\begin{equation}
		\label{eqn:normHY}
		\norm{\boldsymbol{\varphi}}_{\mathscr{H}(Y)}\defeq \left(\sum_{i\in I}\sum_{j\in J_i}\|\varphi_i^j\|^2_{H^1((a_i,a_{i+1}),p_i^j)}\right)^{1/2}.
	\end{equation}
\end{definition}
\begin{remark}
	\label{remark:uijcont}
	Note that, due to \eqref{eqn:pijbound} and the Sobolev embeddings, $H^1((a_{i-1},a_i),p_i^j)\subset C([a_{i-1},a_i])$ for every $i\in I$ and $j\in J_i$. This allows the continuity conditions \eqref{eqn:continuitycond} to make sense in the Definition \ref{def:HY}.
\end{remark}
Let us see that $\mathscr{H}(Y)$ is a Hilbert space.
\begin{lemma}
	Let $Y$ be as in Definition \ref{def:nicedomain}. Then, the space $\mathscr{H}(Y)$ equipped with the scalar product
	\begin{equation}
		\label{eqn:escHY}
		\braket{\boldsymbol{\varphi},\boldsymbol{\psi}}= \sum_{i\in I}\sum_{j\in J_i}\braket{\varphi_i^j,\psi_i^j}_{H^1((a_i,a_{i+1}),p_i^j)}
	\end{equation}
	is a Hilbert space and its norm is given by \eqref{eqn:normHY}.
\end{lemma}

\begin{proof}
	First of all, as, for every $i\in I$, $j\in J_i$, $H^1((a_i,a_{i+1}),p_i^j)$ is a Hilbert space, the Cartesian product $$\prod_{i\in I}\prod_{j\in J_i} H^1((a_i,a_{i+1}),p_i^j)$$ equipped with the scalar product  \eqref{eqn:escHY} is a Hilbert space too. Now, $\mathscr{H}(Y)$ is the subset of elements of $\prod_{i,j} H^1((a_i,a_{i+1}),p_i^j)$ which satisfy the continuity conditions.
	These continuity conditions \eqref{eqn:continuitycond} are of the form $T(\bm{\varphi})=0$ where
	\begin{equation}
		\begin{aligned}
			& T: && \prod_{i\in I}\prod_{j\in J_i} H^1((a_i,a_{i+1}),p_i^j) &&& \to &&&& \R \qquad\qquad \\
			& \ && \quad\qquad\qquad \varphi &&& \mapsto &&&& \varphi_{i_1}^{j_1}(a_{i_1})-\varphi_{i_1+1}^{j_2}(a_{i_1})
		\end{aligned}
	\end{equation}
	for some indices $i_1\in I$, $j_1\in J_{i_1}$ and $j_2\in J_{i_1+1}$. $T$ is linear and bounded because $\varphi_{i_1}^{j_1}(a_i)$ is the trace of $\varphi_{i_1}^{j_1}$ in $y=a_i$ which is a linear bounded operator in $H^1((a_{i_1},a_{i_1+1}),p_{i_1}^{j_1})$, and similarly for $\varphi_{i_1+1}^{j_2}(a_{i_1})$. Then, as $T$ is a bounded linear operator, its kernel is a closed subspace. Therefore, $\mathscr{H}(Y)$ is the finite intersection of closed subspaces, hence, a closed subspace, so a Hilbert space with the inherited scalar product.
\end{proof}

Now, we need the following proposition which states that a function belonging to $\mathscr{H}(Y)$ defines a function in $H^1_{\braket{\xi}}(Y)$.
\begin{prop}
	\label{prop:graphtoh1}
	Let $Y$ be as in Definition \ref{def:nicedomain} and $\boldsymbol{\varphi}=\{\varphi_i^j\}_{i\in I, j\in I_j}\in \mathscr{H}(Y)$. Then, the function
	\begin{equation}
		\label{eqn:graphtoh1eq1}
		\varphi(\xi,y)\defeq \left\{
		\begin{aligned}
			\varphi_i^j(y), \qquad & (\xi,y)\in Y_i^j, \ i\in I, \ j\in J_i \\
			\varphi_i^j(a_i), \qquad & (\xi,a_i)\in Y \cap \adh{Y_i^j}, \ i\in I, \ j\in J_i.
		\end{aligned}	
		\right.
	\end{equation}
	is well defined and, in addition, $\varphi\in H^1_{\braket{\xi}}(Y)$ with $\frac{\partial \varphi}{\partial y}(\xi,y)=\frac{\partial \varphi_i^j}{\partial y}(y)$ for every $(\xi,y)\in Y_i^j$. As a consequence
	\begin{equation}
		\label{eqn:graphtoh1eq3}
		\norm{\varphi}_{H^1_{\braket{\xi}}(Y)}^2=\sum_{i\in I}\sum_{j\in J_i}\int_{a_{i-1}}^{a_{i}}p_i^j\left(\left(\varphi_i^j\right)^2+\left(\frac{\partial \varphi_i^j}{\partial y}\right)^2\right)=\norm{\boldsymbol{\varphi}}^2_{\mathscr{H}(Y)}
	\end{equation}
\end{prop}
\begin{proof}
	Let us see that $\varphi$ is well defined. As the sets $Y_i^j$ are disjoint, we only have to focus on points of the form $(\xi_0,a_i)\in Y$ for some $i\in I$. In this case, $(\xi_0,a_i)$ may belong to $Y\cap \adh{Y_{i}^j}$ for, a priori, multiples values of $j$, so the definition \eqref{eqn:graphtoh1eq1} may be inconsistent. However, this is not the case. Taking a ball $B=B((\xi_0,a_i),r)\subset Y$, Lemma \ref{lemma:ballsjoint} guarantees that there exists, at most, one index $j\in J_i$ such that $B\cap Y_i^j\neq\emptyset$. The other possibility for an inconsistency would be that $(\xi_0,a_i)$ belong to $Y\cap \adh{Y_i^j}$ and $Y\cap \adh{Y_{i+1}^{j'}}$ for some $j\in J_i$ and $j'\in J_{i+1}$ but in that case $B\cap Y_i^j$ and $B\cap Y_{i+1}^{j'}$ are non-empty so $Y_i^j \stackrel{i}{\sim} Y_{i+1}^{j'}$ and, by the continuity conditions \eqref{eqn:continuitycond}, the definition is consistent.

	Now, let us check that $\varphi\in H^1_{\braket{\xi}}(Y)$.
	First of all, note that by Remark \ref{remark:uijcont}, $\varphi_i^j\in C([a_{i-1},a_{i}])$. Therefore, $\varphi\in C(Y_i^j)$ for every $Y_i^j$. Let us see that $\varphi$ is continuous in $Y$. We only need to prove the continuity for every $(\xi_0,a_i)\in Y$ where $\xi_0\in \RN$ and $i\in I$. Taking a ball $B=B((\xi_0,a_i),r)\subset \mathscr{Y}_i=Y \cap \left(\RN\times (a_{i-1},a_{i+1})\right)$, we have, due to Lemma \ref{lemma:ballsjoint}, that $B_i=\{(\xi,y)\in B: y<a_i\}\subset Y_i^j$ for some $j\in J_i$ and $B_{i+1}=\{(\xi,y)\in B: y>a_i\}\subset Y_{i+1}^{j'}$ for some $j'\in J_{i+1}$. Therefore,
	\begin{equation}
		\label{eqn:graphtoh1eq2}
		\restr{\varphi}{B}(\xi,y)=\left\{
		\begin{aligned}
			& \varphi_{i+1}^{j'}(y) \qquad && y> a_i \\
			& \varphi_i^j(y) \qquad && y\leq a_i.
		\end{aligned}
		\right. , \qquad (\xi,y)\in B.
	\end{equation}
	But, as $B\cap Y_i^{j}$ and $B\cap Y_{i+1}^{j'}$ are non-empty, $Y_{i}^j \stackrel{i}{\sim} Y_{i+1}^{j'}$, so, by the continuity conditions, $\varphi_i^j(a_i)=\varphi_{i+1}^{j'}(a_i)$ and, due to \eqref{eqn:graphtoh1eq2}, we have that $\restr{\varphi}{B}$ is continuous, so $\varphi$ is continuous in $(\xi_0,a_i)$. As $(\xi_0,a_i)$ was arbitrary, $\varphi\in C(Y)$.
	
	Now, by using Proposition \ref{prop:constxi}, we have that $\varphi\in H^1_{\braket{\xi}}(Y_i^j)$ for every $Y_i^j$ and  $\frac{\partial \varphi}{\partial y}(\xi,y)=\frac{\partial \varphi_i^j}{\partial y}(y)$ for every $(\xi,y)\in Y_i^j$. But then, as $\varphi$ is continuous, by using Corollary \ref{cor:h1cont2} and Remark \ref{remark:h1cont2} we automatically have that $\varphi\in H^1_{\braket{\xi}}(Y)$ and satisfy \eqref{eqn:graphtoh1eq3}.
\end{proof}
In the same way, if we have a function $\varphi\in H^1_{\braket{\xi}}(Y)$,  $\varphi$ defines a function in the associated graph.
\begin{prop}
	\label{prop:h1tograph}
	Let $Y$ be as in Definition \ref{def:nicedomain} and $\varphi\in H^1_{\braket{\xi}}(Y)$. Then, for $i\in I$ and $j\in J_i$, the functions
	\begin{equation}
		\label{eqn:h1tographeq1}
		\varphi_i^j(y)=\varphi(\xi,y) \qquad (\xi,y)\in Y_i^j
	\end{equation}
	are well defined and $\boldsymbol{\varphi}=\{\varphi_i^j\}_{i\in I, j\in J_i}$ is a function associated to the graph generated by $Y$ which satisfies the continuity conditions \eqref{eqn:continuitycond}. In addition, $\varphi_i^j$ belongs to $H^1((a_i,a_{i+1}),p_i^j)$ and $\frac{\partial \varphi_i^j}{\partial y}=\frac{\partial \varphi}{\partial y}$ in $(a_i,a_{i+1})$. As a consequence $\bm{\varphi}\in \mathscr{H}(Y)$ and
	\begin{equation}
		\label{eqn:h1tographeq2}
		\norm{\boldsymbol{\varphi}}^2_{\mathscr{H}(Y)}=\sum_{i\in I}\sum_{j\in J_i}\int_{Y_i^j}\left(\varphi^2+\left(\frac{\partial \varphi}{\partial y}\right)^2\right)=\norm{\varphi}^2_{H^1(Y)}
	\end{equation}
\end{prop}
\begin{proof}
	For $i\in I$ and $j\in J_i$, due to Proposition \ref{prop:constxi}, $\varphi_i^j\in H^1((a_i,a_{i+1}),p_i^j)$ and $\frac{\partial u_i^j}{\partial y}=\frac{\partial u}{\partial y}$ in $(a_i,a_{i+1})$. In addition, by Remark \ref{remark:uijcont}, $\varphi_i^j\in C([a_i,a_{i+1}])$. 
	
	Let us see that $\{\varphi_i^j\}_{i\in I, j\in J_i}$ satisfies the continuity conditions. If $Y_i^j \stackrel{i}{\sim} Y_{i+1}^{j'}$, then, by Lemma \ref{lemma:improvejoin}, there exists a ball $B=B((\xi_0,y_0),r)\subset \mathscr{Y}_i$ such that
	\eqref{eqn:defjoin} and \eqref{eqn:improvejoineq1} are satisfied.  Then, as $\varphi\in H^1_{\braket{\xi}}(B)$ and $B$ has connected horizontal sections, by Proposition \ref{prop:constxi}, there exists $w\in H^1((y_0-r,y_0+r),p_B)$ such that
	\begin{equation}
		w(y)=\varphi(\xi,y) \qquad \forall (\xi,y)\in B.
	\end{equation}
	Note that $p_B$ is continuous and positive. In addition, by \eqref{eqn:defjoin}, we have $a_i\in (y_0-r,y_0+r)$. Therefore, $w$ is continuous in a neighbourhood of $a_i$. But, from \eqref{eqn:h1tographeq1} and \eqref{eqn:improvejoineq1}, we have that $w(y)=\varphi_i^j(y)$ for every $y<a_i$ and $w(y)=\varphi_{i+1}^{j'}(y)$ for every $y>a_i$. Hence, by the continuity of $w$, $\varphi_i^j(a_i)=\varphi_{i+1}^{j'}(a_{i+1})$. As $i,j$ and $j'$ were arbitrary, the continuity conditions are satisfied.
\end{proof}
\begin{remark}
	\label{remark:isometric}
	Note that Propositions \ref{prop:graphtoh1} and \ref{prop:h1tograph} imply that the spaces $H^1_{\braket{\xi}}(Y)$ and $\mathscr{H}(Y)$ are isometric.
\end{remark}
In Figure \ref{fig:Y_pint} we present a graphical interpretation of the results of Proposition \ref{prop:h1tograph}.

\begin{figure}[h]
	\centering
	\includegraphics[width=0.8\textwidth]{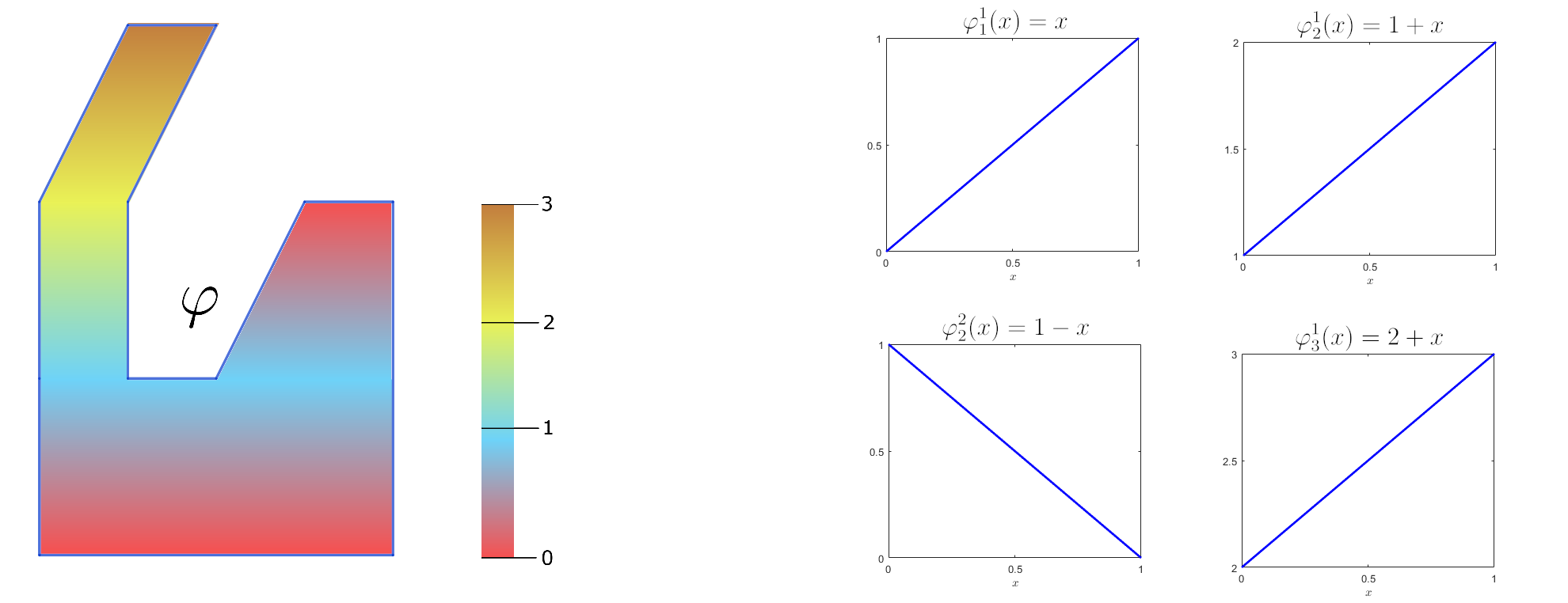}
	\caption{\small The figure shows an example of a function $\varphi \in H^1(Y)$, where $Y$ is the nicely decomposable domain from Figure \ref{fig:example}. The values of $\varphi$ are represented by colours, as indicated in the legend. On the right-hand side, we display the functions $\varphi_i^j$ constructed according to Proposition \ref{prop:h1tograph}. As we can observe, the continuity conditions are satisfied since $\varphi_0^0(1) = \varphi_1^0(1) = \varphi_1^1(1) = 1$ and $\varphi_1^0(2) = \varphi_2^0(2) = 2$.
	}
	\label{fig:Y_pint}
\end{figure}
\par\medskip

Now, we will describe the space of solutions for our problem in graphs. Its definition is analogous to $H(\theta)$ from \eqref{def:spaceofsolY} but instead of having, for each $x\in \Omega'\setminus \Theta_0$, functions in $H^1_{\braket{\xi}}(Y)$, we have functions in $\mathscr{H}(Y)$, which, by Remark \ref{remark:isometric} are isometric spaces. 
\begin{definition} 
	\label{def:spaceofsol}
	Let $Y$ be as in Definition \ref{def:nicedomain}. We define our space of solutions $\mathbb{H}(\theta)$, as the pairs of functions $(\bm{u^a},u^b)$ where
	$u^b\in H^1(\Omega^b)$, $\bm{u^a}\in L^2(\Omega',\theta;\mathscr{H}(Y))$ such that, denoting $\bm{u^a}(x)=\{u_i^j(x)\}_{i\in I,j\in J_i}$, we have
	\begin{equation}
		{Tr}^b(u^b)(x)=u_0^0(x)(0) \qquad a.e. \ x\in \Omega'\setminus \Theta_0.
	\end{equation}
\end{definition}
Note that in the above definition, $u_0^0(x)(0)$ is well defined because, by Remark \ref{remark:uijcont}, $u_0^0(x)\in C([a_0,a_1])$.

The following lemma justify that $\mathbb{H}(\theta)$ is a Hilbert space.
\begin{lemma}
	The space $\mathbb{H}(\theta)$, equipped with the scalar product
	\begin{equation}
		\label{eqn:spaceofsolescalar}
		\braket{(\bm{u^a},u^b),(\bm{v^a},v^b)}_{\mathbb{H}(\theta)}=\braket{u^b, v^b}_{H^1(\Omega^b)}+\braket{\bm{u^a}, \bm{v^a}}_{L^2(\Omega',\theta;\mathscr{H}(Y))},
	\end{equation}
	is a Hilbert space.
\end{lemma}
\begin{proof}
	The Cartesian product $H^1(\Omega^b)\times L^2(\Omega',\theta;\mathscr{H}(Y))$ is a Hilbert space equipped with the scalar product \eqref{eqn:spaceofsolescalar}, just because $L^2(\Omega',\theta;\mathscr{H}(Y))$ and $H^1(\Omega^b)$ are Hilbert spaces. In addition, the operator
	\begin{equation}
		\begin{aligned}
			& T: &&  L^2(\Omega',\theta;\mathscr{H}(Y))\times H^1(\Omega^b) &&& \to &&&& \R \qquad\qquad \\
			& \ && \quad\qquad\qquad (\bm{u^a},u^b) &&& \mapsto &&&& {Tr}^b(u^b)(x)-u_0^0(x)(0)
		\end{aligned}
	\end{equation}
	is linear and bounded because the trace operators in $H^1(\Omega^b)$ and $H^1((a_0,a_1), p_0^0)$ are bounded and linear.
	Since $\mathbb{H}(\theta)$ is the preimage of $0$ under the operator $T$, $\mathbb{H}(\theta)$ is a closed subspace, so a Hilbert space with the inherited scalar product.
\end{proof}

Now, let us state which the limit problem is. The following theorem is equivalent to Theorem \ref{thm:main} when $Y$ is nicely decomposed. The limit problem \eqref{eqn:thmmaineq1} transform into a limit problem in the space $\mathbb{H}(\theta)$ via the isometry between the spaces $H^1_{\braket{\xi}}(Y)$ and $\mathscr{H}(Y)$ noted in Remark \ref{remark:isometric}. The gross of the proof consists on justifying the regularity of solutions so that we can integrate by parts to obtain a classical formulation.
\begin{theorem}
	\label{thm:graphprob}
	Let $Y$ be as in Definition \ref{def:nicedomain}. Let $u_\varepsilon$ be the solution of \eqref{eqn:main}. Then, there exists $u^b\in H^1(\Omega^b)$ and $\bm{u^a}\in L^2(\Omega';\mathscr{H}(Y))$ such that, denoting $\bm{u^a}(x)=\{u_i^j(x)\}_{i\in I,j\in J_i}$ and defining
	\begin{equation}
		\label{eqn:graphprobeq6}
		\bar{u}^a_\varepsilon(x,y) \defeq 
			\left(\fint_{\omega_\varepsilon^n}  u_i^j(z)dz\right)(y) \qquad  (x,y)\in Y_\varepsilon^n \hbox{ and }\left(\frac{x-{\bar x^n_\varepsilon}}{l_\varepsilon^n},y\right)\in Y_i^j,
			\ \ \begin{aligned}
				&\text{\small{$n=1,\ldots, N_\eps$}} \\
				&\text{\small{$i\in I$, $j\in J_i$}}
			\end{aligned}
	\end{equation}
	we have that
	\begin{equation}
		\label{eqn:graphprobeq5}
		\lim_{\varepsilon\to 0}\norm{u_\varepsilon-\bar{u}^a_\varepsilon }_{H^1(\Omega^a_\varepsilon)}=0, \qquad	\lim_{\varepsilon\to 0}\norm{u_\varepsilon-u^b}_{H^1(\Omega^b)}=0.
	\end{equation}
	In addition,
	$(\bm{u^a},u^b)$ belongs to $\mathbb{H}(\theta)$ and is the only function of $\mathbb{H}(\theta)$ satisfying
	\begin{equation}
		\label{eqn:graphprobeq4}
		\int_{\Omega^b}\left(\nabla u^b\nabla \varphi^b+u^b\varphi^b-f\varphi^b\right)+\int_{\Omega'}\theta\sum_{i\in I}\sum_{j\in J_i}\int_{a_{i-1}}^{a_{i}} p_i^j\left(\frac{\partial u_i^j}{\partial y}\frac{\partial \varphi_i^j}{\partial y}+u_i^j\varphi_i^j-f\varphi_i^j\right)=0.
	\end{equation}
	for every $(\boldsymbol{\varphi}^{\bm{a}},\varphi^b)\in \mathbb{H}(\theta)$.
	
	Moreover, $u^b\in H^2_{loc}(\Omega^b)\cap H^1(\Omega^b)$, $u_i^j\in L^2(\Omega';H^2(a_i,a_{i+1}))$ for $i\in I$, $j\in J_i$  and satisfy the equations
	\begin{equation}
		\label{eqn:limgrapheq1}
		\left\{
		\begin{aligned}
			-\Lap u^b(x,y)+u^b(x,y) & = f(x,y) \qquad && a.e. \ (x,y)\in\Omega^b \\
			\frac{1}{p_i^j}\frac{d}{d y}\left(p_i^j\frac{d}{dy}\left(u_i^j(x)\right)\right)(y)+u_i^j(x)(y) & =f(x,y)  \qquad && a.e. \ x\in \Omega'\setminus \Theta_0, \ y\in (a_i,a_{i+1}) \ \forall i\in I_i, \ j\in J_i,
		\end{aligned}
		\right.
	\end{equation}
	the continuity conditions
	\begin{equation}
		\label{eqn:limgrapheq3}
		\left\{
		\begin{aligned}
			{Tr}^b(u^b)(x) & = u_0^0(x)(0) \qquad &&  a.e. \ x\in\Omega'\setminus \Theta_0, \\
			u_i^j(x)(a_i)& =u_{i+1}^{j'}(x)(a_i) \qquad && a.e. \ x\in \Omega'\setminus \Theta_0, \ \forall i\in I, \ \forall k\in K_i, \ \forall j\in B_i^k, \ \forall j'\in A_i^k,
		\end{aligned}
		\right.
	\end{equation}
	and the flux conditions
	\begin{equation}
		\label{eqn:limgrapheq2}
		\left\{
		\begin{aligned}
			\frac{\partial u^b}{\partial n}(x,y) & = 0 \qquad &&  a.e. \ (x,y)\in \partial\Omega\setminus (\Omega'\times\{0\}) \\
			\frac{\partial u^b}{\partial y}(x,0) & = \theta p_0^0 \frac{\partial u_0^0(x)}{\partial y}(0) \qquad && a.e. \ x\in \Omega', \\
			\sum_{j\in B_i^k}p_i^j(a_i)\frac{\partial u_i^j}{\partial y}(x,a_i)& =\sum_{j\in A_i^k}p_{i+1}^j(a_i)\frac{\partial u_{i+1}^j}{\partial y}(x,a_i) \qquad && a.e. \ x\in \Omega'\backslash \Theta_0, \ \forall i\in I, \ k\in K_i \\
			\frac{\partial u_M^j(x)}{\partial y}(a_M) & =0 \qquad && a.e. \ x\in\Omega'\backslash \Theta_0, \ \forall j\in J_M.
		\end{aligned}
		\right.
	\end{equation}
	where $\frac{\partial u^b}{\partial y}$ on $\partial \Omega$ is understood in the sense of distributions (see Remark \ref{remark:normal} below) and $A_i^k$, $B_i^k$ are as in Definition \ref{def:joint}. 
\end{theorem}
\begin{remark}
	\label{remark:normal}
	In the first and second equations of \eqref{eqn:limgrapheq2}, the term $\frac{\partial u^b}{\partial n}$ is not a priori well-defined, not even in the sense of traces, since $u^b \in H^1(\Omega^b)$ only implies that $\frac{\partial u^b}{\partial y} \in L^2(\Omega^b)$, and there is no bounded trace operator from $L^2(\Omega^b)$ to the boundary. However, due to the fact that $\Delta u^b = u^b - f \in L^2(\Omega^b)$, it is possible to define the normal derivative in a distributional sense (see \cite[Page 299]{mitidieri}). The normal derivative $\frac{\partial u^b}{\partial n}$ is then understood as an element of $H^{-1/2}(\partial \Omega)$, the dual of $H^{1/2}(\partial \Omega)$, and acts as follows. Given $\varphi \in H^{1/2}(\partial \Omega^b)$, we consider an extension $\varphi \in H^1(\Omega^b)$ (which is possible since $\Omega^b$ is Lipschitz; see \cite[Theorem 1.5.1.2]{grisvard}), and denote it by the same symbol. Then,
	\begin{equation}
		\label{eqn:defnormalder}
		\braket{\varphi,\frac{\partial u^b}{\partial n}}\defeq \int_{\Omega^b}\nabla \varphi \nabla u^b+\int_{\Omega^b}\Lap u^b \varphi.
	\end{equation}
	Note that, in the case where the domain $\Omega^b$ and $u^b$
	are smooth enough, the definition above coincides with the usual definition of the trace of the normal derivative (which will lie in $L^2(\partial \Omega^b)$) acting on an $L^2(\partial \Omega^b)$ function with the canonical dual correspondence.
	
	The first two equations of \eqref{eqn:limgrapheq2} then state that the above distribution, $\partial_n u^b$ is in fact an $L^2(\partial\Omega)$ function given by
	\begin{equation}
		g(x,y)=
		\left\{
		\begin{aligned}
			0 \qquad (&x,y)\in \partial \Omega \setminus( \Theta_0\times\{0\}) \\
			\theta p_0^0\frac{\partial u_0^0(x)}{\partial y}(0) \qquad &x\in \Omega'\backslash\Theta_0, \ y=0.
		\end{aligned}
		\right.
	\end{equation}
	which acts over $\varphi$ as
	\begin{equation}
		\braket{\varphi,\frac{\partial u^b}{\partial n}}=\int_{\partial \Omega} \varphi g.
	\end{equation}
\end{remark}
\begin{remark}
	Note that solutions to the weak problem \eqref{eqn:graphprobeq4} satisfy that $\bm{u^a}\in L^2(\Omega',\theta;\mathscr{H}(Y))$, so, although $\bm{u^a}(x)$ is defined for $x\in \Theta_0$ for a particular representative, its value is not relevant because it differs for every representative of the class of $L^2(\Omega',\theta;\mathscr{H}(Y))$. However, by redefining $\bm{u^a}(x)=0$ when $x\in \Theta_0$, it is proved that $\bm{u^a}\in L^2(\Omega';\mathscr{H}(Y))$. This is the same that happened in Proposition \ref{prop:uaisnice}.
\end{remark}

\begin{proof}
	\noindent\textbf{Step 1 (proving the convergence): }Consider the functions $(u^a,u^b)$ from Theorem \ref{thm:main}. We have that $u^a\in L^2(\Omega';H^1(Y))$, and, as $(u^a,u^b)\in H(\theta)$, we have $u^a\in L^2(\Omega';H^1_{\braket{\xi}}(Y))$. By Proposition \ref{prop:h1tograph}, as for $a.e$ $x\in\Omega'$, $u^a\in H^1_{\braket{\xi}}(Y)$, we can define $\bm{u^a}(x)\defeq \{u_i^j(x)\}_{i\in I,j\in J_i}$ where
	\begin{equation}
		\label{eqn:graphprobeq7}
		u_i^j(x)(y)=u^a(x)({\xi},y), \qquad ({\xi},y)\in Y_i^j,
	\end{equation}
	so that $\bm{u^a}(x)\in \mathscr{H}(Y)$. In addition, $\bm{u^a}\in L^2(\Omega';\mathscr{H}(Y))$ because
	$$
	\norm{\bm{u^a}}^2_{L^2(\Omega';\mathscr{H}(Y))}=\int_\Omega\norm{\bm{u^a}}^2_{\mathscr{H}(Y)}dx\myeq{\eqref{eqn:h1tographeq2}}\int_\Omega\norm{u^a}^2_{H^1_{\braket{\xi}}(Y)}dx=\norm{u^a}^2_{L^2(\Omega'; H^1_{\braket{\xi}}(Y))}.
	$$

	Now, the definition of $\bar{u}^a_\varepsilon$ from \eqref{eqn:thmmaineq1} and \eqref{eqn:graphprobeq6} coincide. It follows from \eqref{eqn:graphprobeq7} that
	\begin{equation}
		u_i^j(z)(y)=u^a(z)\left(\frac{x-\bar x_\varepsilon^n}{l_\varepsilon^n},y\right), \qquad  a.e. \ z\in \Omega', \ \ a.e.  \ \left(\frac{x-{\bar x^n_\varepsilon}}{l_\varepsilon^n},y\right)\in Y_i^j,
	\end{equation}
	for $i\in I$ and $j\in J_i$. Therefore the convergence \eqref{eqn:graphprobeq5} follows from \eqref{eqn:thmmaineq2}.\\
	
	\noindent\textbf{Step 2 (weak formulation): }Now, let us prove that $(\bm{u^a},u^b)\in \mathbb{H}(\theta)$ and satisfy \eqref{eqn:graphprobeq4}. Since $(u^a,u^b)\in H(\theta)$, we have ${Tr}^b(u^b)(x)=u^a(x)(\xi,0)$ for a.e. $x\in \Omega'\setminus \Theta_0$ and a.e. $\xi \in \omega$. Therefore, ${Tr}^b(u^b)(x)=u_0^0(x)(0)$ for a.e. $x\in \Omega'\setminus \Theta_0$ and $(\bm{u^a},u^b)\in \mathbb{H}(\theta)$. Now, the weak formulation \eqref{eqn:graphprobeq4} is obtained from \eqref{eqn:thmmaineq3}. In the same way we did with $u^a$, given $\varphi^a\in L^2(\Omega',\theta;H^1_{\braket{\xi}}(Y))$, we can define, via Proposition \ref{prop:h1tograph}, for a.e. $x\in \Omega'\setminus \Theta_0$, $\boldsymbol{\varphi}^{\bm{a}}(x)=\{\varphi_i^j(x)\}_{i\in I, j\in J_i}$ given by
	\begin{equation}
		\varphi_i^j(x)(y)=\varphi^a(x)({\xi},y) \qquad ({\xi},y)\in \Omega_i^j.
	\end{equation}
	Then, using that $Y$ can be decomposed in the sets $\{Y_i^j\}_{i\in I,j\in J_i}$ and some other sets of measure $0$ (see Definition \ref{def:nicedomain}), we have
	\begin{equation}
		\label{eqn:graphprobeq8}
		\begin{aligned}
			&\int_{\Omega'}\theta\int_Y\left( \frac{\partial u^a}{\partial y}\frac{\partial \varphi^a}{\partial y}+u^a \varphi^a-f^*\varphi^a\right)=
			\int_{\Omega'}\theta\sum_{i\in I}\sum_{j\in J_i}\int_{Y_i^j}\left( \frac{\partial u^a}{\partial y}\frac{\partial \varphi^a}{\partial y}+u^a \varphi^a-f^*\varphi^a\right)
			\\
			\myeq{\eqref{eqn:graphprobeq7}}
			&\int_{\Omega'}\theta(x)\sum_{i\in I}\sum_{j\in J_i}\int_{Y_i^j}\left(\frac{\partial u_i^j}{\partial y}(x)(y)\frac{\partial \varphi_i^j}{\partial y}(x)(y)+u_i^j(x)(y)\varphi_i^j(x)(y)-f(x,y)\varphi_i^j(x)(y)\right)d\xi dy dx
			\\
			= & \int_{\Omega'}\theta\sum_{i\in I}\sum_{j\in J_i}\int_{a_{i-1}}^{a_{i}} p_i^j\left(\frac{\partial u_i^j}{\partial y}\frac{\partial \varphi_i^j}{\partial y}+u_i^j\varphi_i^j-f\varphi_i^j\right)
		\end{aligned}
	\end{equation}
	so, from \eqref{eqn:thmmaineq3} and \eqref{eqn:graphprobeq8} we obtain \eqref{eqn:graphprobeq4}.
	
	\noindent\textbf{Step 3 (existence and uniqueness): }Now, the existence and uniqueness of solutions of \eqref{eqn:graphprobeq4} is straightforward by the use of Lax-Milgram Theorem just because
	\begin{equation}
		a(u, \varphi)=\int_{\Omega^b}\left(\nabla u^b\nabla \varphi^b+u^b\varphi^b\right)+\int_{\Omega'}\theta\sum_{i=0}^{M}\sum_{j=0}^{m_i}\int_{a_{i-1}}^{a_{i}} p_i^j\left(\frac{\partial u_i^j}{\partial y}\frac{\partial \varphi_i^j}{\partial y}+u_i^j\varphi_i^j\right)=0
	\end{equation}
	is a bilinear coercive form in $\mathbb{H}(\theta)$ and
	\begin{equation}
		L_f(\varphi)=\int_{\Omega^b}f\varphi^b+\int_{\Omega'}\theta\sum_{i=0}^{M}\sum_{j=0}^{m_i}\int_{a_{i-1}}^{a_{i}} p_i^j f\varphi_i^j
	\end{equation}
	is a linear form, so Lax-Milgram gives the existence and uniqueness of solutions. 
	
	\noindent\textbf{Step 4 (regularity): }Now we need to prove that $u^b\in H^2_{loc}(\Omega^b)\cap H^1(\Omega^b)$, $u_i^j\in L^2(\Omega';H^2(a_i,a_{i+1}))$ for $i\in I$, $j\in J_i$, and they satisfy equations \eqref{eqn:limgrapheq1}, \eqref{eqn:limgrapheq3} and \eqref{eqn:limgrapheq2}.

	First, to prove that $u^b\in H^2_{loc}(\Omega^b)$, we take $\varphi^b\in C^\infty_c(\Omega^b)$ and $\varphi_i^j\equiv 0$ for every $i\in I$ and $j\in J_i$, so, as the continuity conditions are satisfied, $(\bm{\varphi^a},\varphi^b)\in \mathbb{H}(\theta)$ and \eqref{eqn:graphprobeq4} reduces to
	\begin{equation}
		\label{eqn:graphprobeq9}
		\int_{\Omega^b}\left(\nabla u^b\nabla \varphi^b+u^b\varphi^b-f\varphi^b\right)=0.
	\end{equation}
	As \eqref{eqn:graphprobeq9} is true for any $\varphi\in C^\infty_c(\Omega^b)$, using local regularity results (see \cite[Section 6.3.1, Theorem 1]{evans}) we have $u^b\in H^2_{loc}(\Omega^b)$ and $-\Lap u^b + u^b = f$ in $\Omega^b$.
	
	Now, let us study $\bm{u^a}$. We first have that $u_i^j(x)=0$ for $i\in I$, $j\in J_i$ and a.e. $x\in \Theta_0$ by definition (see Theorem \ref{thm:main}), so  $u_i^j(x)\in H^2((a_{i-1},a_i))$ for every $x\in \Theta_0$. Hence, we have to study $u^a$ in $\Omega'\setminus\Theta_0$. In order to do that, given $i\in I$ and $j\in J_i$, we take $\varphi_i^j(x)(y)=\phi(x)\psi(y)$ for $x\in\Omega'$, $y\in (a_{i-1},a_i)$  where $\phi\in C^\infty_c(\Omega')$ and $\psi\in C^\infty_c((a_{i-1},a_i))$. We also define $\varphi_{i'}^{j'}\equiv0$ for any other $i'\in I$, $j'\in J_{i'}$, $(i',j')\neq (i,j)$ and we take $\varphi^b\equiv 0$. Then, as the continuity conditions are satisfied, $(\bm{\varphi^a},\varphi^b)\in \mathbb{H}(\theta)$ and \eqref{eqn:graphprobeq4} reduces to
	\begin{equation}
		\int_{\Omega'}\phi(x)\theta(x)\int_{a_{i-1}}^{a_{i}}p_i^j(y)\left(\frac{\partial u_i^j(x)}{\partial y}(y)\frac{\partial \psi}{\partial y}(y)+u_i^j(x)(y)\psi(y)-f(x,y)\psi(y)\right)dydx = 0.
	\end{equation}
	Therefore, as $\phi\in C^\infty_c(\Omega)$ is arbitrary, we have for a.e. $x\in \Omega'$,
	\begin{equation}
		\theta(x)\int_{a_{i-1}}^{a_{i}}p_i^j(y)\left(\frac{\partial u_i^j(x)}{\partial y}(y)\frac{\partial \psi}{\partial y}(y)+u_i^j(x)(y)\psi(y)-f(x,y)\psi(y)\right)dy = 0.
	\end{equation}
	Hence, for a.e. $x\in \Omega'\setminus \Theta_0$,
	\begin{equation}
		\label{eqn:graphprobeq10}
		\int_{a_{i-1}}^{a_{i}}p_i^j(y)\left(\frac{\partial u_i^j(x)}{\partial y}(y)\frac{\partial \psi}{\partial y}(y)+u_i^j(x)(y)\psi(y)-f(x,y)\psi(y)\right)dydx = 0.
	\end{equation}
	As $\psi\in C^\infty_c((a_{i-1},a_i))$, \eqref{eqn:graphprobeq10} means that $u_i^j(x)$
	is a weak solution of
	\begin{equation}
		\label{eqn:graphprobeq11b}
		\frac{1}{p_i^j}\frac{d}{d y}\left(p_i^j\frac{d}{dy}\left(u_i^j(x)\right)\right)(y)+u_i^j(x)(y) =f(x,y), \qquad y\in (a_{i-1},a_i).
	\end{equation}
	Therefore, using local regularity results (see \cite[Section 6.3.1, Theorem 1]{evans}), we obtain $u_i^j(x)\in H^2_{loc}((a_{i-1},a_i))$. But then, we can expand \eqref{eqn:graphprobeq11b} to obtain
	\begin{equation}
		\label{eqn:graphprobeq11}
		\frac{d^2(u_i^j(x))}{dy^2}(y) =f(x,y)-\frac{(p_i^j)'(y)}{p_i^j(y)}\frac{d u_i^j(x)}{d y}(y)-u_i^j(x)(y), \qquad a.e. \ y\in (a_{i-1},a_i).
	\end{equation}
	From Fubini's theorem, we have that $f(x,\cdot)\in L^2((a_{i-1},a_i))$ for a.e. $x\in \Omega'$. In addition, since $u_i^j\in L^2(\Omega';H^1((a_{i-1},a_i)))$, $u_i^j(x)\in H^1((a_{i-1},a_i))$ for a.e. $x\in \Omega'\setminus\Theta_0$. Therefore, from \eqref{eqn:pijbound} and
	\eqref{eqn:graphprobeq11}, we obtain that $\frac{d^2(u_i^j(x))}{dy^2}\in L^2((a_{i-1},a_i))$ for a.e. $x\in \Omega'\setminus \Theta_0$, so $u_i^j(x)\in H^2((a_{i-1},a_i))$ for a.e. $x\in \Omega'\setminus\Theta_0$. In addition, from \eqref{eqn:graphprobeq11} and the fact that $u_i^j\in L^2(\Omega'; H^1(a_{i-1},a_{i}))$ and $f\in L^2(\R^{N+1})=L^2(\RN; L^2(\R))$ (see Lemma \ref{lemma:bochnerproduct} and Remark \ref{remark:bochnerproduct}), we obtain
	\begin{equation}
		\bnorm{\frac{d^2 u_i^j(\cdot)}{dy^2}}_{L^2(\Omega';L^2((a_{i-1},a_{i})))}\leq \norm{f}_{L^2(\R^{N+1})}+\norm{\bm{u^a}}_{L^2(\Omega';\mathscr{H}(Y))}<\infty.
	\end{equation}
	so $u_i^j \in L^2(\Omega';H^2((a_{i-1},a_{i})))$.
	
	\noindent \textbf{Step 5 (pointwise equations): }
	In the previous step we have already proved the equations \eqref{eqn:limgrapheq1}. In addition, from Step 1 and 2, we have that $(\bm{u^a},u^b)\in \mathbb{H}(\theta)$, so the continuity conditions \eqref{eqn:limgrapheq3} are satisfied. Therefore, we only need to obtain the flux conditions \eqref{eqn:limgrapheq2}. Let us obtain them from the weak formulation \eqref{eqn:graphprobeq4}.
	
	In order to do that, given $i\in I$, $i\neq M$ and $k\in K_i$, we take $\varphi_i^j=\phi\psi_1$ for every $j\in B_i^k$ and $\varphi_i^{j'}=\phi\psi_2$ for every $j\in A_i^k$ where $\phi\in C^\infty_c(\Omega')$, $\psi_1\in C^\infty_c((a_{i-1},a_i])$ and $\psi_2\in C^\infty_c([a_{i},a_{i+1}))$ satisfying $\psi_1(a_i)=\psi_2(a_i)$. We also define $\varphi_{i'}^{j'}=0$ for the rest of the indices $i'\in I$ and $j'\in J_{i'}$, and we take $\varphi^b=0$ too. Then, as the continuity conditions are satisfied, $(\bm{\varphi^a},\varphi^b)\in \mathbb{H}(\theta)$ and \eqref{eqn:graphprobeq4} reduces to
	\begin{equation}
		\begin{aligned}
			&\int_{\Omega'}\phi(x)\theta(x)\left(\sum_{j\in B^k_i}\int_{a_{i-1}}^{a_{i}}p_i^j(y)\left(\frac{\partial u_i^j(x)}{\partial y}(y)\frac{\partial \psi_1}{\partial y}(y)+u_i^j(x)(y)\psi_1(y)-f(x,y)\psi_1(y)\right)dy\right.
			\\ +& \left.\sum_{j\in A_i^k}\int_{a_{i}}^{a_{i+1}}p_{i+1}^j(y)\left(\frac{\partial u_{i+1}^j(x)}{\partial y}(y)\frac{\partial \psi_2}{\partial y}(y)+u_{i+1}^j(x)(y)\psi_2(y)-f(x,y)\psi_2(y)\right)dy\right)
			dx = 0.
		\end{aligned}
	\end{equation}
	so, as $\phi\in C^\infty_c(\Omega')$ is arbitrary, for a.e. $x\in\Omega'\setminus \Theta_0$,
	\begin{equation}
		\begin{aligned}
			&\sum_{j\in B^k_i}\int_{a_{i-1}}^{a_{i}}p_i^j(y)\left(\frac{\partial u_i^j(x)}{\partial y}(y)\frac{\partial \psi_1}{\partial y}(y)+u_i^j(x)(y)\psi_1(y)-f(x,y)\psi_1(y)\right)dy
			\\ +& \sum_{j\in A_i^k}\int_{a_{i}}^{a_{i+1}}p_{i+1}^j(y)\left(\frac{\partial u_{i+1}^j(x)}{\partial y}(y)\frac{\partial \psi_2}{\partial y}(y)+u_{i+1}^j(x)(y)\psi_2(y)-f(x,y)\psi_2(y)\right)dy
			= 0.
		\end{aligned}
	\end{equation}
	Then, integrating by parts, we obtain
	\begin{equation}
		\begin{aligned}
			&\sum_{j\in B^k_i}\int_{a_{i-1}}^{a_{i}}
			p_i^j(y)\left(
			-\frac{1}{p_i^j}\left( \frac{\partial}{\partial y}\left(p_i^j\frac{\partial u_i^j(x)}{\partial y}\right)\right)
			+u_i^j(x)-f
			\right)\psi_1
			\\ +& \sum_{j\in A_i^k}\int_{a_{i}}^{a_{i+1}}
			p_{i+1}^j(y)\left(
			-\frac{1}{p_{i+1}^j}\left( \frac{\partial}{\partial y}\left(p_{i+1}^j\frac{\partial u_{i+1}^j(x)}{\partial y}\right)\right)
			+u_{i+1}^j(x)-f
			\right)\psi_2
			\\
			+
			& \sum_{j\in B^k_i} p_i^j(a_i)\frac{\partial u_i^j(x)}{\partial y}(a_i)-\sum_{j\in A_i^k}p_{i+1}^j(a_i)\frac{\partial u_{i+1}^j(x)}{\partial y}(a_i)
			= 0.
		\end{aligned}
	\end{equation}
	so using \eqref{eqn:limgrapheq1} we obtain the third equation in \eqref{eqn:limgrapheq2}. 
	
	The fourth equation in \eqref{eqn:limgrapheq2} can be obtained in a similar way. Given $j\in J_M$, choose $\varphi_M^j(x)(y)=\phi(x)\psi(y)$ for $x\in \Omega'$ and $y\in (a_{M-1},a_M)$, where $\phi\in C^\infty_c(\Omega')$ and $\psi\in C^\infty_c((a_{M-1},a_M])$. For the rest of the indices $i'\in I$, $j'\in J_i$, $(i',j')\neq (i,j)$, choose $\varphi_{i'}^{j'}=0$, and take $\varphi^b=0$ too. Then, as the continuity conditions are satisfied, $(\varphi^b,\bm{\varphi^a})\in \mathbb{H}(\theta)$ and \eqref{eqn:graphprobeq4} reduces to
	\begin{equation}
		\begin{aligned}
			\int_{\Omega'}\phi(x)\theta(x)\int_{a_{M-1}}^{a_{M}}p_M^j(y)\left(\frac{\partial u_M^j(x)}{\partial y}(y)\frac{\partial \psi}{\partial y}(y)+u_M^j(x)(y)\psi(y)-f(x,y)\psi(y)\right)dy		
			dx = 0.
		\end{aligned}
	\end{equation}
	so, as $\phi\in C^\infty_c(\Omega')$ was arbitrary, for a.e. $x\in \Omega'\setminus \Theta_0$, we obtain
	\begin{equation}
		\begin{aligned}
			\int_{a_{M-1}}^{a_{M}}p_M^j(y)\left(\frac{\partial u_M^j(x)}{\partial y}(y)\frac{\partial \psi}{\partial y}(y)+u_M^j(x)(y)\psi(y)-f(x,y)\psi(y)\right)dy		
			= 0.
		\end{aligned}
	\end{equation}
	and integrating by parts
	\begin{equation}
		\begin{aligned}
			\int_{a_{M-1}}^{a_{M}}p_M^j\left(-\frac{1}{p_M^j}\frac{\partial}{\partial y}\left(p_M^j\frac{\partial u_M^j(x)}{\partial y}(y)\right)+u_M^j(x)-f(x,\cdot)\right)\psi+ p_M^j(a_M)\frac{\partial u_M^j(x)}{\partial y}(a_M) 		
			= 0.
		\end{aligned}
	\end{equation}
	so using \eqref{eqn:limgrapheq1} we obtain the fourth equation in \eqref{eqn:limgrapheq2}.

	Finally, let us obtain the first two equations in \eqref{eqn:limgrapheq2}. In order to do that, take $\varphi^b \in C^\infty(\adh{\Omega^b})$. Then, define $\phi(x)=\varphi^b(x,0)$ for any $x\in \Omega'$ so $\phi\in C^\infty(\adh{\Omega'})$ and take $\psi\in C^\infty_c([a_0,a_1))$ with $\psi(0)=1$. Define then $u_0^0(x)(y)=\phi\psi$ for $x\in \Omega'$ and $y\in (a_{M-1},a_M)$, and $u_i^j\equiv 0$ for $i\in I$, $j\in J_i$, $(i,j)\neq (0,0)$. Then, as the continuity conditions are satisfied, $(\bm{\varphi^a},\varphi^b)\in \mathbb{H}(\theta)$ and \eqref{eqn:graphprobeq4} reduces to
	\begin{equation}
		\int_{\Omega^b}\left(\nabla u^b\nabla \varphi^b+u^b\varphi^b-f\varphi^b\right)+\int_{\Omega'}\phi\theta\int_{a_0}^{a_1} p_0^0\left(\frac{\partial u_0^0}{\partial y}\frac{\partial \psi}{\partial y}+u_0^0\psi-f\psi\right)=0.
	\end{equation}
	so, using the definition of the distributional normal derivative \eqref{eqn:defnormalder} in the first integral and integrating by parts in the second one, we have
	\begin{equation}
		\begin{aligned}
			& \int_{\Omega^b}\left(-\Lap u^b+u^b-f\right)\varphi^b+\int_{\Omega'}\phi\theta\int_{a_0}^{a_1} p_0^0\left(\frac{1}{p_0^0}\frac{\partial}{\partial y}\left(p_0^0\frac{\partial u_0^0}{\partial y}\right)+u_0^0-f\right)\psi
			\\
			& +\braket{\frac{\partial u^b}{\partial n},\varphi^b}-\int_{\Omega'}\phi(x)\psi(a_0)\theta(x) p_0^0(a_0) \frac{\partial u_0^0(x)}{\partial y}(a_0)dx.
		\end{aligned}
	\end{equation}
	so using equation \eqref{eqn:graphprobeq4} and the fact that $a_0=0$, $\phi(x)=\varphi^b(x,0)$ and $\psi(0)=1$, we obtain
	\begin{equation}
		\braket{\frac{\partial u^b}{\partial n},\varphi^b}-\int_{\Omega'}\varphi^b(x)\theta(x) p_0^0(0) \frac{\partial u_0^0(x)}{\partial y}(0)dx=0,
	\end{equation}
	so, defining
	\begin{equation}
		g(x,y)\defeq \left\{
		\begin{aligned}
			& 0, \qquad &&(x,y)\in \partial \Omega^b \setminus \{\Omega'\times 0\} \\
			& \theta(x) p_0^0(0) \frac{\partial u_0^0(x)}{\partial y}(0) \qquad && x\in \Omega', \ y=0.
		\end{aligned}
		\right.
	\end{equation}
	we have
	\begin{equation}
		\braket{\frac{\partial u^b}{\partial n},\varphi^b} = \int_{\partial \Omega} g\varphi^b
	\end{equation}
	so $g=\frac{\partial u^b}{\partial n}$ in the sense of Remark \ref{remark:normal} and we obtain the first and second equations in \eqref{eqn:limgrapheq2}. Note that $g$ belongs to $L^2(\partial \Omega^b)$ because $u_0^0\in L^2(\Omega';H^2((a_0,a_1)))$ so $\frac{\partial u_0^0}{\partial y}\in L^2(\Omega';H^1((a_0,a_1)))$ so, using that the trace on $y=0$ is bounded from $H^1((a_0,a_1))$ to $\R$, we obtain $\frac{\partial u_0^0(\cdot)}{\partial y}(0)\in L^2(\Omega')$.

\end{proof}

\section*{Acknowledgments}
The author JDdT expresses his gratitude to Professor Antonio Gaudiello (University of Cassino and Southern Lazio) for his warm hospitality and the fruitful discussions held during the Sixth Workshop on Thin Structures in 2023, which helped giving rise to this work.

\textbf{Funding.} Both authors JMA and JDdT were partially supported by grants PID2022-137074NB-I00, CEX2023-001347-S, and (for JDdT) CEX2019-000904-S; the last two from the ``Severo Ochoa Programme for Centres of Excellence in R\&D''. All these grants were funded by MCIN/AEI/\allowbreak10.13039/501100011033. Both authors also received support from ``Grupo de Investigaci\'on 920894 -
CADEDIF'', UCM, Spain.

	\appendix

	\section{Lebesgue-Bochner spaces}
	\label{app:bochner}
	In this section, we prove some technical results of measure theory that we needed in the proofs, but which we decided to separate in order to facilitate the readability of the proofs. Through the section, we will use the term $\theta$-measurability. When considering the (Lebesgue) bounded measurable function $\theta: \Omega \to [0,\infty)$, we can consider the associated measure $d\theta = \theta dx$ where $dx$ is the Lebesgue measure. However, the space $(\Omega, \Sigma, d\theta)$, where $\Sigma$ denotes the Lebesgue $\sigma$-algebra, is not necessarily a complete measure space (because of the non-measurable sets of $\Omega$ where $\theta$ vanishes). However, it can easily completed adding the necessary sets to the $\sigma$-algebra, $(\Omega, \Sigma_\theta, d\theta)$. We will refer to $\theta$-measurability, to distinguish from standard (Lebesgue) measurability, to the measurability of functions with respect to that $\sigma$-algebra.
	
	We begin by proving a result that relates product spaces to Bochner spaces when ${\hil}$ is a Lebesgue or Sobolev space.
	\begin{lemma}
		\label{lemma:bochnerproduct}
		Let $\Omega'$, $\theta$ and $Y$ as in Section \ref{sec:jtxtools}. Recall the notation $W=\Omega'\times Y$. Then,
		\begin{equation}
			L^2(\Omega',\theta;L^2(Y)) = L^2(W, \theta)
		\end{equation}
		in the sense that, given $u\in L^2(\Omega',\theta;L^2(Y))$, there exists $\hat{u}\in L^2(W, \theta)$ such that, for $\theta$-a.e. $x\in \Omega'$,
		\begin{equation}
			\label{eqn:bochnerproducteq1}
			\hat{u}(x,\xi,y)=u(x)(\xi,y), \qquad  a.e. \ (\xi,y)\in Y,
		\end{equation}
		The converse is also true, that is, if $\hat{u}$ belongs to $L^2(W,\theta)$, then, $u$ defined as in \eqref{eqn:bochnerproducteq1} belongs to $L^2(\Omega',\theta;L^2(Y))$. In the same way $L^2(\Omega';L^2(Y)) = L^2(W)$.
	\end{lemma}

	\begin{proof}
		\noindent\textbf{Step 1 (direct result): }Let $u\in L^2(\Omega',\theta;L^2(Y))$. As $u$ is square integrable, we can find a sequence $\{u_n\}_{n\in\mathbb{N}}$ of simple functions satisfying $\lim_{n\to\infty} \int_{\Omega'}\theta(x)\norm{u(x)-u_n(x)}^2_{L^2(Y)}dx = 0$.
		Now, define $\hat{u}_n (x,\xi,y) = u_n(x)(\xi,y)$ for $(x,\xi,y)\in \Omega'\times Y$. Let us see that $\hat{u}_n$ is $\theta$-measurable. As $u_n$ is a simple function, let us express it as
		\begin{equation}
			u_{n}(x)= \sum_{i\in I_n}\chi_{{\Omega'}_{n,i}}(x) u_{n,i} \qquad x\in \Omega',
		\end{equation}
		where $u_{n,i}\in L^2(Y)$ and ${\Omega'}_{n,i}$ are $\theta$-measurable subsets of $\Omega'$ for every $i\in I_n$. Therefore, given $E\subset \R$ Borel set, we have $\hat{u}_n^{-1}(E)=\bigcup_{i\in I_n}\left(\Omega_{n,i}'\times u_{n,i}^{-1}(E)\right)$
		which is a $\theta$-measurable set as it is the finite union of $\theta$-measurable sets. Therefore, $\hat{u}_n$ is $\theta$-measurable. In addition, $\hat{u}_n\in L^2(\Omega'\times Y)$ because, applying Tonelli's theorem
		\begin{equation}
			\int_{\Omega'\times Y} \theta(x)\hat{u}^2_n(x,\xi,y)dx d\xi dy = \int_{\Omega'} \theta(x)\int_{Y}u_n^2(x)(\xi,y)d\xi dy dx = \int_{\Omega'}\theta(x) \norm{u_n(x)}^2_{L^2(Y)}dx < \infty.
		\end{equation}
		Now, with the same argument
		\begin{equation}
			\lim_{n,m\to\infty} \norm{\hat{u}_n(x)-\hat{u}_m(x)}^2_{L^2(\Omega'\times Y,\theta)} = \lim_{n,m\to\infty}\int_{\Omega'}\theta(x)\norm{u_n(x)-u_m(x)}^2_{L^2(Y)}dx = 0.
		\end{equation}
		Hence, $\{\hat{u}_n\}_{n\in\mathbb{N}}$ is a Cauchy sequence in $L^2(\Omega'\times Y,\theta)$, so we can find $\hat{u}\in L^2(\Omega'\times Y,\theta)$ such that, through a subsequence, that we denote the same, $\hat{u}_n \to \hat{u}$ in $L^2(\Omega'\times Y,\theta)$. Finally, let us prove that $\hat{u}$ satisfies \eqref{eqn:bochnerproducteq1}
		\begin{equation}
			\begin{aligned}
				& \int_{\Omega'} \theta(x)\int_Y \abs{u(x)(\xi,y)-\hat{u}(x,\xi,y)}^2  d\xi dy dx 
				\leq  \lim_{n\to \infty}\left(  \int_{\Omega'} \theta(x)\norm{u(x)-u_n(x)}_{L^2(Y)}^2dx\right.
				\\+ &\left.\int_{\Omega'}\theta(x)\int_Y \abs{\hat{u}_n(x,\xi,y)-\hat{u}(x,\xi,y)}^2d\xi dy dx\right) =
				\lim_{n\to \infty}  \int_{\Omega'\times Y} \theta(x)\abs{\hat{u}_n(x,\xi,y)-\hat{u}(x,\xi,y)}^2dx d\xi dy
				= 0.
			\end{aligned}
		\end{equation}
		\noindent\textbf{Step 2 (converse result): } 
		Given $\hat{u}\in L^2(W,\theta)$, define $u$ as in \eqref{eqn:bochnerproducteq1}. Let us prove that $u: \Omega' \to L^2(Y)$ is $\theta$-measurable. As $L^2(Y)$ is separable, weak and strong $\theta$-measurability are equivalent, so we only need to prove that $u$ is weakly $\theta$-measurable. As $L^2(Y)$ is a reflexive space, for every $v\in L^2(Y)$, we need to check the $\theta$-measurability of the numerical function $x\mapsto \int_{Y}u(x)(\xi,y)v(\xi,y)d\xi dy=\int_{Y}\hat{u}(x,\xi,y)v(\xi,y)d\xi dy$,
		which is a straightforward consequence of Fubini's theorem because 
		$$\int_{W}\theta(x)\hat{u}(x,\xi,y)v(\xi,y)dxd\xi dy \leq \norm{\hat{u}}_{L^2(W,\theta)}\int_{\Omega'}\theta(x)\norm{v}_{L^2(Y)}dx<\infty.$$ The $L^2$ integrability of $u$ then follows from Fubini's theorem again,
		\begin{equation}
			\label{eqn:bochnerproducteq4}
			\begin{aligned}
				& \int_{W} \theta(x)\abs{\hat{u}(x,\xi,y)}^2 dxd\xi dy  =
				\int_{\Omega'}\theta(x) \int_{Y}\abs{\hat{u}(x,\xi,y)}^2d\xi dy dx 
				\\ 
				= & \int_{\Omega'}\theta(x) \int_{Y}\abs{u(x)(\xi,y)}^2d\xi dy dx 
				= \int_{\Omega'}\theta(x)\norm{u(x)}_{L^2(Y)}^2dx = \norm{u}_{L^2(\Omega',\theta;L^2(Y))}.
			\end{aligned}
		\end{equation}
	\end{proof}
	Following the same idea of the previous lemma, we have
	\begin{lemma}
		\label{lemma:bochnerproductb}
		Let $\Omega'$, $\theta$ and $Y$ as in Section \ref{sec:jtxtools}. Recall the notation $W=\Omega'\times Y$. Then,
		\begin{equation}
			\label{eqn:bochnerproductbeq2}
			L^2(\Omega',\theta;H^1(Y)) = \left\{ \hat{u} \in L^2(W, \theta): \frac{\partial \hat{u}}{\partial y}\in L^2(W, \theta), \ \nabla_{\xi} \hat{u} \in (L^2(W, \theta))^N\right\},
		\end{equation}
		in the sense that, given $u\in L^2(\Omega',\theta;H^1(Y))$, there exists $\hat{u}\in L^2(W,\theta)$
		\begin{equation}
			\label{eqn:bochnerproducteqb}
			\hat{u}(x,\xi,y)=u(x)(\xi,y), \qquad \theta-a.e. \ x\in \Omega', \ a.e. \ (\xi,y)\in Y,
		\end{equation}
		so that $\frac{\partial \hat{u}}{\partial y}\in L^2(W, \theta)$, $\nabla_{\xi} \hat{u} \in (L^2(W, \theta))^N$ and
		\begin{equation}
			\label{eqn:bochnerproductbeq3}
			\frac{\partial \hat{u}}{\partial y}(x,\xi,y)=\frac{\partial u(x)}{\partial y}(\xi,y), \ \ \ \nabla_{\xi}\hat{u}(x,\xi,y)=\nabla_{\xi}u(x)(\xi,y) \qquad \theta-a.e. \ x\in \Omega', \ a.e. \ (\xi,y)\in Y.
		\end{equation}
		The converse is also true, that is, if $\hat{u}$ belongs to the space in the right hand side of \eqref{eqn:bochnerproductbeq2}, then, $u$ defined as in \eqref{eqn:bochnerproducteqb} belongs to $L^2(\Omega',\theta;H^1(Y))$ and satisfies \eqref{eqn:bochnerproductbeq3}.
		In the same way $L^2(\Omega';H^1(Y)) = \{ \hat{u} \in L^2(W): \frac{\partial \hat{u}}{\partial y}\in L^2(W), \ \nabla_{\xi} \hat{u} \in (L^2(W))^N\}$,.
	\end{lemma}
	\begin{proof}
		\noindent\textbf{Step 1 (direct result): }Given $u\in L^2(\Omega',\theta;H^1(Y))$, take $\hat{u}$ as in \eqref{eqn:bochnerproducteqb}. The proof of the $\theta$-measurability and integrability of $\hat{u}$ is the same as in the proof of Lemma \ref{lemma:bochnerproduct}. Denote now $v(x,\xi,y) = \frac{\partial u(x)}{\partial y}(\xi,y)$ $\theta$-a.e. $x\in \Omega'$, a.e. $(\xi,y)\in Y$. With a similar argument, the $\theta$-measurability and integrability of $v$ is proved. Therefore, we only need to prove that $v=\frac{\partial \hat{u}}{\partial y}$. Take $\varphi \in C^\infty_c(W)$.
		 Then, as $\int_W\abs{\hat{u}\frac{\partial \varphi}{\partial y} \theta}<\infty$ by using Fubini's theorem
		\begin{equation}
			\begin{aligned}
				& \int_W \theta(x)\hat{u}(x,\xi,y)\frac{\partial \varphi}{\partial y}(x,\xi,y) dxd\xi dy = \int_{\Omega'} \theta(x)\int_Y\ u(x)(\xi,y)\frac{\partial \varphi}{\partial y}(x,\xi,y)d\xi dy dx 
				\\
				= -& \int_{\Omega'}\theta(x)\int_Y\frac{\partial u(x)}{\partial y}(\xi,y)\varphi(x,\xi,y)d\xi dy dx = \int_W \theta(x) v(x,\xi,y)\varphi(x,\xi,y)  dx d\xi dy
			\end{aligned}
		\end{equation}
		so $v=\frac{\partial \hat{u}}{\partial y}$. The argument for the $\xi-$derivatives is analogous.
		
		\noindent\textbf{Step 2 (converse result): } 
		Let $\hat{u}\in L^2(W,\theta)$ such that $\frac{\partial \hat{u}}{\partial y}\in L^2(W, \theta)$ and $\nabla_{\xi} \hat{u} \in (L^2(W, \theta))^N$. Define $u$ as in \eqref{eqn:bochnerproducteq1} and denote $v(x)(\xi,y)=\frac{\partial \hat{u}}{\partial y}(x,\xi,y)$ and $w(x)(\xi,y)=\nabla_{\xi} \hat{u}(x,\xi,y)$ for $\theta-$a.e. $x\in \Omega'$, a.e. $(\xi,y)\in Y$. By the same arguments as in the proof of Lemma \ref{lemma:bochnerproduct} we have that $u,v$ as functions $\Omega'\to L^2(Y)$ and $w$ as a function $\Omega'\to (L^2(Y))^N$, are $\theta$-measurable and $u,v\in L^2(\Omega',\theta;L^2(Y))$, $w\in (L^2(\Omega',\theta;L^2(Y)))^N$.
		
		Let us prove that, for $\theta$-a.e. $x\in \Omega'$, $v(x)=\frac{\partial u(x)}{\partial y}$. We take $\varphi \in C^\infty_c(\Omega')$ and $\psi\in C^\infty_c(Y)$. Then, by using Fubini's theorem
		\begin{equation}
			\begin{aligned}
				\int_{\Omega'}\theta(x)\varphi(x)\int_Y u(x)(\xi,y)\frac{\partial\psi}{\partial y}(\xi,y) d\xi dy dx = \int_W \varphi(x)\theta(x)\hat{u}(x,\xi,y)\frac{\partial \psi}{\partial y}(\xi,y)dxd\xi dy
				\\
				= -\int_W \theta(x)\varphi(x)\frac{\partial \hat{u}}{\partial y}(x,\xi,y)\psi(\xi,y)dxd\xi dy
				= -\int_{\Omega'}\theta(x)\varphi(x)\int_Y v(x)(\xi,y)\psi(\xi,y) d\xi dy dx.
			\end{aligned}
		\end{equation}
		As the above equality is true for every $\varphi \in C^\infty_c(\Omega')$, we have, for $\theta$-a.e. $x\in \Omega'$,
		\begin{equation}
			\int_Y u(x)(\xi,y)\frac{\partial\psi}{\partial y}(\xi,y) d\xi dy=-\int_Y v(x)(\xi,y)\psi(\xi,y) d\xi dy
		\end{equation}
		so $v(x)=\frac{\partial u(x)}{\partial y}$. In a similar way, it is proved that $w(x)=\nabla_{\xi} u(x)$.
		
		With the above arguments, we have proved that for $\theta$-a.e. $x\in \Omega'$, $u(x)\in H^1(Y)$ and \eqref{eqn:bochnerproductbeq3}. Let us prove that $u$ is $\theta$-measurable as a function $\Omega'\to H^1(Y)$. Again, as in the proof of Lemma \ref{lemma:bochnerproduct}, as $H^1(Y)$ is separable, we only need to prove the weak $\theta$-measurability. As $H^1(Y)$ is a reflexive space, for every $v\in H^1(Y)$, we need to check the $\theta-$measurability of the numerical function
		\begin{equation}
			x\mapsto \int_Y \left(u(x)(\xi,y)v(\xi,y)+\frac{\partial u(x)}{\partial y}\frac{\partial v}{\partial y}(\xi,y)+\nabla_{\xi}u(x)(\xi,y)\nabla_{\xi}v(\xi,y)\right)d\xi dy
		\end{equation}
		which is a straightforward consequence of Fubini's theorem because
		\begin{equation}
			\int_W \left(\hat{u}v+\frac{\partial \hat{u}}{\partial y}\frac{\partial v}{\partial y}+\nabla_{\xi} \hat{u}\nabla_{\xi} v\right) \theta \leq \left(\norm{\hat{u}}_{L^2(W,\theta)}+\norm{\frac{\partial \hat{u}}{\partial y}}_{L^2(W,\theta)}+\norm{\nabla_{\xi}\hat{u}}_{(L^2(W,\theta))^N}\right)\norm{v}_{H^1(Y)}<\infty.
		\end{equation}
		Finally, the $\theta$-(square-)integrability of $u$ as a function $\Omega'\to H^1(Y)$. follows from the fact that $u, v\in L^2(\Omega',\theta;L^2(Y))$ and $w\in (L^2(\Omega',\theta;L^2(Y))^N$.
	\end{proof}
	
	\begin{remark}
		\label{remark:bochnerproduct}
		Note that in Lemmas~\ref{lemma:bochnerproduct} and \ref{lemma:bochnerproductb} there is nothing specific about $\Omega'$, $\theta$, or $Y$ that is relevant for the proofs. In fact, the lemmas hold in a more general setting, but we have stated and proved it in this context to facilitate readability.
	\end{remark}

	\subsection{Density in Lebesgue-Bochner spaces}
	\label{app:density}
	
	The following theorems are density results of smooth functions in Bochner spaces. These results were used to derive the variational equations of the limit problem in Theorem \ref{thm:main} proved in Section~\ref{sec:jtxresults}. That is why we will use the notations of that section.
	
	In the following theorems, we will make a slight abuse of notation by comparing functions that belong to different functional spaces. Therefore, before stating the theorems, let us clarify what we mean by certain notations.
	
	We will deal with functions $C^\infty_c(\RN;H^1_{\braket{\xi}}(Y_e))$, for which we will study the norms $L^2(\Omega',\theta;H^1_{\braket{\xi}}(Y))$ and $H^1(\Omega^b)$. The first norm is straightforward to interpret. Note that, if $\varphi\in C^\infty_c(\RN;H^1_{\braket{\xi}}(Y))$, then $\varphi\in L^2(\Omega';H^1_{\braket{\xi}}(Y))$ and
	\begin{equation}
		\label{eqn:varphiinterpretHt}
		\norm{\varphi}_{L^2(\Omega',\theta;H^1_{\braket{\xi}}(Y))}=\int_{\Omega'}\theta(x)\norm{\varphi(x)}^2_{H^1_{\braket{\xi}}(Y)}dx \leq C \norm{\varphi}_{C(\RN;H^1_{\braket{\xi}}(Y_e))}.
	\end{equation}
	where, if we would like to be strictly rigorous, we would have to use
	$\restr{\varphi}{\Omega'}$, the restriction of $\varphi$ to $\Omega'$ and $\restr{\varphi(x)}{Y}$ the restriction of $\varphi(x)$ to $Y$, but we are using Convention 4. from Section \ref{sec:notations}.
	
	The norm $H^1(\Omega^b)$ may be more confusing. That is why we include the following lemma.
	\begin{lemma}
		\label{lemma:interpret}
		Let $\varphi\in C^\infty_c(\RN;H^1_{\braket{\xi}}(Y_e))$. Then, there exists $\hat{\varphi}\in H^1(\Omega^b)$ such that
		\begin{equation}
			\label{eqn:lemmainterpreteq0}
			\varphi(x)(\xi,y)=\hat{\varphi}(x,y) \qquad a.e \ (x,y)\in \Omega^b, \ a.e. \ \xi\in\omega.
		\end{equation}
		In addition,
		\begin{equation}
			\nabla_x \hat{\varphi}(x,y) = \nabla_x \varphi(x)(\xi,y), \quad \frac{\partial\hat{\varphi}}{\partial y}(x,y) = \frac{\partial \varphi(x)}{\partial y}(\xi,y) \qquad a.e \ (x,y)\in \Omega^b, \ a.e. \ \xi\in\omega.
		\end{equation}
		As a consequence,
		\begin{equation}
			\label{eqn:lemmainterpreteq0b}
			\norm{\hat{\varphi}}_{H^1(\Omega^b)}\leq \norm{\varphi}_{L^2(\RN;H^1_{\braket{\xi}}(Y_e))}+\norm{\nabla_x\varphi}_{L^2(\RN\times Y_e)} \leq C \norm{\varphi}_{C^1(\RN;L^2(Y_e))}
		\end{equation}
	\end{lemma}
	\begin{proof}
		First, since $Y_e=Y\cup(\omega\times (-R_0,0])$, for every $x\in \RN$, we will have $\varphi(x)\in H^1_{\braket{\xi}}(\omega\times (-R_0,0))$. Thus, using Proposition \ref{prop:constxi} and the fact that $\abs{\omega}=1$, we know that, for every $x\in \RN$, there exists $\tilde{\varphi}(x)\in H^1((-R_0,0))$ such that
		\begin{equation}
			\varphi(x)(\xi,y)=\tilde{\varphi}(x)(y), \quad  \frac{\partial \varphi (x)}{\partial y}(\xi,y)=\frac{d \tilde{\varphi}(x)}{d y}(y)  \quad a.e. \ y\in (-R_0,0)
		\end{equation}
		and
		\begin{equation}
			\label{eqn:lemmainterpreteq1}
			\norm{\tilde{\varphi}(x)}_{H^1((-R_0,0))}=\norm{\varphi(x)}_{H^1_{\braket{\xi}}(\omega\times(-R_0,0))}\leq \norm{\varphi(x)}_{H^1_{\braket{\xi}}(Y_e)}.
		\end{equation}
		Therefore, we can define
		$
			\hat{\varphi}(x,y)=\tilde{\varphi}(x)(y)$ for $x\in \RN, \ y\in (-R_0,0),
		$
		and its definition is consistent with \eqref{eqn:lemmainterpreteq0}. Now, let us prove that $\hat{\varphi}\in L^2(\Omega^b)$,
		\begin{equation}
			\label{eqn:lemmainterpreteq2}
			\begin{aligned}
				\int_{\Omega^b} \hat{\varphi}^2 &\leq \int_{\RN\times(-R_0,0)}\hat{\varphi}^2\leq
				\int_{\RN}\int_{-R_0}^0 \hat{\varphi}^2(x,y)dydx =
				\int_{\RN}\int_{-R_0}^0 \tilde{\varphi}^2(x)(y)dydx 
				\\
				&=
				\int_{\RN} \norm{\tilde{\varphi}(x)}^2_{L^2((-R_0,0))}dx \myleq{\eqref{eqn:lemmainterpreteq1}} \int_{\RN} \norm{\varphi(x)}^2_{H^1_{\braket{\xi}}(Y_e)}dx\leq \norm{\varphi}_{L^2(\RN;H^1_{\braket{\xi}}(Y))}<\infty.
			\end{aligned}
		\end{equation}
		where we have used Fubini's theorem in the first inequality.
		
		Now, let us study the weak differentiability of $\hat{\varphi}$. 
		We first define $w(x,y)\defeq \frac{\partial \varphi(x)}{\partial y}(y)$ for a.e. $x\in \RN$, $y\in (-R_0,0)$ and we will prove that $w=\frac{\partial \hat{\varphi}}{\partial y}$. In an analogous way as in \eqref{eqn:lemmainterpreteq2}, we obtain
		$
			\int_{\Omega^b} w^2 \leq \norm{\varphi}_{L^2(\RN;H^1_{\braket{\xi}}(Y))}<\infty.
		$
		Now, we take $\psi\in C^\infty_c(\Omega^b)$. We can extend $\psi$ by $0$ so $\psi\in C^\infty_c(\R^{N}\times (-R_0,0))$. Then,
		\begin{equation}
			\label{eqn:lemmainterpreteq4}
			\begin{aligned}
				\int_{\Omega^b} \hat{\varphi}\frac{\partial \psi}{\partial y} =\int_{\RN\times(-R_0,0)}\hat{\varphi}\frac{\partial \psi}{\partial y}=
				\int_{\RN}\int_{-R_0}^0 \hat{\varphi}(x,y)\frac{\partial \psi}{\partial y}(x,y)dydx  =
				\int_{\RN}\int_{-R_0}^0 \tilde{\varphi}(x)(y)\frac{\partial \psi}{\partial y}(x,y)dydx
			\end{aligned}
		\end{equation}
		where we could use Fubini's theorem because $\int_{\Omega^b} \abs{\hat{\varphi}\frac{\partial \psi}{\partial y}}\leq \norm{\hat{\varphi}}_{L^2(\Omega^b)}\norm{\psi_y}_{L^2(\Omega^b)}\ensuremath{\stackrel{\text{\eqref{eqn:lemmainterpreteq2}}}{<}}\infty$.
		Then, we use that, for fixed $x\in \RN$, $\psi(x,\cdot)\in C^\infty_c((-R_0,0))$, so
		\begin{equation}
			\label{eqn:lemmainterpreteq5}
			\begin{aligned}
				\int_{\RN}\int_{-R_0}^0 \tilde{\varphi}(x)(y)\frac{\partial \psi}{\partial y}(x,y)dydx = -\int_{\RN}\int_{-R_0}^0 \frac{\partial \tilde{\varphi}(x)}{\partial y}(y)\psi(x,y)dydx  = -\int_{\Omega^b}w\psi.
			\end{aligned}
		\end{equation}
		where we could use Fubini's theorem because $\int_{\Omega^b} \abs{w\psi}\leq \norm{w}_{L^2(\Omega^b)}\norm{\psi}_{L^2(\Omega^b)}<\infty$. Hence, from \eqref{eqn:lemmainterpreteq4}, \eqref{eqn:lemmainterpreteq5} and the arbitrariness of $\psi\in C^\infty_c(\Omega^b)$, we have that $\frac{\partial \hat{\varphi}}{\partial y}=w$. Finally, let us prove the weak $x-$differentiability of $\hat{\varphi}$. We define $v(x,y)\defeq \nabla_x\varphi(x)(y)$ for a.e. $(x,y)\in \Omega^b$.
		In an analogous way as we did in \eqref{eqn:lemmainterpreteq2}, we obtain
		$
			\int_{\Omega^b} v^2 \leq \norm{\nabla_x \varphi}_{L^2(\RN;(H^1_{\braket{\xi}}(Y))^N)}<\infty.
		$
		We take again $\psi\in C^\infty_c(\Omega^b)$ and extend it to $\psi\in C^\infty_c(\R^{N+1})$. 
		Then,
		\begin{equation}
			\label{eqn:lemmainterpreteq7}
			\begin{aligned}
				\int_{\Omega^b} \hat{\varphi}\nabla_x\psi & =\int_{\RN\times(-R_0,0)}\hat{\varphi}\nabla_x\psi=
				\int_{\RN}\int_{-R_0}^0 \hat{\varphi}(x,y)\nabla_x\psi(x,y)dydx  \\ & =
				\int_{\RN}\int_{-R_0}^0 \tilde{\varphi}(x)(y)\nabla_x\psi(x,y)dydx = \int_{\RN}\braket{ \tilde{\varphi}(x),\nabla_x\psi(x)}_{L^2((-R_0,0))}dx
			\end{aligned}
		\end{equation}
		where we could use Fubini's theorem because $\int_{\Omega^b} \abs{\hat{\varphi}\nabla_x\psi}\leq \norm{\hat{\varphi}}_{L^2(\Omega^b)}\norm{\nabla_x\psi}_{L^2(\Omega^b)}\ensuremath{\stackrel{\text{\eqref{eqn:lemmainterpreteq2}}}{<}}\infty$.
		Now, $x\mapsto \tilde{\varphi}(x)$ and $x\mapsto \psi(x,\cdot)$ are functions belonging to $C^\infty_c(\RN;L^2(-R_0,0))$ and $(f,g) \mapsto \braket{f,g}_{L^2((-R_0,0))}$ is a bilinear continuous form, so, using the chain rule, $$x\mapsto \braket{\tilde{\varphi}(x),\psi(x,\cdot)}_{L^2((-R_0,0))}, \qquad x\in\RN,$$ is a $C^\infty_c(\RN)$ function whose derivative is given by
		\begin{equation}
			x\mapsto \braket{\nabla_x \tilde{\varphi}(x),\psi(x,\cdot)}_{L^2((-R_0,0))}+\braket{ \tilde{\varphi}(x),\nabla_x\psi(x,\cdot)}_{L^2((-R_0,0))}.
		\end{equation}
		Hence, integrating in $\RN$, we obtain
		\begin{equation}
			\label{eqn:lemmainterpreteq8}
			\int_{\RN}\left(\braket{\nabla_x \tilde{\varphi}(x),\psi(x,\cdot)}_{L^2((-R_0,0))}+\braket{ \tilde{\varphi}(x),\nabla_x\psi(x,\cdot)}_{L^2((-R_0,0))}\right)dx=0.
		\end{equation}
		Therefore, combining \eqref{eqn:lemmainterpreteq7} and \eqref{eqn:lemmainterpreteq8} we obtain
		\begin{equation}
			\label{eqn:lemmainterpreteq9}
			\begin{aligned}
				\int_{\Omega^b} \hat{\varphi}\nabla_x\psi & \myeq{\eqref{eqn:lemmainterpreteq7}} \int_{\RN}\braket{ \tilde{\varphi}(x),\nabla_x\psi(x,\cdot)}_{L^2((-R_0,0))}dx \myeq{\eqref{eqn:lemmainterpreteq8}}
				-\int_{\RN}\braket{ \nabla_x\tilde{\varphi}(x),\psi(x,\cdot)}_{L^2((-R_0,0))}
				\\
				& = -\int_{\RN}\int_{-R_0}^0\nabla_x\tilde{\varphi}(x)(y)\psi(x,y) = -\int_{\RN}\int_{-R_0}^0 v\psi = -\int_{\Omega^b} v\psi
			\end{aligned}
		\end{equation}
		where we could use Fubini's theorem because $\int_{\Omega^b} \abs{v\psi}\leq \norm{v}_{L^2(\Omega^b)}\norm{\psi}_{L^2(\Omega^b)}<\infty$. Hence, from \eqref{eqn:lemmainterpreteq9} and the arbitrariness of $\psi\in C^\infty_c(\Omega^b)$, we have that $ \nabla_x\hat{\varphi}=v$.
	\end{proof}
	
	\begin{remark}
		\label{remark:interpret}
		Now, thanks to \eqref{eqn:varphiinterpretHt} and Lemma \ref{lemma:interpret}, we know how to interpret a function $\varphi\in C^\infty_c(\RN;H^1_{\braket{\xi}}(Y))$ in our space of solutions $H(\theta)$. Therefore, given any $\varphi \in C^\infty_c(\RN;H^1_{\braket{\xi}}(Y_e))$, when there is no risk of confusion, we will use the same notation $\varphi$ or $\varphi^a$ to refer to an $L^2(\RN;H^1_{\braket{\xi}}(Y))$ which is just the restriction of $\varphi$ to $\RN\times Y$ and we will use the same notation $\varphi$ or $\varphi^b$ to refer to an $H^1(\Omega^b)$ function given by the $\hat{\varphi}$ constructed in the Lemma \ref{lemma:interpret}. In particular, the $H^1(\Omega^b)$ of $\varphi$ will be given by
		\begin{equation}
			\label{eqn:defh1cinf}
			\norm{\varphi}_{H^1(\Omega^b)}\defeq \norm{\hat{\varphi}}_{H^1(\Omega^b)},
		\end{equation}
		where $\hat{\varphi}$ is as in Lemma \ref{lemma:interpret}.
		
	\end{remark}
	
	The following theorem states that any function belonging to $C^\infty_c(\RN;H^1_{\braket{\xi}})$ can be approximated with the topology of $H(\theta)$ by sum of separable functions, i.e. functions of the form $\sum_i\phi_i(x)\psi_i(\xi,y)$ where $\phi_i\in C^\infty_c(\RN)$ and $\psi_i\in H^1_{\braket{\xi}}(Y_e)$.

	\begin{theorem}
		\label{thm:density2}
		Given $\varphi \in C^\infty_c(\RN; H^1_{\braket{\xi}}(Y_e))$ and $\epsilon>0$, we can find a finite number of functions $\{\phi_i\}_{i=1}^k\subset C^\infty_c(\RN)$ and $\{\psi_i\}_{i=1}^k\subset H^1_{\braket{\xi}}(Y_e)$ such that
		\begin{equation}
			\label{eqn:densitypropeq1}
			\normdos{\varphi-\sum_{i=1}^k \phi_i\cdot \psi_i}_{L^2(\Omega',\theta;H^1_{\braket{\xi}}(Y))}+\normdos{\varphi-\sum_{i=1}^k \phi_i\cdot \psi_i}_{H^1(\Omega^b)}\leq \epsilon
		\end{equation} 
		Recall that the $L^2(\Omega',\theta;H^1_{\braket{\xi}}(Y))$ and $H^1(\Omega^b)$ norm for these functions is given in Remark \ref{remark:interpret}.
		
		Equation \eqref{eqn:densitypropeq1} means, in other words, that the set of finite linear combination of product of functions $\phi\psi$, where $\phi\in C^\infty(\adh{\Omega})$ and $\psi \in H^1_{\braket{\xi}}(Y_e)$ is dense in $C^\infty_c(\RN; H^1_{\braket{\xi}}(Y_e))$ with the $H(\theta)$ topology.
	\end{theorem}

	\begin{proof}
		The proof follows from the fact that, as $H^1_{\braket{\xi}}(Y_e)$ is a Hilbert space, it satisfies the approximation property (see \cite[Section III.9]{MR1741419}). This is the following: for every $K\subset H^1_{\braket{\xi}}(Y_e)$ compact and $\epsilon>0$, there exists a linear operator of finite range
		\begin{equation}
			\begin{aligned}
				T: \ \ & H^1_{\braket{\xi}}(Y_e) && \longrightarrow &&& H^1_{\braket{\xi}}(Y_e) \ \ \\
				\ & \ \ \ \psi && \ \mapsto &&& \sum_{i=1}^k \alpha_i(\psi)\psi_i,
			\end{aligned}
		\end{equation}
		where $\alpha_i \in (H^1_{\braket{\xi}}(Y_e))^*$ for every $i=1,\dots, k$, such that
		\begin{equation}
			\label{eqn:density2eq2}
			\norm{T\psi - \psi}_{H^1_{\braket{\xi}}(Y_e)} \leq \epsilon.
		\end{equation}
		Now, take $K_0=\varphi(\RN)\subset H^1_{\braket{\xi}}(Y_e)$ which is a compact set, because $\varphi$ is continuous and of compact support. Consider also $K_1=\{\nabla_x\varphi(x)(v) : x\in \RN, \ v\in \RN, \ \abs{v}=1\}\subset H^1_{\braket{\xi}}(Y_e)$ which is compact for the same reason. Let $K=K_0\cup K_1$ and $\epsilon>0$ and consider the operator $T$ from the approximation property. Define
		\begin{equation}
			\varphi_\epsilon(x) = T(\varphi(x))=\sum_{i=1}^k \alpha_i(\varphi_k(x))\psi_i.
		\end{equation}
		Note that $\varphi_\epsilon$ is a function of the form $\sum_{i=1}^k\phi_i\psi_i$ where $\phi_i\in C^\infty_c(\RN)$ because the scalar product is linear and $\varphi\in C^\infty_c(\RN)$. Now, for every $x\in \RN$, $\varphi(x)\in K_0\subset K$ so
		\begin{equation}
			\label{eqn:density2eq3}
			\norm{\varphi_\epsilon(x)-\varphi(x)}_{H^1_{\braket{\xi}}(Y_e)}=\norm{T(\varphi(x))-\varphi}_{H^1_{\braket{\xi}}(Y_e)}\myleq{\eqref{eqn:density2eq2}} \epsilon
		\end{equation}
		In the same way, for every $x\in \RN$ and $v\in \RN$ with $\abs{v}=1$, as $\nabla_x\varphi(x)(v)\in K_1\subset K$ and $T$ is a linear operator,
		\begin{equation}
			\label{eqn:density2eq4}
			\norm{\nabla_x\varphi_\epsilon(x)(v)-\nabla_x\varphi(x)(v)}_{H^1_{\braket{\xi}}(Y_e)} = \norm{T(\nabla_x\varphi(x)(v))-\nabla_x\varphi(x)(v)}_{H^1_{\braket{\xi}}(Y_e)} \leq \epsilon.
		\end{equation}
		The combination of \eqref{eqn:density2eq2} and \eqref{eqn:density2eq3} leads to $\norm{\varphi_\epsilon-\varphi}_{C^1(\RN;H^1_{\braket{\xi}}(Y_e))}\myleq{\eqref{eqn:density2eq2}} \epsilon$. The inequality \eqref{eqn:densitypropeq1} then follows from \eqref{eqn:varphiinterpretHt} and \eqref{eqn:lemmainterpreteq0b} by choosing $\epsilon$ smaller if necessary.
	\end{proof}

	\begin{theorem}
		\label{thm:density}
		The space of test functions $C^\infty_c(\RN;H^1_ {\left<\xi\right>}(Y_e))$ is dense in $H(\theta)$ in the sense that, for every $(v^a,v^b)\in H(\theta)$ and $\epsilon>0$, there exists
		$v_\epsilon\in C^\infty_c(\RN;H^1_{\braket{\xi}}(Y_e))$ such that 
		\begin{equation}
			\norm{v^a-v_\epsilon}_{L^2(\Omega',\theta;H^1_{\braket{\xi}}(Y))}\leq \epsilon
		\end{equation}
		and
		\begin{equation}
			\norm{v^b-v_\epsilon}_{H^1(\Omega^b)}\leq \epsilon
		\end{equation}
		where $v_\epsilon$ is interpreted as an $L^2(\Omega',\theta;H^1_{\braket{\xi}}(Y))$ or $H^1(\Omega^b)$ function via Remark \ref{remark:interpret}. 
	\end{theorem}
	\begin{proof}
		We want to approximate $(v^a,v^b)\in H(\theta)$ by a $C^\infty_c(\RN;H^1_ {\left<\xi\right>}(Y_e))$. We will split the proof into six steps. First of all, we will make the proof assuming
		that $v^a\in L^2(\Omega';H^1_ {\left<\xi\right>}(Y))\subset L^2(\Omega',\theta;H^1_ {\left<\xi\right>}(Y))$, and in Step 6, we will generalize to the case $v^a\in  L^2(\Omega',\theta;H^1_ {\left<\xi\right>}(Y))$.
		
		\noindent\textbf{\underline{Step 1}: }
		As $\Omega^b$ is Lipschitz, we can extend $v^b$ to $\tilde{v}^b\in H^1(\RN\times(-R_0,0))$ with compact support (see e.g. \cite[Theorem 4.7]{evansmeasure}). In the same way, we can extend $v^a$ as follows.
		\begin{equation}
			\tilde{v}^a (x,{\xi},y) = \left\{
			\begin{aligned}
				& v^a(x,{\xi},y) \qquad && (x,{\xi},y)\in W \\
				& {Tr}^b(v^b)(x) \qquad && x\in \RN\setminus \Omega', \ ({\xi},y)\in Y
			\end{aligned}
			\right.
		\end{equation}
		where ${Tr}^b(v^b)\in L^2(\RN)$ denotes the trace of $v^b$ in $\{(x,0)\in \R^{N+1}:x\in \RN\}$. Then, $\tilde{v}^a\in L^2(\RN; H^1_ {\left<\xi\right>}(Y))$, so without loss of generality we can assume $\Omega'=\RN$, $\Omega^b=\RN\times(-R_0,0)$ and $v^a, v^b$ are compactly supported in the $x$ component. From now on we will make that assumption.
		
		\noindent\textbf{\underline{Step 2}: }
		Define
		\begin{equation}
			v(x,\xi,y)\defeq \left\{
			\begin{aligned}
				& v^a(x,\xi,y) \qquad && x\in \RN, (\xi,y)\in Y \\
				& \restr{v}{y<0}(x,\xi,y)\defeq v^b(x,y) \qquad && x\in \RN, \ \xi\in \omega,  \ y\in(-R_0,0).
			\end{aligned}
			\right.
		\end{equation}
		Let us prove that $v\in L^2(\RN; H^1_{\braket{\xi}}(Y_e))$. We start by proving that, for a.e. $x\in\RN$, $v(x,\cdot,\cdot\cdot)\in H^1_{\braket{\xi}}(Y_e)$. For a.e. $x\in \RN$, $v^a\in H^1_{\braket{\xi}}(Y)$ and $v^b(x,\cdot)\in H^1((-R_0,0))$. In particular, due to Proposition \ref{prop:constxi}, we have that $\restr{v}{y<0}\in H^1_{\braket{\xi}}(\omega\times(-R_0,0))$. As $(v^a,v^b)\in H(\theta)$, for a.e. $x\in \RN$, we have ${Tr}^a_{\braket{\xi}}(v^a)(x)={Tr}^b(v^b)(x)$, so for a.e. $x\in \RN$ and $\xi\in \omega$, ${Tr}^a(v^a)(x,\xi)={Tr}^b(\restr{v}{y<0})(x,\xi)$. Therefore, by using Lemma \ref{lemma:h1cont2}, we have that, for a.e. $x\in \RN$, $v(x,\cdot,\cdot \cdot)\in H^1_{\xi}(Y_e)$.
		
		Now, by using Fubini's theorem, 
		\begin{equation}
			\begin{aligned}
				\int_{\RN} & \norm{v(x)}_{L^2(Y_e)} = \int_{\RN}\int_Y (v^a(x,\xi,y))^2d\xi dy dx + \int_{\RN}\int_{\omega\times(-R_0,0)}(\restr{v}{y<0}(x,\xi,y))^2d\xi dy dx
				\\
				= & \norm{v^a}^2_{L^2(\RN;L^2(Y))}+\int_\RN \int_{-R_0}^0 (v^b(x,y))^2dy dx = \norm{v^a}^2_{L^2(\RN;L^2(Y))} + \norm{v^b}^2_{L^2(\RN\times(-R_0,0))}<\infty.
			\end{aligned}
		\end{equation}
		so $v\in L^2(\RN;L^2(Y_e))$. Using the same argument with the $y$-derivatives and $\xi$-derivatives, we obtain $v\in L^2(\RN;H^1_{\braket{\xi}}(Y_e))$.

		\noindent\underline{\textbf{Step 3:}} 
		We consider the set of mollifiers,
		\begin{equation}
			\eta_\epsilon(x)\defeq \left\{
			\begin{aligned}
				& \frac{c}{\epsilon^N}\exp\left(\frac{1}{\abs{\frac{x}{\epsilon}}^2-1}\right) \qquad && \abs{x}\leq \epsilon, \\
				& 0 \qquad && \abs{x}\geq \epsilon,
			\end{aligned}
			\right. \quad x\in\RN,
		\end{equation}
		where $\epsilon>0$ and $c>0$ is a normalization constant so that $\int_\RN \eta_\epsilon=1$. 
		Now, define
		\begin{equation}
			\label{eqn:densityeq4}
			v_\epsilon(x,\xi,y) = \int_{\RN}\eta_\epsilon(x-\bar x) v(\bar x,\xi,y)d\bar x, \qquad x\in \RN, \ (\xi,y)\in Y_e.
		\end{equation}
		Let us see that $v_\epsilon\in L^2(\RN; H^1_{\braket{\xi}}(Y_e))$.
		First, by using Fubini's theorem as well as Young's convolution inequality, we have		
		\begin{equation}
			\label{eqn:densityeq1}
			\begin{aligned}
				\int_{\RN}\int_{Y_e} & \abs{\int_{\RN} \eta_\epsilon(x-\bar x) v(x,\xi,y)d\bar x}^2 d\xi dy d x = \int_{Y_e}\left(\int_{\RN}  \abs{\int_{\RN} \eta_\epsilon(x-\bar x) v(x,\xi,y)d\bar x}^2 d x\right)d\xi dy 
				\\
				& \ \myleq{Young's} 
				\int_{Y_e}\left(\int_{\RN} \eta_\epsilon(x) dx\right)^2 \left(\int_{\RN}\abs{v(\bar x,\xi, y)}^2d\bar x\right) d\xi dy = \norm{v}^2_{L^2(\RN;L^2(Y_e))}<\infty.
			\end{aligned}
		\end{equation}
		where we have used $\int_{\RN} \eta_\epsilon(x) dx=1$. Then, $v_\epsilon\in L^2 (\RN; L^2(Y_e))$ and
		\begin{equation}
			\label{eqn:densityeq5}
			\norm{v_\epsilon}_{L^2 (\RN; L^2(Y_e))}\leq \norm{v}^2_{L^2(\RN;L^2(Y_e))}.
		\end{equation}
		We define $w(x,\xi,y)=\int_\RN \eta_\epsilon (x-\bar x)\frac{\partial v}{\partial y}(\bar x,\xi,y)$ and try to prove that $w(x)=\frac{\partial v_\epsilon}{\partial y}$. Using the same procedure as in \eqref{eqn:densityeq1} we have $w\in L^2 (\RN; L^2(Y_e))$, so $w(x)\in L^2(Y_e)$ for a.e. $x\in \RN$. Now, take $\psi\in C^\infty_c(Y_e)$. Then, for a.e. $x\in \RN$,
		\begin{equation}
			\begin{aligned}
				& \int_{Y_e}v_\epsilon(x,\xi,y)\psi_y(\xi,y)d\xi dy =
				\int_{Y_e}\left(\int_{\RN} \eta_\epsilon(x - \bar x)v(\bar x,\xi,y)d\bar x\right)\psi_y(\xi,y) d\xi dy 
				\\
				= & \int_{\RN} \eta_\epsilon(x - \bar x)\left(\int_{Y_e}v(\bar x,\xi,y)\psi_y(\xi,y) d\xi dy\right) d\bar x 
				=  \int_{\RN} \eta_\epsilon(x - \bar x)\left(\int_{Y_e}v_y(\bar x,\xi,y)\psi(\xi,y) d\xi dy\right) d\bar x 
				\\ =  & \int_{Y_e} \left(\int_{\RN}\eta_\epsilon(x - \bar x)v_y(\bar x,\xi,y)d\bar x\right)\psi(\xi,y)  d\xi dy 
				= \int_{Y_e}w(x,\xi,y)\psi(\xi,y)d\xi dy
			\end{aligned}
		\end{equation}
		where we could use Fubini's theorem because $\int_{Y_e} \abs{v_\epsilon(x)\psi_y}\leq \norm{v_\epsilon(x)}_{L^2(Y_e)}\norm{\psi_y}_{L^2(Y_e)}<\infty$ and \\ $\int_{Y_e} \abs{w(x)\psi_y}\leq \norm{w(x)}_{L^2(Y_e)}\norm{\psi}_{L^2(Y_e)}<\infty$ for a.e. $x\in \RN$. Therefore $w(x)=\frac{\partial v_\epsilon (x)}{\partial y}$ for a.e. $x\in \RN$. In an analogous way it is proved the case of the $\xi$-derivatives and we obtain $v_\epsilon\in L^2(\RN;H^1_{\braket{\xi}}(Y_e))$. Moreover, from equation \eqref{eqn:densityeq5}  together with the analogous equations that can be obtained for the derivatives in $\xi$ and $y$, one obtains
		\begin{equation}
			\label{eqn:densityeq6}
			\norm{v_\epsilon}_{L^2 (\RN; H^1_{\braket{\xi}}(Y_e))}\leq \norm{v}^2_{L^2(\RN;H^1_{\braket{\xi}}(Y_e))}.
		\end{equation}
		
		Finally, let us prove that $v_\epsilon\in C^\infty_c(\RN;H^1_{\braket{\xi}}(Y_e))$. Define $w(x,\xi,y) = \int_\RN (\eta_\epsilon)_{x_i}(x-\bar x)v(\bar x, \xi, y)d\bar x$. Let us prove that $w(x)=\frac{\partial v_\epsilon}{\partial x_i}(x)$ for every $x\in \RN$. Given $x\in \RN$,
		\begin{equation*}
			\begin{aligned}
				& \bnorm{\frac{v_\epsilon(x+h e_i)-v_\epsilon(x)}{h}-w(x)}_{L^2(Y_e)} \\
				= & \int_{Y_e}\abs{ \int_{\RN}\left(\frac{\eta_\epsilon(x-\bar x+h e_i)-\eta_\epsilon(x-\bar x)}{h}-(\eta_\epsilon)_{x_i}(x-\bar x)\right)v(\bar x, \xi, y)d\bar x}^2 d\xi dy
				\\
				\leq 
				& \sup_{z\in\RN} \left\{\frac{\eta_\epsilon(z+h e_i)-\eta_\epsilon(z)}{h}-(\eta_\epsilon)_{x_i}(z)\right\}\norm{v}_{L^2(\RN\times Y_e)}\to 0
			\end{aligned}
		\end{equation*}
		when $h\to 0$. The same can be done with the $y$-derivatives and $\xi$-derivatives so one obtains
		\begin{equation*}
			\bnorm{\frac{v_\epsilon(x+h e_i)-v_\epsilon(x)}{h}-w(x)}_{H^1_{\braket{\xi}}(Y_e)} \to 0
		\end{equation*}
		when $h\to 0$. Therefore, $v_\epsilon\in C^1(\RN; H^1_{\braket{\xi}}(Y_e))$ and $\frac{\partial v}{\partial x_i}(x)(\xi,y)=\int_\RN (\eta_\epsilon)_{x_i}(x-\bar x)v(\bar x,\xi,y)$. The infinite differentiability is obtained iterating this procedure.
		
		\textbf{\underline{Step 4}: }Let us see that $v_\epsilon\to v^a$ in $L^2(\RN; H^1_ {\left<\xi\right>}(Y))$. 
		First, using Cauchy-Schwartz inequality and Fubini's theorem,
		\begin{equation}
			\label{eqn:densityeq2}
			\begin{aligned}
				& \norm{v_\epsilon(x)-v^a(x)}^2_{L^2(Y)} = 
				\int_{Y}\left(\int_{\RN}\eta_\epsilon(x-\bar x)(v^a(\bar x, y, \xi)-v^a(x, y, \xi))d\bar x\right)^2 d\xi dy
				\\
				\leq & \int_Y \left(\int_{\RN}\eta(x-\bar x)d\bar x\right)^2\left(\int_{\RN}\eta_\epsilon(x-\bar x)(v^a(\bar x, y, \xi)-v^a(x, y, \xi))^2d \bar x\right)d\xi dy
				\\
				\leq &  \int_{\RN}\eta_\epsilon(x-\bar x)\int_Y(v^a(\bar x, y, \xi)-v^a(x, y, \xi))^2 d\xi dy d\bar x \leq \int_{\RN}\eta_\epsilon(x-\bar x) \norm{v^a(\bar x)-v^a(x)}_{L^2(Y)}
			\end{aligned}
		\end{equation}
		Suppose $v^a\in C(\RN;H^1_ {\left<\xi\right>}(Y))$. Then, given $\delta>0$, we can find $\epsilon>0$ small enough such that, for every $\abs{x-\tilde{x}}\leq \epsilon$, $\norm{v^a(x)-v^a(\tilde{x})}_{H^1_ {\left<\xi\right>}(\RN)}\leq \delta$. Therefore, making a change of variables $z=\frac{x-\bar x}{\epsilon}$ and using that $\eta$ is supported in $B_\RN=B_{\RN}(0,1)$,
		\begin{equation}
			\label{eqn:densityeq3}
			\int_{\RN}\eta_\epsilon(x-\bar x) \norm{v^a(\bar x)-v^a(x)}_{L^2(Y)} = \int_{B_\RN}\eta(z) \norm{v^a(x-\epsilon z)-v^a(x)}_{L^2(Y)}dz \leq \delta.
		\end{equation}
		
		Hence, from \eqref{eqn:densityeq2} and \eqref{eqn:densityeq3}, we obtain that $v^a_\epsilon$ converges to $v^a$ in $C(\RN;L^2(Y))$ when $\epsilon\to 0$. Using an analogous argument for the $y$-derivatives and $\xi$-derivatives one obtains the convergence in $C(\RN;H^1_ {\left<\xi\right>}(Y))$. In particular, as both functions are compactly supported, $v^a_\epsilon\to v^a$  in $L^2(\RN;H^1_ {\left<\xi\right>}(Y)))$. 
		
		Now, let us treat the case $v^a\notin C(\RN;H^1_ {\left<\xi\right>}(Y))$. By Proposition \ref{prop:approxc0} (see below), we can find an approximation $\tilde{v}^a\in C_c(\RN;H^1_ {\left<\xi\right>}(Y))$ such that $\norm{\tilde{v}^a-v^a}_{L^2(\RN;H^1(Y))}\leq \delta$. 
		
		Then, denoting $\tilde{v}^a_\epsilon$ the convolution of $\eta_\epsilon$ with $\tilde{v}^a$, we have $\norm{\tilde{v}^a_\epsilon- \tilde{v}^a}_{L^2(\RN;H^1_ {\left<\xi\right>}(Y))}\to 0$ when $\epsilon\to 0$. In addition, by the linearity of convolution $v_\epsilon-\tilde{v}^a_\epsilon=(v-\tilde{v^a})_{\epsilon}$, where $(v-\tilde{v^a})_{\epsilon}$ denotes the convolution of $v-\tilde{v^a}$ with $\eta_\epsilon$. Then, applying \eqref{eqn:densityeq6} to this function we have 
		\begin{equation}
			\norm{v_\epsilon-\tilde{v}^a_\epsilon}_{L^2 (\RN; H^1_{\braket{\xi}}(Y_e))}=\norm{(v-\tilde{v}^a)_\epsilon}_{L^2 (\RN; H^1_{\braket{\xi}}(Y_e))}\leq \norm{v-\tilde{v}^a}_{L^2 (\RN; H^1_{\braket{\xi}}(Y_e))}\leq \delta
		\end{equation}
		Then,
		\begin{equation}
			\begin{aligned}
				\norm{v_\epsilon- v^a}_{L^2(\RN;H^1_ {\left<\xi\right>}(Y))} & \leq \norm{v_\epsilon- \tilde{v}^a_\epsilon}_{L^2(\RN;H^1_ {\left<\xi\right>}(Y))}+\norm{\tilde{v}^a_\epsilon- \tilde{v}^a}_{L^2(\RN;H^1_ {\left<\xi\right>}(Y))}+\norm{\tilde{v}^a- v^a}_{L^2(\RN;H^1_ {\left<\xi\right>}(Y))} \\
				& \leq \norm{\tilde{v}^a_\epsilon- \tilde{v}^a}_{L^2(\RN;H^1_ {\left<\xi\right>}(Y))} +2\delta \to 2\delta
			\end{aligned}
		\end{equation}
		when $\epsilon\to 0$. As $\delta>0$ was arbitrary, we have $\norm{v^a_\epsilon- v^a}_{L^2(\RN;H^1_ {\left<\xi\right>}(Y))}\to 0$ when $\epsilon\to 0$.
		
		\textbf{\underline{Step 5}: }Let us prove that $v^b_\epsilon\to v^b$ in $H^1(\RN\times(-R_0,0))$. As $v^b_\epsilon\in C^\infty_c(\RN;H^1(Y_e))$, to interpret this convergence, we need to use Lemma \ref{lemma:interpret} which guarantees that there exists $\hat{v}_\epsilon^b\in H^1(\RN\times(-R_0, 0))$ such that
		\begin{equation}
			v_\epsilon(x,\xi, y)= \hat{v}_\epsilon^b(x,y) \qquad x\in \RN, \ y\in (-R_0,0), \ \xi\in \omega.
		\end{equation}
		However, from definition \eqref{eqn:densityeq4}, we immediately see that
		\begin{equation}
			\hat{v}_\epsilon^b(x,y)=\int_{\RN}\eta_\epsilon(x-\bar x)v^b(\bar x, y)d\bar x dy, \qquad x\in \RN, \ y\in (-R_0,0).
		\end{equation}
		Then, it is a standard result that, as $v^b\in H^1(\RN\times(-R_0,0))$, the convolution with a mollifier $\hat{v}_\epsilon^b$ converges to $v^b$ in $H^1(\RN\times(-R_0,0))$ (see for example \cite[Appendix C.4]{evans}).

		\textbf{\underline{Step 6}: } Finally, we treat the case $v^a\in L^2(\Omega',\theta;H^1_ {\left<\xi\right>}(Y))$. As $v^b\in H^1_ {\left<\xi\right>}(\Omega^b)$, its trace on $\Omega'\times\{0\}$, ${Tr}^b(v^b)$ belongs to $L^2(\Omega')$. Then, we split $(v^a,v^b)$ into two functions $(v^a,v^b)=(v_1^a,v_1^b)+(v_2^a,v_2^b)$. First, we define
		\begin{equation}
			\left\{
			\begin{aligned}
				& v_1^a(x,\xi,y)={Tr}^b(v^b)(x) \qquad && (x,{\xi},y)\in W \\
				& v_1^b(x,\xi,y)=v^b(x,y) \qquad && (x,{\xi},y)\in \Omega^b.
			\end{aligned}
			\right.
		\end{equation}
		so $(v_1^a,v_1^b)\in H(\theta)$, and
		\begin{equation}
			\left\{
			\begin{aligned}
				& v_2^a(x,\xi,y)=v^a(x,{\xi},y)-{Tr}^b(v^b)(x) \qquad && (x,{\xi},y)\in W \\
				& v_2^b(x,\xi,y)=0 \qquad && (x,y)\in \Omega^b.
			\end{aligned}
			\right.
		\end{equation}
		so $(v_2^a,v_2^b)\in H(\theta)$. In addition, $v_1^a$ belongs to $L^2(\Omega';H^1_ {\left<\xi\right>}(Y))$ (note there is no weight $\theta$), so, by the previous steps, we can approximate $(v_1^a,v_1^b)$ by a $C^\infty_c(\RN;H^1_ {\left<\xi\right>}(Y_e))$ function. Hence, we only need to approximate $(v_2^a,v_2^b)$. 
		
		To do this, we will use an argument similar to the one of the proof of Lemma \ref{lemma:l2dense}. We define $A_\delta \defeq \{x\in\Omega': \theta(x)\geq \delta\}$. Note that, for any $\delta>0$, $v\chi_{A_\delta}\in H(\theta)$. In addition, we know that $\cup_{\delta>0} A_\delta = \Omega'\setminus \Theta_0$ (recall that $\Theta_0=\{x\in \Omega': \theta(x)=0\}$), so, given $\epsilon>0$, we can choose $\delta>0$ small enough such that 
		\begin{equation}
			\norm{v_2^a(1-\chi_{A_\delta})}_{L^2(\Omega',\theta;H^1_ {\left<\xi\right>}(Y))}=\norm{v_2^a(1-\chi_{A_\delta})}_{L^2(\Omega',\theta;H^1_ {\left<\xi\right>}(Y))}\leq \epsilon.
		\end{equation}
		As $(v_2^a\chi_{A_\delta},0)$ belongs to $H(\theta)$ too, we only need to approximate $(v_2^a\chi_{A_\delta},0)$. But $v^a_2\chi_{A_\delta}$ belongs to $L^2(\Omega';H^1_ {\left<\xi\right>}(Y))$ because
		\begin{equation}
			\norm{v^a_2}_{L^2(\Omega';H^1_ {\left<\xi\right>}(Y))}\int_{A_\delta} \norm{v^a_2(x)}_{H^1_ {\left<\xi\right>}(Y)}\leq \frac{1}{\delta} \int_{A_\delta}\theta(x)\norm{v^a_2(x)}_{H^1_ {\left<\xi\right>}(Y)}=\frac{\norm{v^a_2}_{L^2(\Omega';H^1_ {\left<\xi\right>}(Y))}}{\delta}<\infty,
		\end{equation}
		so we can approximate $(v_2^a\chi_{A_\delta},0)$ by a $C^\infty_c(\RN;H^1_ {\left<\xi\right>}(Y_e))$ function by the previous steps and the proof is finished.
	\end{proof}

	Combining Theorem \ref{thm:density2} and Theorem \ref{thm:density}, we obtain the following corollary. In this case, we use the notation of Section \ref{sec:test}, as the result will be used with the appropriate test functions explained in that section.
	\begin{corollary}
		\label{cor:density}
		Let $(\varphi^a,\varphi^b)\in H(\theta)$ and $\delta>0$. We can find a finite number of functions $\{\phi_i\}_{i=1}^k\subset C^\infty_c(\RN)$ and $\{\psi_i\}_{i=1}^k\subset H^1_{\braket{\xi}}(Y_e)$ such that
		\begin{equation}
			\label{eqn:densitycoreq1}
			\normdos{\varphi^a-\sum_{i=1}^k \phi_i\cdot \psi^a_i}_{L^2(\Omega',\theta;H^1_{\braket{\xi}}(Y))}+\normdos{\varphi^b-\sum_{i=1}^k \phi_i\cdot \psi^b_i}_{H^1(\Omega^b)}\leq \delta
		\end{equation} 
		where $\psi^a_i\in H^1_{\braket{\xi}}(Y)$ and $\psi^b_i\in H^1((-R_0,0))$ follow the notation of \eqref{eqn:psia} and \eqref{eqn:psib} from Section~\ref{sec:test}. In other words, the set of finite linear combination of product of functions $\phi\psi$, where $\phi\in C^\infty_c(\RN)$ and $\psi \in H^1_{\braket{\xi}}(Y_e)$ is dense in $H(\theta)$.
	\end{corollary}
	
	We present the Proposition we needed to use in Theorem \ref{thm:density}
	\begin{prop}
		\label{prop:approxc0}
		Let ${\hil}$ be a separable Hilbert space and $f\in L^2(\RN; {\hil})$. Then, for every $\epsilon>0$, there exists $g\in C_c(\RN;{\hil})$ such that
		\begin{equation}
			\norm{f-g}_{L^2(\RN;{\hil})}\leq \epsilon.
		\end{equation}
	\end{prop}
	\begin{proof}
		By definition of Bochner integrable function, we know that simple functions are dense in $L^2(\RN; {\hil})$. As simple functions are linear combinations of indicator functions, we only need to approximate such these functions. Now, take an indicator function $f=\chi_E x$ where $E\subset\RN$ bounded and $x\in {\hil}$. If $\norm{x}_{\hil}=0$ we are done. If not, we can find a continuous function $\tilde{g}\in C^\infty_c(\RN)$ such that $\norm{\chi_E-\tilde{f}}_{L^2(\RN)}\leq \epsilon/\norm{x}$. Define then $g=\tilde{g}x$. Then, $\norm{f-g}_{L^2(\RN;{\hil})}\leq \norm{x}_{\hil}\norm{\chi_E-\tilde{g}}_{L^2(\RN)}\leq\epsilon$.
	\end{proof}

	\section{Other auxiliary results}

	\subsection{$H^1(\Omega)$ and continuity}
	The following results concern when an $H^1$ function defined on an open set minus some hyperplanes can be extended to the entire open set.
	
	\begin{lemma}
		\label{lemma:h1cont2}
		Let $\Omega\subset \R^{N+1}$, $y_0\in \R$ and denote $\Omega^a=\{(x,y)\in \R^{N+1}: y>y_0\}$ and $\Omega^b=\{(x,y)\in \R^{N+1}: y<y_0\}$. Let $u:\Omega \to \R$ such that $u\in H^1(\Omega^a)$ and $u\in H^1(\Omega^b)$. Given any ball $B=B((x,y_0),r)\subset \Omega$ and letting $B^a=\{(x,y)\in B:y>y_0\}$ and $B^b=\{(x,y)\in B:y>y_0\}$, we can define $u^a\in H^1(B^a)$ and $u^b\in H^1(B^b)$. As $B^a$ and $B^b$ are bounded Lipschitz domains, they have a well-defined trace in $B^0=\{(x,y_0)\in B\}$ that we denote $tr(u^a)$ and $tr(u^b)$. Then, if for every $(x,0)\in Y$, there exists a ball $B=B((x,y_0),r)\subset \Omega$ such that
		\begin{equation}
			tr(u^a)\equiv tr(u^b) \qquad \text{on} \ B^0,
		\end{equation}
		then $u\in H^1(\Omega)$ and $\norm{u}_{H^1(\Omega)}=\norm{u}_{H^1(\Omega^a)}+\norm{u}_{H^1(\Omega^b)}$.
	\end{lemma}
	\begin{proof}
		The weak differentiability is a local property, so we only have to prove the differentiability in a neighbourhood of every point of $\Omega$. Given $(x,y)\in \Omega$, if $x\in \Omega^a$, we can take a ball $B\subset \Omega^a$ such that $x\in B$. But then, as $u\in H^1(\Omega)$, we have $u\in H^1(\Omega^a)$. The case in which $x\in \Omega^b$ is analogous. Therefore, let us assume $(x,y)=(x,y_0)$ and take a ball $B=B((x,y_0),r)\subset \Omega$ such that $tr(u^a)\equiv tr(u^b)$ on $B^0$. Now, take $\varphi\in C^\infty_c(B)$. As $B^a$ and $B^b$ are bounded Lipschitz domains, we can use the integration by parts formula (see \cite[Theorem 4.6]{evansmeasure}). Then for $i=1,\dots, N$,
		\begin{equation}
			\int_B u\varphi_{x_i} = \int_{B^a} u\varphi_{x_i}+ \int_{B^b} u\varphi_{x_i}=\int_{B^a} u_{x_i}\varphi+ \int_{B^b} u_{x_i}\varphi = \int_B u_{x_i}\varphi.
		\end{equation}
		and for $i=N+1$
		\begin{equation}
			\int_B u\varphi_{x_N} = \int_{B^a} u\varphi_{x_N}+ \int_{B^b} u\varphi_{x_N}=\int_{B^a} u_{x_N}\varphi+\int_{B^0} u\varphi+ \int_{B^b} u_{x_N}\varphi-\int_{B^0} u\varphi = \int_B u_{x_N}\varphi.
		\end{equation}
		so $u$ is weakly differentiable. Finally, as $\Omega_0=\{(x,y)\in \Omega: y=y_0\}$ is a set of measure zero, we have
		\begin{equation*}
			\norm{u}^2_{H^1(\Omega)}=\int_\Omega\left(\abs{\nabla u}^2+ u^2\right)=\int_{\Omega^a}\left(\abs{\nabla u}^2+ u^2\right)+\int_{\Omega^b}\left(\abs{\nabla u}^2+ u^2\right)=\norm{u}^2_{H^1(\Omega^a)}+\norm{u}^2_{H^1(\Omega^b)}
		\end{equation*}
	\end{proof}
	An straightforward generalization is the following.
	\begin{corollary}
		\label{cor:h1cont2}
		Let $\Omega\subset \R^{N+1}$.  Given the real numbers $a_0<a_1< a_2<\ldots<a_M$, define $\Omega^i=\{(x,y)\in \Omega: a_{i-1}<y<a_i\}$. Assume $\Omega\subset \{(x,y)\in \R^{N+1}: a_0<y<a_M\}$ let $u:\Omega \to \R$ such that $u\in H^1(\Omega^i)$ for every $i=1,\ldots,M$. Then, if for every $(x,a_i)$ where $x\in \RN$ and $i=1,\ldots, M-1$, there exists a ball $B=B((x,a_i),r)\subset \Omega$ such that, using the notation of Lemma \ref{lemma:h1cont2},
		\begin{equation}
			\label{eqn:h1cont2eq1}
			tr(u^a)\equiv tr(u^b) \qquad \text{on} \ B^0,
		\end{equation}
		then $u\in H^1(\Omega)$ and $\norm{u}_{H^1(\Omega)}=\sum_{i=1}^M\norm{u}_{H^1(\Omega^i)}$.
	\end{corollary}
	\begin{proof}
		By applying Lemma \ref{lemma:h1cont2} to $u$ with $\Omega^b=\Omega^1$, $\Omega^a=\Omega^2$ and $y_0=a_1$, we obtain $u\in H^1(\{(x,y)\in \Omega: a_0<y<a_2\})$. Then, by applying Lemma \ref{lemma:h1cont2} to $u$ with $\Omega^b=\{(x,y)\in \Omega: a_0<y<a_2\}$, $\Omega^a=\Omega^3$ and $y_0=a_2$, we obtain $u\in H^1(\{(x,y)\in \Omega: a_0<y<a_3\})$. Therefore, in a finite number of steps we obtain $u\in H^1(\{(x,y)\in \Omega: a_0<y<a_M\})=H^1(\Omega)$.
	\end{proof}
	\begin{remark}
		\label{remark:h1cont2}
		Note that, if $u:\Omega \to \R$ is continuous, equation \eqref{eqn:h1cont2eq1} is automatically satisfied.
	\end{remark}

\end{document}